\documentclass[11pt]{elsarticle}

\usepackage{lineno,hyperref}
\modulolinenumbers[5]
\journal{Studies in Applied Mathematics}

\usepackage{amsmath,accents,amssymb,amsthm,amsfonts,afterpage,float,bm,stmaryrd,paralist,wrapfig,bbm,soul}
\usepackage{mathbbol}

\usepackage{cancel}
\usepackage{color}
\usepackage{fancyhdr}
\usepackage{graphics}
\usepackage{hyperref}
\usepackage{graphicx,epsf}
\usepackage{makeidx}
\usepackage{epsfig}
\usepackage{lscape}
\usepackage{subfigure}
\usepackage{listings}
\usepackage{multirow}

\usepackage{mathtools}
\usepackage{amsfonts}
\usepackage{xcolor}
\usepackage{hyperref}
\usepackage{graphicx}
\graphicspath{{Figure/}{Figures/}}
\usepackage{amsopn}
\DeclareMathOperator{\diag}{diag}

\ifpdf
  \DeclareGraphicsExtensions{.eps,.pdf,.png,.jpg}
\else
  \DeclareGraphicsExtensions{.eps}
\fi

\usepackage{empheq,xcolor}

\pssilent

\newtheorem{theorem}{Theorem}[section]
\newtheorem{lemma}{Lemma}[section]
\newtheorem{assumption}{Assumption}[section]

\newtheorem{remark}{Remark}[section]

\numberwithin{equation}{section}
\numberwithin{figure}{section}
\numberwithin{table}{section}

\newcommand{\bbT}{\mathbb T}

\newcommand{\calM}{\mathcal M}
\newcommand{\calR}{\mathcal R}
\newcommand{\bz}{\boldsymbol z}
\newcommand{\bv}{\boldsymbol v}
\newcommand{\bmu}{\boldsymbol\mu}
\newcommand{\ve}{\varepsilon}
\newcommand{\ud}{\mathrm d}

\newcommand{\kb}{\kappa_{{\rm bend}}}
\renewcommand{\div}{\operatorname{div}}
\newcommand{\norm}[1]{\left\lVert #1\right\rVert}
\newcommand{\abs}[1]{\left\lvert #1\right\rvert}

\newcommand{\calL}{\mathcal L}

\newcommand{\calH}{\mathcal H}
\newcommand{\one}{\boldsymbol 1}
\newcommand{\bPhi}{\boldsymbol\Phi}
\newcommand{\bU}{\mathbf U}

\newcommand{\bPsi}{\boldsymbol\Psi}

\newcommand{\bB}{\boldsymbol B}

\newcommand{\odotv}{\mathbin{\odot}}

\newcommand{\iph}[2]{\left\langle #1,#2\right\rangle_h}

\newcommand{\ThreeDPanel}[1]{%
  \makebox[0.31\linewidth][c]{%
    \includegraphics[
      height=0.255\linewidth,
      keepaspectratio
    ]{#1}%
  }%
}

\newlength{\ThreeDImageHeight}
\begin{document}

\setlength{\pdfpageheight}{\paperheight}
\setlength{\pdfpagewidth}{\paperwidth}
\title{A Thermodynamically Consistent Model for Multicomponent Vesicles}

\author{Wangbo Luo\fnref{cor1}}
\address{Department of Applied Mathematics, The Hong Kong Polytechnic University, Hung Hom, Kowloon, Hong Kong}

\author{Zhonghua Qiao\fnref{cor2}}
\address{Department of Applied Mathematics, The Hong Kong Polytechnic University, Hung Hom, Kowloon, Hong Kong}

\author{Yanxiang Zhao\fnref{cor3}}
\address{Department of Mathematics, The George Washington University, Washington DC, 20052}

\fntext[cor1]{Email: wangbo.luo@polyu.edu.hk}
\fntext[cor2]{Corresponding author. Email: zhonghua.qiao@polyu.edu.hk}
\fntext[cor3]{Corresponding author. Email: yxzhao@gwu.edu}

\begin{abstract}
We develop a thermodynamically consistent diffuse-interface model for multicomponent membranes. The proposed model is derived from a coupled free energy functional that incorporates protein-dependent bending elasticity, diffuse surface tension, a volume penalty, and a membrane-associated Ohta--Kawasaki energy. Applying the Onsager variational principle, we derive a coupled $L^2$ gradient flow system and its energy dissipation law. We then construct a stabilized alternating ETD1 scheme and a stabilized alternating ETDRK2 scheme with Strang-type composition (alternating Strang-ETDRK2). To the best of our knowledge, the proposed alternating Strang-ETDRK2 scheme has not previously been developed and analyzed for coupled phase field systems. We further prove the discrete energy dissipation for both schemes under some regularity assumptions on the numerical solutions. Numerical experiments in two and three dimensions validate the discrete energy dissipation law, and present the protein segregation and membrane deformation produced by the proposed model.
\end{abstract}


\date{\today}
\maketitle

\section{Introduction}
\label{sec:introduction}

In biological systems, lipid-bilayer membranes are deformable surfaces whose composition and geometry are strongly coupled. Experiments on giant unilamellar vesicles (GUVs) \cite{veatch2003,baumgart2003,baumgart2005,schmid2017} have shown that coexisting liquid phases form membrane domains with different thicknesses and mechanical properties, while membrane curvature and line tension can generate buds, necks, tubes, and periodic shapes. When membrane-associated proteins are present, their composition can also modify the elastic bending rigidity, and curvature gradients can redistribute or localize softer and stiffer membrane areas \cite{tian2009,yuan2021,usery2018,rinaldin2020}. These observations show that the protein composition changes the bending rigidity of the membrane, while membrane geometry affects the energetic preference for the proteins. The main modeling objective of this paper is to formulate the coupling between membrane and proteins from a coupled free energy functional.

In recent years, a classic model for lipid-bilayer membranes is developed based on the Helfrich energy \cite{helfrich1973},
\begin{align*}
    E = \int_{\Gamma}\left[a_1+a_2(H-c_0)^2+a_3G\right]\mathrm{d} s,
\end{align*}
where $a_1$, $a_2$, and $a_3$ denote the surface tension, bending rigidity, and Gaussian-curvature rigidity, respectively. $H$ is the mean curvature of the membrane surface, $c_0$ is the spontaneous mean curvature, and $G$ is the Gaussian curvature. Diffuse-interface methods provide an effective approach for modeling membrane deformation. The phase field bending energy introduced in \cite{du2004} has been used to study vesicle equilibria and shape dynamics. This formulation was later extended to model multicomponent membranes \cite{wang2008,lowengrub2009,bellettini2010}. phase field formulations for two-component biological membranes have also been studied on fixed computational domains \cite{ratz2006,elliott2010}. These approaches allow shape deformations and topological changes without explicitly tracking the moving surface.

The Ohta-Kawasaki (OK) model, originally introduced to describe self-assembled structures in diblock copolymers \cite{ohta1986,bates1990block}, provides a useful framework for simulating microphase separation \cite{LuoZhao_PhysicaD2024,LuoZhao_AAMM2024,luo2026double,luo2026nesterov, Wang_Ren_Zhao2019, Choi_Zhao2021}. Its free energy functional is given by \cite{Xu_Zhao2019, LuoZhao_NumPDE2024}
\begin{align}\label{eqn:OK}
        E^{\mathrm{OK}}[u] = \int_{\Omega}\left[\frac{\ve_u}{2}|\nabla u|^2+\frac{1}{\ve_u}W(u)\right] \mathrm{d}x + \frac{\gamma}{2}\int_{\Omega}|(-\Delta)^{-\frac{1}{2}}(u-\omega)|^2 \ \mathrm{d}x,
\end{align}
with the volume constraint
\begin{align*}
    \int_{\Omega}u\ \text{d}x = \omega|\Omega|.
\end{align*}
Here, $0<\ve_u\ll1$ represents the width of the diffuse interface, $\gamma>0$ is the strength of the long-range interaction, and $W(u) = 18(u^2-u)^2$ is a double-well potential with minima at $u=0$ and $u=1$. The variable $u$ represents the density of one component, while the other component is represented implicitly by $1-u$. The first integral in \eqref{eqn:OK} represents the local surface energy
and favors larger domains. In contrast, the second integral describes the long-range repulsive interaction and favors smaller domains. The competition between these two terms produces microphase separation. Therefore, the OK energy provides a mechanism for describing protein segregation on the membrane \cite{luo2026ok}.

In our previous work \cite{luo2026ok}, we localized the OK energy to a diffuse membrane and coupled it to a force-balance phase field equation. The previous model describes mechanochemical dynamics in which proteins segregate on the membrane and generate protein-dependent forces that deform the membrane. A subsequent study developed exponential time differencing (ETD) methods for the membrane-associated OK dynamics and the coupled phase field system \cite{luoETDmembrane}. In contrast to the mechanochemical formulation, the current work develops a thermodynamically consistent diffuse-interface model for multicomponent membranes, in which the membrane and protein equations are derived as a coupled $L^2$ gradient flow system associated with a new coupled free energy functional. Consequently, the influence of protein composition on membrane deformation and the effect of membrane geometry on protein segregation arise from the variational derivatives of the coupled energy with respect to two phase field variables. Here, we focus on model development, the construction of structure-preserving stabilized ETD schemes, and their numerical implementation. Compared with the previous mechanochemical model \cite{luo2026ok}, the resulting variational structure is also more suitable for rigorous mathematical analysis. In particular, the coupled free energy functional and corresponding gradient flow structure provide a natural framework for exploring the existence of energy minimizers and the well-posedness of the coupled $L^2$ gradient flow dynamics, in a manner similar to that of \cite{joo2021analysis}. We leave these analytical questions for future work.

The model proposed in this work is derived using the Onsager variational principle \cite{onsager1931a,onsager1931b}. In this framework, the total free energy determines the thermodynamic driving forces, while the dissipation potential specifies the kinetic response \cite{doi2011}. Energetic and Onsager-type variational principles have been used to derive vesicle-fluid interactions and thermodynamically consistent phase field systems, and they also guided the construction of structure-preserving numerical methods \cite{du2009energetic,li2019,yue2004diffuse,liu2003phase}. In a related study, the authors of \cite{magi2017} developed a model based on minimum energy dissipation for multicomponent membranes by combining mixing energies with Helfrich elasticity in a two-phase viscous fluid.

For numerical simulations, ETD methods are well suited for stiff equations since they integrate the linear part exactly in time while evaluating the nonlinear remainder explicitly \cite{hochbruck2010}. More importantly, the ETD methods have been shown to preserve physical properties for phase field equations, such as discrete maximum-bound principles and the discrete energy dissipation law \cite{du2019mbp,du2021mbp,fu2022}. ETD methods have also been applied to coupled phase field systems, in which the discrete energy dissipation and bound preservation have been established \cite{duan2024convergence,duan2026third}. These studies show that ETD methods inherit some important physical properties in coupled phase field systems. In phase field models for multicomponent vesicles \cite{du2004,wang2008,lowengrub2009}, nonlinear coupling between membrane and composition variables makes both scheme construction and stability analysis challenging. ETD Runge-Kutta (ETDRK) methods have been successfully applied to phase field elastic-bending equations \cite{wang2016efficient}, providing an effective treatment of the stiff fourth-order operator in long-time simulation. In Section~\ref{sec:etd}, we develop stabilized ETD schemes for the coupled phase field system of multicomponent membranes and prove their discrete energy dissipation under some regularity assumptions.

The main contributions of this work are three-fold. Firstly, we develop a thermodynamically consistent diffuse-interface model that couples membrane deformation and membrane-associated protein segregation through a coupled free energy functional. Using the Onsager variational principle, we derive the coupled $L^2$ gradient flow system and its continuous energy dissipation law. Secondly, we construct a stabilized alternating ETD1 method and a stabilized alternating ETDRK2 method with Strang-type composition (Strang-ETDRK2) \cite{strang1968construction} for the coupled system. To the best of our knowledge, this Strang-ETDRK2 scheme has not previously been studied for the coupled phase field system for multicomponent membranes. Thirdly, we prove the discrete energy dissipation for both schemes under some regularity assumptions on the numerical solutions. Numerical experiments in two and three dimensions validate the discrete energy dissipation and illustrate protein segregation and membrane deformation produced by the proposed model.

The remainder of the paper is organized as follows. Section~\ref{sec:setting} introduces basic model settings and the coupled free energy functional. Section~\ref{sec:onsager} derives the coupled $L^2$ gradient flow system using the Onsager variational principle and establishes the continuous energy dissipation law. Section~\ref{sec:etd} presents the stabilized linear-nonlinear splitting technique, the alternating ETD schemes, and their discrete energy dissipation. Section~\ref{sec:numerics} reports numerical experiments for the proposed model and demonstrates the discrete energy dissipation of the schemes. Finally, Section~\ref{sec:conclusion} gives concluding remarks and future directions.

\section{Model Setting and the Coupled Free Energy Functional}
\label{sec:setting}

In this section, we present the coupled free energy functional for multicomponent membranes. Let $\Omega=\bbT^d$, $d\in\{2,3\}$ be a periodic computational domain. We introduce two phase field variables $\phi$ and $u$ to describe the vesicle membrane and the membrane-associated protein, respectively. The phase field function $\phi$ with an interfacial width $\ve_{\phi}$ represents the vesicle membrane, which takes values near $1$ in the interior and $0$ in the exterior. The phase field variable $u$ with an interfacial width $\ve_u$ represents the density of membrane-associated proteins, which takes values near $1$ in protein-rich regions and $0$ in protein-poor regions. 

The model is constructed from the coupled free energy functional
\begin{align}
E_{\rm total}[\phi,u]
={}&
E_{\rm bend}[\phi,u]
+
E_{\rm geom}[\phi]
+
E_{\rm OK}[\phi,u].
\label{eq:total_energy}
\end{align}
Here, $E_{\rm bend}$ describes the protein-dependent elastic bending energy, $E_{\rm geom}$ contains the surface tension energy and the enclosed-volume penalty, and $E_{\rm OK}$ is the membrane-associated OK energy. The $E_{\rm OK}$ consists of a short-range interaction $E_{\rm loc}$, a long-range interaction $E_{\rm nl}$, and a soft constraint on protein density $E_{\rm mass}$. For simplicity, we use the double-well potential $W(s)=18s^2(1-s)^2$ for both the membrane variable $\phi$ and the protein variable $u$. For the discussion of energy dissipation in Sections~\ref{sec:onsager} and \ref{sec:etd}, we adopt the $W(s)$ by a quadratic extension in \cite{Xu_Zhao2019}, so that $W''$ has an upper bound and is Lipschitz continuous. In this paper, we denote by $L_{W''}$ the upper bound of $W''$ and by $L_{W'''}$ the Lipschitz constant of $W''$.

The protein-dependent elastic bending energy is defined by \cite{du2004,wang2008}
\begin{align}
E_{\rm bend}[\phi,u]
=
\int_\Omega
\frac{\kappa_{{\rm bend}}(u)}{2\ve_\phi}
\left(\ve_\phi\Delta\phi -\frac{1}{\ve_\phi}W'(\phi)\right)^2
\,\ud x,
\label{eq:bending-energy}
\end{align}
where $\kappa_{{\rm bend}}(u)$ is a protein-dependent bending rigidity. Following the phase-dependent rigidity proposed in \cite{wang2008}, we take
\begin{align}
\kappa_{{\rm bend}}(u)
=
\kappa_-
+
(\kappa_+-\kappa_-)p(u),
\quad
\kappa_-,\kappa_+>0,
\nonumber
\end{align}
where $p$ is the piecewise $C^2$ smooth function
\begin{align}
p(u)=
\begin{cases}
0, & u\leq0,\\[1mm]
6u^5-15u^4+10u^3, & 0<u<1,\\[1mm]
1, & u\geq1.
\end{cases}
\nonumber
\end{align}
In particular, $0\leq p(u)\leq1$ and $p'(0)=p'(1)=p''(0)=p''(1)=0$. This choice gives $\kappa_{{\rm bend}}(1) = \kappa_{+}$ and $\kappa_{{\rm bend}}(0) = \kappa_{-}$ in protein-rich and protein-poor regions, respectively. We further denote by $L_{\kappa''}$ the upper bound of $\kappa''_{{\rm bend}}(u)$. 

We then collect the surface tension energy and volume constraint as
\begin{align}
E_{\rm geom}[\phi]
=
\lambda_{\rm surf}
\int_\Omega
\left[
\frac{\ve_\phi}{2}|\nabla\phi|^2
+
\frac{1}{\ve_\phi}W(\phi)
\right]
\,\ud x
+
\frac{K_{{\rm vol}}}{2}
\left(
V[\phi]-V_0
\right)^2,
\label{eq:geometric-energy}
\end{align}
where $V[\phi]$ denotes the enclosed volume and is given by
\begin{align*}
    V[\phi]& = \int_\Omega\phi\,\ud x,
\end{align*}
$\lambda_{\rm surf}\geq0$ is the surface tension coefficient, $V_0$ represents the prescribed vesicle volume, and $K_{{\rm vol}}\geq0$ denotes the volume penalty coefficient.

The membrane-associated OK energy is given by
\begin{align}
E_{\rm OK}[\phi,u]
=
\lambda_u
\left(
E_{\rm loc}[\phi,u]
+
E_{\rm nl}[\phi,u]
+
E_{\rm mass}[\phi,u]
\right),
\label{eq:MOK-energy}
\end{align}
where $\lambda_u>0$ is the trade-off between 
the membrane energy and the membrane-associated OK energy. To confine it near the diffuse membrane, we introduce a localization function
\begin{align}
    g(\phi) =  \frac{18}{\ve_{\phi}}(\phi^2-\phi)^2, \nonumber
\end{align}
to describe the interfacial region of $\phi$. In order to study the discrete energy dissipation law (see Section~\ref{sec:etd}),
we introduce an extension $\widetilde{g}(s)$ as 
\begin{align}
\widetilde{g}(s)
=
\frac{18}{\ve_\phi}
\begin{cases}
M^2(1+M)^2
+
M(1+M)(1+2M)\bar{M},
& s\leq-M-\bar{M},
\\[2mm]
\begin{aligned}
&M^2(1+M)^2
+2M(1+M)(1+2M)(-M-s)
\\[0mm]
&
-\frac{M(1+M)(1+2M)}{\bar{M}}
(-M-s)^2,
\end{aligned}
& -M-\bar{M}<s<-M,
\\[2mm]
\displaystyle
s^2(1-s)^2,
& -M\leq s\leq1+M,
\\[3mm]
\begin{aligned}
&M^2(1+M)^2
+2M(1+M)(1+2M)(s-1-M)
\\[0mm]
&
-\frac{M(1+M)(1+2M)}{\bar{M}}
(s-1-M)^2,
\end{aligned}
& 1+M<s<1+M+\bar{M},
\\[2mm]
M^2(1+M)^2
+
M(1+M)(1+2M)\bar{M},
& s\geq1+M+\bar{M},
\end{cases}
\nonumber
\end{align}
for sufficiently large $M$ and $\bar{M}>0$. This extension will guarantee that $g, g'$ has upper bounds and $\widetilde{g}'$ is Lipschitz continuous. For brevity, we will still use $g(\phi)$ to represent $\widetilde{g}(\phi)$. Throughout this paper, we denote by $L_{g}$, $L_{g'}$ the upper bounds of $g$ and $g'$, respectively, and by $L_{g''}$ the Lipschitz constant of $g'$.

The short-range interaction of the membrane-associated OK energy is given by
\begin{align}
E_{\rm loc}[\phi,u]
=
\int_\Omega
g(\phi)
\left[
\frac{\ve_u}{2}|\nabla u|^2
+
\frac{1}{\ve_u}W(u)
\right]
\,\ud x.
\label{eq:local-energy}
\end{align}
To model the long-range interaction, let $L_0^2(\Omega)=\left\{f\in L^2(\Omega):\int_\Omega f\,\ud x=0\right\}$, and denote by $(-\Delta)^{-1}$ the inverse of the periodic negative Laplacian on $L_0^2(\Omega)$. If $f$ is not zero-mean, $(-\Delta)^{-1}$ is naturally extended as $(-\Delta)^{-1}f(x):=(-\Delta)^{-1}(f(x)-\bar{f})$ by removing the mean value $\bar{f}$. The long-range interaction is then given by the inverse Laplacian as
\begin{align}
\begin{aligned}
E_{\rm nl}[\phi,u]
&=\frac\gamma2\int_\Omega \left[(-\Delta)^{-1/2}
\left[g(\phi)(u-\omega)\right]\right]^2\,\ud x
\\
&=\frac\gamma2
\left\langle
g(\phi)(u-\omega),(-\Delta)^{-1}\left[g(\phi)(u-\omega)\right]
\right\rangle
=\frac\gamma2\int_\Omega\abs{\nabla\psi}^2\,\ud x,
\end{aligned}
\label{eq:nonlocal-energy}
\end{align}
where $\gamma>0$ is the long-range interaction strength, $\omega\in(0,1)$ is a prescribed protein fraction, and $\psi = (-\Delta)^{-1}\left[g(\phi)(u-\omega)\right]$. The competition between \eqref{eq:local-energy} and \eqref{eq:nonlocal-energy} follows the OK model \cite{ohta1986}, and the membrane localization is motivated by the membrane-associated OK framework in \cite{luo2026ok}. Soft penalty of the protein area is given by
\begin{align}
E_{\rm mass}[\phi,u]
=
\frac{K_{{\rm mass}}}{2}
\left[
\int_\Omega
g(\phi)(u-\omega)
\,\ud x
\right]^2,
\label{eq:mass-energy}
\end{align}
with $K_{{\rm mass}}\geq0$.  Note that penalties for the enclosed membrane volume in (\ref{eq:geometric-energy}) and protein fraction in (\ref{eq:mass-energy}) favor their prescribed values but do not impose exact conservation.

\begin{remark}\label{rem:line-tension-OK}
In earlier diffuse-interface models for multicomponent vesicles \cite{wang2008}, the line tension energy is represented by a phase field formulation of the form
\begin{align}
L[\phi,u]
=
\lambda_{{\rm line}}\int_\Omega
\left[
\frac{\ve_\phi}{2}|\nabla\phi|^2
+
\frac{1}{\ve_\phi}W(\phi)
\right]
\left[
\frac{\ve_u}{2}|\nabla u|^2
+
\frac{1}{\ve_u}W(u)
\right]
\,\ud x.
\label{eqn:line-ten}
\end{align}
Here, $\lambda_{{\rm line}}$ is the line tension coefficient. In the sharp interface limit, the two terms in the sum $\frac{\ve_\phi}{2}|\nabla\phi|^2+\frac{1}{\ve_\phi}W(\phi)$ become equivalent \cite{modica1987,li2013variational}, thus \(\lambda_u E_{\mathrm{loc}}\) approximates the line-tension energy in \eqref{eqn:line-ten}, up to this normalization factor. Since the membrane-associated OK energy $E_{\rm OK}[\phi,u]$ \eqref{eq:MOK-energy} includes an additional long-range interaction for microphase separation, we introduce the strength $\gamma$ to compete the short-range and long-range interactions.
\end{remark}

\section{$L^2$ Gradient Flow from the Onsager Variational Principle}
\label{sec:onsager}

In this section, we derive the coupled evolution equations from the coupled free energy \eqref{eq:total_energy} through the Onsager variational principle \cite{onsager1931a,doi2011,wang2021}. We introduce the rate space $\mathcal V= L^2(\Omega)\times L^2(\Omega)$ with the inner product
\begin{align}
\left\langle
\begin{pmatrix}\eta\\v\end{pmatrix},
\begin{pmatrix}\zeta\\w\end{pmatrix}
\right\rangle_{\mathcal V}
=
\int_\Omega \eta\zeta\,\ud x
+
\int_\Omega vw\,\ud x, \quad \forall\left(\eta,v\right)^\top,\left(\zeta,w\right)^\top \in \mathcal V,
\nonumber
\end{align}
and the associated norm $\left\|
\left(\eta,v\right)^\top
\right\|_{\mathcal V}^2
=
\|\eta\|_{L^2}^2+\|v\|_{L^2}^2$.

Let $\bz=(\phi,u)^\top\in\mathcal X=\left\{(\phi,u)\in H^4_{\rm per}(\Omega)\times H^2_{\rm per}(\Omega)\right\}$ be the state variable.
We take the first variation of $E_{\rm total}[\phi,u]$ along smooth periodic perturbations $(\widetilde{\phi},\widetilde{u})^\top \in \mathcal V$, which admits the $L^2$ Riesz representation 
\begin{align}
D E_{\rm total}[\phi,u]\left(\widetilde{\phi},\widetilde{u}\right)
=\int_\Omega\frac{\delta E_{\rm total}}{\delta\phi}\widetilde{\phi}\,\ud x
+\int_\Omega\frac{\delta E_{\rm total}}{\delta u}\widetilde{u}\,\ud x.
\nonumber
\end{align}
For a sufficiently smooth trajectory $t\mapsto\bz(t)=(\phi(t),u(t))^\top$ in $\mathcal X$, we have
\begin{align}
\frac{\ud}{\ud t} E_{\rm total}[\phi(t),u(t)]
=
\left\langle\bmu,\bv\right\rangle_{\mathcal V}
=
\int_\Omega\frac{\delta E_{\rm total}}{\delta\phi}\phi_t\,\ud x
+\int_\Omega\frac{\delta E_{\rm total}}{\delta u}u_t\,\ud x,
\label{eq:energy-rate-abstract}
\end{align}
where 
\begin{align}
\bv= \left(\phi_t,u_t\right)^\top,
\quad
\bmu
=\left(\frac{\delta E_{\rm total}}{\delta\phi}, \frac{\delta E_{\rm total}}{\delta u}\right)^\top.
\nonumber
\end{align}

Consider the Onsager dissipation potential given by
\begin{align}
\Phi(\bv)
&=\frac12\left\langle\bv,\mathbb K\bv\right\rangle_{\mathcal V},
\nonumber
\end{align}
where $\mathbb K:\mathcal V \to \mathcal V$ is a symmetric positive-definite operator similar to the Onsager operator in \cite{cang2026synchronization}. Throughout this paper, we take $\mathbb K = \diag\{M_\phi^{-1}I,M_u^{-1}I\}$ with mobility constants $M_{\phi},M_{u}>0$ for simplicity, and the corresponding potential becomes
\begin{align}
\Phi(\bv)
&=\frac12\left\langle\bv,\mathbb K\bv\right\rangle_{\mathcal V}
=\frac12\int_\Omega
\left(\frac{|\phi_t|^2}{M_\phi}
+\frac{|u_t|^2}{M_u}\right)\ud x.
\nonumber
\end{align}
At a fixed state $\bz$, Onsager's variational principle selects the optimal rate $\bv$ that minimizes the Rayleighian $\mathcal{R}_{\bz}$,
\begin{align}
\calR_{\bz}(\bv)
=
\Phi(\bv)
+
D E_{\rm total}[\bz](\bv)
=
\frac12
\left\langle\bv,\mathbb K \bv\right\rangle_{\mathcal V}
+
\left\langle\bmu,\bv\right\rangle_{\mathcal V}.
\label{eq:rayleighian}
\end{align}
Since \(M_\phi,M_u>0\), the Rayleighian is strictly convex and coercive. Its unique minimizer satisfies
\begin{align}
\phi_t=-M_\phi\frac{\delta E_{\rm total}}{\delta\phi},
\quad
u_t=-M_u\frac{\delta E_{\rm total}}{\delta u}.
\label{eq:onsager-flow}
\end{align}
Consequently, every sufficiently smooth solution satisfies
\begin{align}
\frac{\ud}{\ud t} E_{\rm total}[\phi,u]
=-M_\phi\norm{\frac{\delta E_{\rm total}}{\delta\phi}}_{L^2}^2
-M_u\norm{\frac{\delta E_{\rm total}}{\delta u}}_{L^2}^2
=-2\Phi(\bv)\leq0.
\label{eq:onsager-energy-law}
\end{align}
This choice gives nonconservative $L^2$ gradient flow dynamics. 

We next collect the first variations of all terms in $E_{\rm total}[\phi,u]$ \eqref{eq:total_energy} in the following lemma. For simplicity, we omit the detailed derivation.

\begin{lemma}\label{lemma:first-varia}
Let $(\widetilde{\phi},\widetilde{u})$ be a smooth perturbation of $(\phi,u)\in \mathcal X$. Then, each term of the total energy satisfies
\begin{align*}
& DE_{\rm bend}[\phi,u]\left(\widetilde{\phi},\widetilde{u}\right)
={}\int_\Omega\frac1{\ve_{\phi}}
\left[
{\ve_{\phi}}\Delta\bigl(\kb(u)\cdot\bigr)
-\frac1{\ve_{\phi}}\kb(u)W''(\phi)
\right]\left(\ve_\phi\Delta\phi -\frac{1}{\ve_\phi}W'(\phi)\right)\widetilde{\phi}\,\ud x\nonumber\\
& \hspace{3.5cm} +\int_\Omega\left[\frac{\kb'(u)}{2{\ve_{\phi}}}\left(\ve_\phi\Delta\phi -\frac{1}{\ve_\phi}W'(\phi)\right)^2\right]\widetilde{u}\,\ud x, \\
& DE_{\rm geom}[\phi](\widetilde{\phi})
={}\int_\Omega
\left[\lambda_{\rm surf}\left(-{\ve_{\phi}}\Delta\phi+\frac1{\ve_{\phi}} W'(\phi)\right)
+K_{{\rm vol}}\bigl(V[\phi]-V_0\bigr)\right]\widetilde{\phi}\,\ud x,\\
& DE_{\rm loc}[\phi,u]\left(\widetilde{\phi},\widetilde{u}\right)
={}\int_\Omega g'(\phi)\left( \frac{\ve_u}{2}|\nabla u|^2+\frac1{\ve_u}W(u) \right)\widetilde{\phi}\,\ud x
\nonumber\\
&\hspace{3.5cm}+\int_\Omega
\left[-\ve_u\div(g(\phi)\nabla u)
+\frac{g(\phi)}{\ve_u}W'(u)\right]\widetilde{u}\,\ud x,\\
& DE_{\rm nl}[\phi,u]\left(\widetilde{\phi},\widetilde{u}\right)
={}\gamma\int_\Omega g'(\phi)
\bigl(u-\omega\bigr)(-\Delta)^{-1}\left[g(\phi)(u-\omega)\right]\widetilde{\phi}\,\ud x + \gamma\int_\Omega g(\phi)(-\Delta)^{-1}\left[g(\phi)(u-\omega)\right] \widetilde{u}\,\ud x, \\
& DE_{\rm mass}[\phi,u]\left(\widetilde{\phi},\widetilde{u}\right)
={}K_{{\rm mass}}\left(\int_\Omega g(\phi)(u-\omega)\,\ud x\right)\int_\Omega
\left[g'(\phi)(u-\omega)\widetilde{\phi}+g(\phi)\widetilde{u}\right]\ud x.
\end{align*}
\end{lemma}
Combining the expressions in Lemma~\ref{lemma:first-varia} gives $\frac{\delta E_{\rm total}}{\delta\phi}$ and $\frac{\delta E_{\rm total}}{\delta u}$. Finally, the $L^2$ gradient flow system is obtained from the unique Onsager minimizer as
\begin{align}
\phi_t={}& -M_\phi\frac{\delta E_{\rm total}}{\delta\phi} \nonumber\\= {}&
-M_\phi\left[\begin{aligned}
    &\frac1{\ve_{\phi}}
\left[
{\ve_{\phi}}\Delta\bigl(\kb(u)\cdot\bigr)
-\frac1{\ve_{\phi}}\kb(u)W''(\phi)
\right]\left(\ve_\phi\Delta\phi -\frac{1}{\ve_\phi}W'(\phi)\right)
\\
&+\lambda_{\mathrm{surf}}\left(-{\ve_{\phi}}\Delta\phi+\frac1{\ve_{\phi}} W'(\phi)\right)
+K_{{\rm vol}}\bigl(V[\phi]-V_0\bigr)
\\
& +\lambda_u g'(\phi) \left[\begin{aligned}
    &\left(\frac{\ve_u}{2}|\nabla u|^2+\frac1{\ve_u}W(u)\right)
+\gamma \bigl(u-\omega\bigr)(-\Delta)^{-1}\left[g(\phi)(u-\omega)\right]
\\
&+K_{{\rm mass}}(u-\omega)\left(\int_\Omega g(\phi)(u-\omega)\,\ud x\right)
\end{aligned}\right]
\end{aligned}\right],
\label{eq:phi-flow}\\
u_t={}&-M_u\frac{\delta E_{\rm total}}{\delta u} \nonumber\\
= {}&-M_u\left[
\begin{aligned}
    &\frac{\kb'(u)}{2\ve_{\phi}}\left(\ve_\phi\Delta\phi -\frac{1}{\ve_\phi}W'(\phi)\right)^2
\\
&+\lambda_u \left[
\begin{aligned}
   & -\ve_u\div\bigl(g(\phi)\nabla u\bigr)
+\frac{g(\phi)}{\ve_u}W'(u)+\gamma g(\phi)(-\Delta)^{-1}\left[g(\phi)(u-\omega)\right]
\\
&+K_{{\rm mass}}\left(\int_\Omega g(\phi)(u-\omega)\,\ud x\right)g(\phi)
\end{aligned}\right]
\end{aligned}
\right].
\label{eq:u-flow}
\end{align}
This gradient flow system satisfies the continuous energy dissipation law
\begin{align}
\frac{\ud}{\ud t} E_{\rm total}[\phi(t),u(t)]
=-\frac{1}{M_\phi}\int_\Omega\abs{\phi_t}^2\,\ud x
-\frac{1}{M_u}\int_\Omega\abs{u_t}^2\,\ud x \leq 0.
\label{eq:energy-law-time}
\end{align}
In what follows, our main focus is to develop proper numerical schemes for the coupled system \eqref{eq:phi-flow}-\eqref{eq:u-flow} to inherit the dissipation law at the discrete level. Hereafter, we will always take $M_\phi = M_u = 1$ for the sake of simplicity. 

\section{Exponential Time Differencing for the Coupled System}\label{sec:etd}

In this section, we construct the stabilized alternating ETD1 and Strang-ETDRK2 schemes and prove the discrete energy dissipation law.

\subsection{Stabilized Linear-nonlinear Splitting for the Coupled System}\label{subsec:linear-split}

We first apply a stabilized linear-nonlinear splitting technique \cite{luoETDmembrane,wang2016efficient} to rewrite the coupled system \eqref{eq:phi-flow}-\eqref{eq:u-flow}. To retain the resulting linear operator being constant-coefficient, self-adjoint, and positive-definite, we introduce a constant reference rigidity for the protein-dependent bending rigidity $\kappa_{{\rm bend}}(u)$
\begin{align*}
    \kappa_{\rm ref}>0,
    \quad
    \widetilde\kappa_{{\rm bend}}(u)
    =\kappa_{{\rm bend}}(u)-\kappa_{\rm ref}.
\end{align*}

For the $\phi$-equation \eqref{eq:phi-flow}, introducing two stabilization constants $\alpha_1,\alpha_2>0$, we can rewrite the equation as
\begin{align}
    \phi_t +\mathcal L_\phi\phi
    =
    \mathcal R_\phi(\phi,u),
    \label{eq:phi-linear-splitting}
\end{align}
where the linear part is
\begin{align}
    \mathcal L_\phi
    ={}&
    \kappa_{\rm ref}{\ve_{\phi}}
    \left(
        \Delta-\frac{\alpha_2}{\ve_{\phi}^2}I
    \right)
    \left(
        \Delta-\frac{\alpha_1}{\ve_{\phi}^2}I
    \right)
    +
    \lambda_{\rm surf}
    \left(
        -\ve_{\phi}\Delta+\frac{\alpha_1}{{\ve_{\phi}}}I
    \right).
\label{eq:phi-linear-operator}
\end{align}
The nonlinear remainder is then
\begin{align}
\begin{aligned}
    \mathcal R_\phi(\phi,u)
    ={}&
    \frac{\kappa_{\rm ref}}{{\ve_{\phi}}}
    \left[
        \Delta \left(W'(\phi)-\alpha_1\phi\right)
        +
        \left(W''(\phi)-\alpha_2\right)
        \left(
            \Delta\phi-\frac{\alpha_1}{\ve_{\phi}^2}\phi
        \right)
    \right]
\\
    &-
    \left[
        \frac{\kappa_{\rm ref}}
             {\ve_{\phi}^3}W''(\phi)
        +
        \frac{\lambda_{\rm surf}}{{\ve_{\phi}}}
    \right]\left(W'(\phi)-\alpha_1\phi\right)
\\
    &-
    \frac{1}{{\ve_{\phi}}}
    \left[
        {\ve_{\phi}}\Delta
        \left(
            \widetilde\kappa_{{\rm bend}}(u)\cdot
        \right)
        -
        \frac1{\ve_{\phi}}
        \widetilde\kappa_{{\rm bend}}(u)
        W''(\phi)
    \right]\left({\ve_{\phi}}\Delta\phi-\frac1{\ve_{\phi}}W'(\phi)\right)
\\
    &-
     K_{{\rm vol}}
    \bigl(V[\phi]-V_0\bigr)
    -
    \lambda_u g'(\phi)\left(\frac{\ve_u}{2}|\nabla u|^2+\frac1{\ve_u}W(u)\right)
\\
    &-
    \lambda_u \gamma g'(\phi)
    \bigl(u-\omega\bigr)\psi
    -
    \lambda_u K_{{\rm mass}}\left(\int_\Omega g(\phi)(u-\omega)\,\ud x\right)
     g'(\phi)(u-\omega).
\end{aligned}
\nonumber
\end{align}
For the $u$-equation \eqref{eq:u-flow}, we introduce two stabilization constants $\beta_1,\beta_2>0$, so it can be written exactly as
\begin{align}
u_t+\mathcal L_u u
=
\mathcal R_u(\phi,u),
\label{eq:u-linear-splitting}
\end{align}
where
\begin{align}
\mathcal L_u
=
\lambda_u
\left(
    \beta_2I-\ve_u\beta_1\Delta
\right),
\label{eq:u-linear-operator}
\end{align}
and the nonlinear remainder is
\begin{align}
\begin{aligned}
\mathcal R_u(\phi,u)
={}&
\lambda_u
\left[
\begin{aligned}
    & \ve_u
\nabla\cdot
\left(
    \bigl(g(\phi)-\beta_1\bigr)\nabla u
\right)
+
\beta_2u
-
\frac{g(\phi)}{\ve_u}W'(u) \\
& -
\gamma g(\phi)\psi
-
K_{{\rm mass}}g(\phi)\left(\int_\Omega g(\phi)(u-\omega)\,\ud x\right)
\end{aligned}
\right]
\\
&-
\frac{\kb'(u)}
{2\ve_{\phi}}
\left({\ve_{\phi}}\Delta\phi-\frac1{\ve_{\phi}}W'(\phi)\right)^2.
\end{aligned}
\nonumber
\end{align}

\subsection{Spatial Discretization by Second-order Finite Difference}
\label{sec:discrete-setting}
In this subsection, we discretize the spatial differential operators by the second-order finite difference method. For clarity, we present the notation in two dimensions. The three-dimensional case is defined similarly. Let $\Omega=[-X,X)\times[-Y,Y)\subset\mathbb R^2$ be a rectangular domain with periodic boundary conditions. Let $N_x$ and $N_y$ be positive even integers, define $h_x=\frac{2X}{N_x}$ and $h_y=\frac{2Y}{N_y}$, and discretize the spatial domain $\Omega$ by a uniform mesh as $\Omega_h = \Omega\cap (h_x\mathbb{Z}\times h_y\mathbb{Z})$. We then define the index set 
\begin{align}
    S_{h} = \{ (i,j)\in\mathbb{Z}^2;\ i = 1:N_x,\ j = 1:N_y\}.
\end{align}
Let $\mathcal M_h$ denote the space of real-valued periodic grid functions, defined by
\begin{align*}
\mathcal M_h
=
\left\{
V=\{V_{ij}\}_{(i,j)\in\mathbb Z^2}:
V_{i+m_1N_x,j+m_2N_y}=V_{ij},
\
(m_1,m_2)\in\mathbb Z^2
\right\}.
\end{align*}
The corresponding zero-mean subspace is $\mathcal M_{h,0}
=
\left\{
V\in\mathcal M_h:
\left\langle V,\mathbf 1\right\rangle_h=0
\right\}$, where $\mathbf 1$ denotes the constant grid function whose values are equal to one.

For any $V, Z, G \in\mathcal M_h$, the discrete inner product, discrete $L^2$ norm, and discrete $L^{\infty}$ norm are defined as
\begin{align*}
\left\langle V,Z\right\rangle_h
&=
h_xh_y
\sum_{i=1}^{N_x}
\sum_{j=1}^{N_y}
V_{ij}Z_{ij},\quad
\|V\|_{2,h}
=
\sqrt{\left\langle V,V\right\rangle_h},
\quad
\|V\|_{\infty,h}
=
\max_{(i,j)\in S_h}|V_{ij}|.
\end{align*}
Define the following operators
\begin{align*}
&\left(D_x^+V\right)_{i+\frac12,j}
=
\frac{V_{i+1,j}-V_{ij}}{h_x},\quad
\left(D_y^+V\right)_{i,j+\frac12}
=
\frac{V_{i,j+1}-V_{ij}}{h_y}, \\
& \left(d_x^-V\right)_{ij}
=
\frac{V_{i+\frac12,j}-V_{i-\frac12,j}}{h_x},\quad
\left(d_y^-V\right)_{ij}
=
\frac{V_{i,j+\frac12}-V_{i,j-\frac12}}{h_y}, \\
& G^{x}_{i+\frac12,j}
=
\frac{G_{i+1,j}+G_{ij}}{2},\quad
G^{y}_{i,j+\frac12}
=
\frac{G_{i,j+1}+G_{ij}}{2}.
\end{align*}
The standard second-order discrete Laplacian is then defined by
\begin{align*}
\Delta_hV
&=
d_x^-D_x^+V+d_y^-D_y^+V,
\end{align*}
and $-\Delta_h \succeq0$. Equivalently,
\begin{align*}
(\Delta_hV)_{ij}
={}&
\frac{V_{i+1,j}-2V_{ij}+V_{i-1,j}}{h_x^2}+
\frac{V_{i,j+1}-2V_{ij}+V_{i,j-1}}{h_y^2}.
\end{align*}
For $P,Q,R,S \in\mathcal M_h$, define 
\begin{align*}
\left\langle P,R\right\rangle_{x,h}
&=
h_xh_y
\sum_{i=1}^{N_x}
\sum_{j=1}^{N_y}
P_{i+\frac12,j}R_{i+\frac12,j},\quad
\left\langle Q,S\right\rangle_{y,h}
=
h_xh_y
\sum_{i=1}^{N_x}
\sum_{j=1}^{N_y}
Q_{i,j+\frac12}S_{i,j+\frac12}.
\end{align*}
The summation-by-parts identities give
\begin{align*}
\left\langle d_x^-P,V\right\rangle_h
&=
-\left\langle P,D_x^+V\right\rangle_{x,h},\quad\left\langle d_y^-Q,V\right\rangle_h
=
-\left\langle Q,D_y^+V\right\rangle_{y,h}.
\end{align*}
Consequently,
\begin{align*}
\left\langle -\Delta_hV,Z\right\rangle_h
={}&
\left\langle D_x^+V,D_x^+Z\right\rangle_{x,h}
+
\left\langle D_y^+V,D_y^+Z\right\rangle_{y,h}.
\end{align*}

We further define the discrete variable-coefficient diffusion operator $\Delta_{\mathrm S,h}:\mathcal M_h\to\mathcal M_h$ by
\begin{align}
\Delta_{\mathrm S,h}V
={}&
d_x^-\left(G^{x}D_x^+V\right)
+d_y^-\left(G^{y}D_y^+V\right),
\label{eq:Kh}
\end{align}
which reads componentwise as
\begin{align}
\bigl(\Delta_{\mathrm S,h}V\bigr)_{ij}
={}&
\frac{1}{h_x^2}
\left[
G^{x}_{i+\frac12,j}(V_{i+1,j}-V_{ij})
-
G^{x}_{i-\frac12,j}(V_{ij}-V_{i-1,j})
\right]
\nonumber\\
&+
\frac{1}{h_y^2}
\left[
G^{y}_{i,j+\frac12}(V_{i,j+1}-V_{ij})
-
G^{y}_{i,j-\frac12}(V_{ij}-V_{i,j-1})
\right].
\nonumber
\end{align}
By the summation-by-parts identities, 
\begin{align} 
\langle -\Delta_{\mathrm S,h}V,Z\rangle_h ={}& \langle G^xD_x^+V,D_x^+Z\rangle_{x,h} + \langle G^yD_y^+V,D_y^+Z\rangle_{y,h}. \nonumber
\end{align} 
If $0\leq G_{ij}\leq G_{\max} = \max_{(i,j)\in S_h}\{G_{ij}\}<\infty$, then $-\Delta_{\mathrm S,h}$ is self-adjoint and positive semidefinite 
\begin{align} 
0 \leq \langle-\Delta_{\mathrm S,h}V,V\rangle_h &= \langle G^xD_x^+V,D_x^+V\rangle_{x,h} + \langle G^yD_y^+V,D_y^+V\rangle_{y,h} \nonumber\\ &\leq G_{\max} \left( \|D_x^+V\|_{x,h}^2 + \|D_y^+V\|_{y,h}^2 \right) \nonumber\\ &= G_{\max}\langle-\Delta_hV,V\rangle_h. \nonumber
\end{align}
For any $\bPhi,\bU\in \mathcal M_h$, we set
\begin{align}
  & \bB_h(\bPhi)
  =\ve_{\phi}\Delta_h\bPhi-\frac1{\ve_{\phi}} W'(\bPhi), \quad V_h(\bPhi)=\iph{\bPhi}{\one},
  \nonumber\\
  & m_h(\bPhi,\bU)
  =\iph{g(\bPhi)\odotv(\bU-\omega\one)}{\one}, \quad \bPsi_h= (-\Delta_h)^{-1}\left[g(\bPhi)\odotv\left(\bU-\omega\one\right)\right], \nonumber
\end{align}
where $(-\Delta_h)^{-1}$ denotes the discrete inverse Laplacian operator on $\calM_{h,0}$, and if $f\in\calM_{h}$ is not zero-mean, $(-\Delta_h)^{-1}$ is naturally extended as $(-\Delta)_{h}^{-1}f=(-\Delta)_{h}^{-1}\left[f-\frac{1}{|\Omega|}\langle f,\one \rangle_{h}\one\right]$. Thus $\langle \bPsi_h,\one\rangle_h=0$. We further take $G = g(\bPhi)$ in \eqref{eq:Kh} and define 
\begin{align} 
\bigl(\boldsymbol e_{u,h}(\bU)\bigr)_{ij} ={}& \frac{\ve_u}{4}
\left[\begin{aligned}
    &\left(D_x^+\bU\right)_{i+\frac12,j}^2 + \left(D_x^+\bU\right)_{i-\frac12,j}^2 \nonumber\\  &+ \left(D_y^+\bU\right)_{i,j+\frac12}^2 + \left(D_y^+\bU\right)_{i,j-\frac12}^2 
\end{aligned}\right] + \frac{1}{\ve_u}W(\bU_{ij}), \nonumber
\end{align}
where $\boldsymbol e_{u,h}(\bU)$ is the finite difference approximation of term $\frac{\ve_u}{2}\abs{\nabla u}^2+\frac1{\ve_u}W(u)$ in the short-range energy $E_{\rm loc}[\phi,u]$ \eqref{eq:local-energy}. Therefore, $E_{{\rm total}}$ is approximated by the discrete energy $E_{{\rm total},h}(\bPhi,\bU)$
\begin{align}
\begin{aligned}
    E_{{\rm total},h}(\bPhi,\bU)
={}&\underbrace{\frac1{2\ve_{\phi}}
  \iph{\kb(\bU)\odotv\bB_h(\bPhi)^2}{\one}}_{E_{{\rm bend},h}(\bPhi,\bU)}\\
&+\underbrace{\lambda_{\rm surf}
  \left[\frac{\ve_{\phi}}2\iph{-\Delta_h\bPhi}{\bPhi}
  +\frac1{\ve_{\phi}}\iph{W(\bPhi)}{\one}\right]
  +\frac{K_{{\rm vol}}}{2}\left(V_h(\bPhi)-V_0\right)^2}_{E_{{\rm geom},h}(\bPhi)}\\
&+\lambda_u\left[
\begin{aligned}
    & \underbrace{\frac{\ve_u}{2}\iph{-\Delta_{\rm S,h}\bU}{\bU}
  +\frac1{\ve_u}\iph{g(\bPhi)\odotv W(\bU)}{\one}}_{E_{{\rm loc},h}(\bPhi,\bU)}+\underbrace{\frac{K_{{\rm mass}}}{2}m_h(\bPhi,\bU)^2}_{E_{{\rm mass},h}(\bPhi,\bU)}\\
&+\underbrace{\frac\gamma2\iph{g(\bPhi)\odotv\left(\bU-\omega\one\right)}
  {(-\Delta_h)^{-1}\left[g(\bPhi)\odotv\left(\bU-\omega\one\right)\right]}}_{E_{{\rm nl},h}(\bPhi,\bU)}
\end{aligned}
\right].
\end{aligned}
  \label{eq:discrete-energy}
\end{align}

Moreover, the spatial discretization of the stabilized coupled system \eqref{eq:phi-linear-splitting}-\eqref{eq:u-linear-splitting} becomes 
\begin{align}
\bPhi_t + \calL_{\phi,h}\bPhi
    =
    \calR_{\phi,h}(\bPhi;\bU),\quad
\bU_t+\calL_{u,h} \bU
=
\calR_{u,h}(\bU;\bPhi),
\label{eqn:discrete-linear}
\end{align}
where the discrete linear operators are 
\begin{align}
\calL_{\phi,h}
={}&\kappa_{\rm ref}\ve_{\phi}(-\Delta_h)^2
-\left[
  \frac{\kappa_{\rm ref}(\alpha_1+\alpha_2)}{\ve_{\phi}}
  +\lambda_{\rm surf}\ve_{\phi}\right]\Delta_h
+\left[
  \frac{\kappa_{\rm ref}\alpha_1\alpha_2}{\ve_{\phi}^3}
  +\frac{\lambda_{\rm surf}\alpha_1}{\ve_{\phi}}\right]I,
\label{eq:Sphi-direct}\\
\calL_{u,h}
={}&\lambda_u\left(-\ve_u\beta_1\Delta_{h}+\beta_2I\right).
\label{eq:Su-direct}
\end{align}
Since $-\Delta_{h}\succeq0$ and $\kappa_{\rm ref},\alpha_1,\alpha_2,\beta_1,\beta_2>0$, $\calL_{\phi,h}$ and $\calL_{u,h}$ are both self-adjoint and positive-definite operators. The discrete nonlinear remainders can be written as 
\begin{align}
& \begin{aligned}
    \calR_{\phi,h}(\bPhi;\bU)
    ={}&
    \frac{\kappa_{\rm ref}}{{\ve_{\phi}}}
    \left[
        \Delta_h \left(W'(\bPhi)-\alpha_1\bPhi\right)
        +
        \left(W''(\bPhi)-\alpha_2\right)
        \odotv\left(
            \Delta_h\bPhi-\frac{\alpha_1}{\ve_{\phi}^2}\bPhi
        \right)
    \right]
\\
    &-
    \left[
        \frac{\kappa_{\rm ref}}
             {\ve_{\phi}^3}W''(\bPhi)
        +
        \frac{\lambda_{\rm surf}}{{\ve_{\phi}}}
    \right]\odotv\left(W'(\bPhi)-\alpha_1\bPhi\right)
\\
    &-
    \frac{1}{{\ve_{\phi}}}
    \left[
        {\ve_{\phi}}\Delta_h
        \left(
            \widetilde\kappa_{{\rm bend}}(\bU)\odotv \bB_h(\bPhi)
        \right)
        -
        \frac1{\ve_{\phi}}
        \widetilde\kappa_{{\rm bend}}(\bU)\odotv
        W''(\bPhi)\odotv \bB_h(\bPhi)
    \right]
\\
    &-
     K_{{\rm vol}}
    \bigl(V_h(\bPhi)-V_0\bigr)\one
    -
    \lambda_u g'(\bPhi)\odotv\boldsymbol e_{u,h}(\bU)
\\
& -\lambda_u
  \gamma g'(\bPhi)\odotv \left(
    \bU-\omega\one
  \right)\odotv\bPsi_h - \lambda_u K_{{\rm mass}}m_h(\bPhi,\bU)g'(\bPhi)\odotv
  \left(\bU-\omega\one\right),
\end{aligned}
\nonumber \\
& \begin{aligned}
\calR_{u,h}(\bU;\bPhi)
={}&-\lambda_u
\left[
\begin{aligned}
    & \ve_u\beta_1\Delta_h\bU-\ve_u
  \Delta_{\rm S,h}\bU -\beta_2\bU
+\frac{1}{\ve_u}
  g(\bPhi)\odotv W'(\bU)\\
&+\gamma\,
  g(\bPhi)\odotv\bPsi_h
+K_{{\rm mass}} m_h(\bPhi,\bU)\,
  g(\bPhi)
\end{aligned}
\right]
\\
&-\frac1{2\ve_{\phi}}
  \kb^{\prime}(\bU)
  \odotv\bB_h(\bPhi)^2.
\end{aligned}
\nonumber
\end{align}
The discrete nonlinear remainders $\calR_{\phi,h}(\bPhi;\bU)$ and $\calR_{u,h}(\bU;\bPhi)$ can also be expressed directly as
\begin{align}
\calR_{\phi,h}(\bPhi;\bU)=
\calL_{\phi,h}\bPhi-\frac{\delta E_{{\rm total},h}(\bPhi,\bU)}{\delta\bPhi},\quad
\calR_{u,h}(\bU;\bPhi)=
\calL_{u,h}\bU-
\frac{\delta E_{{\rm total},h}(\bPhi,\bU)}{\delta \bU}.
\nonumber
\end{align}
Here, $\frac{\delta E_{{\rm total},h}(\bPhi,\bU)}{\delta\bPhi}$ and $\frac{\delta E_{{\rm total},h}(\bPhi,\bU)}{\delta \bU}$ are the discrete variational derivatives of the discrete free energy,
\begin{align}
& \begin{aligned}
     \frac{\delta E_{{\rm total},h}(\bPhi,\bU)}{\delta\bPhi}
={}&
\underbrace{\Delta_h
\left(
  \kb(\bU)\odotv\bB_h(\bPhi)
\right)
-\frac1{\ve_{\phi}^2}
  \kb(\bU)\odotv W''(\bPhi)
  \odotv\bB_h(\bPhi)}_{\delta E_{{\rm bend},h}(\bPhi,\bU)/\delta\bPhi}\\
&+\underbrace{\lambda_{\rm surf}
\left(
  -\ve_{\phi}\Delta_h\bPhi
  +\frac1{\ve_{\phi}} W'(\bPhi)
\right)
+K_{{\rm vol}}\left(V_h(\bPhi)-V_0\right)\one}_{\delta E_{{\rm geom},h}(\bPhi)/\delta\bPhi}\\
&+\lambda_u
   \left[
\begin{aligned}
    & \underbrace{g'(\bPhi)\odotv\boldsymbol e_{u,h}(\bU)}_{\delta E_{{\rm loc},h}(\bPhi,\bU)/\delta\bPhi} + \underbrace{\gamma
  g'(\bPhi)\odotv\left(
    \bU-\omega\one
  \right)\odotv\bPsi_h}_{\delta E_{{\rm nl},h}(\bPhi,\bU)/\delta\bPhi}\\
  &+ \underbrace{K_{{\rm mass}} m_h(\bPhi,\bU)
  g'(\bPhi)\odotv\left(\bU-\omega\one\right)}_{\delta E_{{\rm mass},h}(\bPhi,\bU)/\delta\bPhi}
\end{aligned}\right],
\end{aligned}
\label{eqn:discrete-mu-phi-explicit} 
\end{align}
\begin{align}
& \begin{aligned}
   \frac{\delta E_{{\rm total},h}(\bPhi,\bU)}{\delta \bU}
={}&
\underbrace{\frac1{2\ve_{\phi}}
  \kb^{\prime}(\bU)
  \odotv\bB_h(\bPhi)^2}_{\delta E_{{\rm bend},h}(\bPhi,\bU)/\delta\bU}\\
  &+\lambda_u
  \left[\begin{aligned}
      &\underbrace{-\ve_u
  \Delta_{\rm S,h}\bU
+\frac{1}{\ve_u}
  g(\bPhi)\odotv W'(\bU)}_{\delta E_{{\rm loc},h}(\bPhi,\bU)/\delta\bU}\\
&+\underbrace{\gamma\,
  g(\bPhi)\odotv\bPsi_h}_{\delta E_{{\rm nl},h}(\bPhi,\bU)/\delta\bU}
+\underbrace{K_{{\rm mass}} m_h(\bPhi,\bU)\,
  g(\bPhi)}_{\delta E_{{\rm mass},h}(\bPhi,\bU)/\delta\bU}
  \end{aligned}\right].
\end{aligned}
\label{eqn:discrete-mu-u-explicit}
\end{align}

\subsection{Temporal Discretization by Stabilized Alternating ETD Schemes}
\label{sec:etd-splitting}

In this subsection, we introduce the alternating ETD1 and Strang-ETDRK2 schemes for the coupled system \eqref{eqn:discrete-linear}. Given a time interval $[0,T]$ and an integer $K>0$, we take the time step size $\tau = T/K$ and set $t_n = n\tau$ for $n=0,1,\cdots,K$. Let $\bPhi^{n},\bU^{n}$ denote the approximate solutions of $\phi,u$ at time $t_n$. The alternating ETD1 method \cite{luoETDmembrane} for the discrete stabilized coupled system \eqref{eqn:discrete-linear} is given by
\begin{align}\label{eqn:ETDRK1}
    \bPhi^{n+1} = \mathrm{ETD1}_{\phi,\tau}(\bPhi^{n};\bU^{n}), \quad
    \bU^{n+1} = \mathrm{ETD1}_{u,\tau}(\bU^{n};\bPhi^{n+1}), 
\end{align}
in which the two operators $\mathrm{ETD1}_{\phi,\tau}$ and $\mathrm{ETD1}_{u,\tau}$ are defined by
\begin{align}
\begin{aligned}
  \bPhi^{+} = \mathrm{ETD1}_{\phi,\tau}(\bPhi;\bU): \ \bPhi^+
  & =e^{-\tau\calL_{\phi,h}}\bPhi
  +\tau\varphi_1(\tau\calL_{\phi,h})
  \calR_{\phi,h}(\bPhi;\bU),\\
  \bU^+ = \mathrm{ETD1}_{u,\tau}(\bU;\bPhi): \ \bU^+
  & =e^{-\tau\calL_{u,h}}\bU
  +\tau\varphi_1(\tau\calL_{u,h})
  \calR_{u,h}(\bU;\bPhi).
\end{aligned}
\nonumber
\end{align}
Here, $\varphi_1(z)=\frac{1-e^{-z}}{z}>0$ for all $z>0$ with $\varphi_1(0)=1$. 

The standard ETDRK2 method \cite{du2021mbp}
with alternating updates between $\bPhi$ and $\bU$ is given as follows
\begin{align}\label{eqn:ETDRK2_traditional}
    \bPhi^{n+1} = \mathrm{ETDRK2}_{\phi,\tau}(\bPhi^{n};\bU^{n}), \quad
    \bU^{n+1} = \mathrm{ETDRK2}_{u,\tau}(\bU^{n};\bPhi^{n+1}), 
\end{align}
in which the two operators $\mathrm{ETDRK2}_{\phi,\tau}$ and $\mathrm{ETDRK2}_{u,\tau}$ are defined as follows
\begin{align*}
    \bPhi^+ = \mathrm{ETDRK2}_{\phi,\tau}(\bPhi;\bU): \ 
    &\widetilde{\bPhi}
  =e^{-\tau\calL_{\phi,h}}\bPhi
  +\tau\varphi_1\left(\tau\calL_{\phi,h}\right)
  \calR_{\phi,h}(\bPhi;\bU),\\
  &\bPhi^+
  =\widetilde{\bPhi}
  +\tau\varphi_2\left(\tau\calL_{\phi,h}\right)
  \left[
  \calR_{\phi,h}(\widetilde{\bPhi};\bU)
  -\calR_{\phi,h}(\bPhi;\bU)
  \right];\\
  \bU^+ = \mathrm{ETDRK2}_{u,\tau}(\bU;\bPhi): \
  & \widetilde{\bU}
   =e^{-\tau\calL_{u,h}}\bU
  +\tau\varphi_1(\tau\calL_{u,h})
  \calR_{u,h}(\bU;\bPhi),\\
  & \bU^+
   =\widetilde{\bU}
  +\tau\varphi_2(\tau\calL_{u,h})
  \left[
  \calR_{u,h}(\widetilde{\bU};\bPhi)
  -\calR_{u,h}(\bU;\bPhi)
  \right].
\end{align*}
Here, $\varphi_2(z)=\frac{e^{-z}-1+z}{z^2} >0$ for all $z>0$ with $\varphi_2(0)=1/2$. 

However, since the traditional ETDRK2 method \eqref{eqn:ETDRK2_traditional} cannot guarantee the second-order temporal accuracy for the coupled system, we propose an alternating ETDRK2 method through a Strang-type composition (Strang-ETDRK2) \cite{strang1968construction} as
\begin{align}
\begin{aligned}
    \widetilde{\bPhi}^{\,n+\frac{1}{2}}
  & =e^{-\frac{\tau}{2}\calL_{\phi,h}}\bPhi^n
  +\frac{\tau}{2}\varphi_1\left(\frac{\tau}{2}\calL_{\phi,h}\right)
  \calR_{\phi,h}(\bPhi^n;\bU^n),\\
  \bPhi^{n+\frac{1}{2}}
  & =\widetilde{\bPhi}^{\,n+\frac{1}{2}}
  +\frac{\tau}{2}\varphi_2\left(\frac{\tau}{2}\calL_{\phi,h}\right)
  \left[
  \calR_{\phi,h}(\widetilde{\bPhi}^{\,n+\frac{1}{2}};\bU^n)
  -\calR_{\phi,h}(\bPhi^n;\bU^n)
  \right];\\
  \widetilde{\bU}^{\,n+1}
  & =e^{-\tau\calL_{u,h}}\bU^n
  +\tau\varphi_1(\tau\calL_{u,h})
  \calR_{u,h}(\bU^n;\bPhi^{n+\frac{1}{2}}),\\
  \bU^{n+1}
  & =\widetilde{\bU}^{\,n+1}
  +\tau\varphi_2(\tau\calL_{u,h})
  \left[
  \calR_{u,h}(\widetilde{\bU}^{\,n+1};\bPhi^{n+\frac{1}{2}})
  -\calR_{u,h}(\bU^n;\bPhi^{n+\frac{1}{2}})
  \right]; \\
\widetilde{\bPhi}^{\,n+1}
  & =e^{-\frac{\tau}{2}\calL_{\phi,h}}\bPhi^{n+\frac{1}{2}}
  +\frac{\tau}{2}\varphi_1\left(\frac{\tau}{2}\calL_{\phi,h}\right)
  \calR_{\phi,h}(\bPhi^{n+\frac{1}{2}};\bU^{n+1}),\\
  \bPhi^{n+1}
  & =\widetilde{\bPhi}^{\,n+1}
  +\frac{\tau}{2}\varphi_2\left(\frac{\tau}{2}\calL_{\phi,h}\right)
  \left[
  \calR_{\phi,h}(\widetilde{\bPhi}^{\,n+1};\bU^{n+1})
  -\calR_{\phi,h}(\bPhi^{n+\frac{1}{2}};\bU^{n+1})
  \right].
\end{aligned}
\nonumber
\end{align}
In compact form, Strang-ETDRK2 reads
\begin{align}\label{eqn:ETDRK2_Strang}
\begin{aligned}
     & \bPhi^{n+\frac{1}{2}} = \mathrm{ETDRK2}_{\phi,\frac{\tau}{2}}(\bPhi^{n};\bU^{n}), \\
    & \bU^{n+1} = \mathrm{ETDRK2}_{u,\tau}(\bU^{n};\bPhi^{n+\frac{1}{2}}), \\
    & \bPhi^{n+1} = \mathrm{ETDRK2}_{\phi,\frac{\tau}{2}}(\bPhi^{n+\frac{1}{2}};\bU^{n+1}).
\end{aligned}
\end{align}
We briefly study the second-order temporal accuracy for the Strang-ETDRK2 scheme in the following remark.

\begin{remark}\label{rem:Strang-ETDRK2-order}
We study the one-step local truncation error of the Strang-ETDRK2 scheme \eqref{eqn:ETDRK2_Strang}. Consider the semi-discrete system
\begin{align}
    \bPhi_t=\mathcal F(\bPhi,\bU),
    \quad
    \bU_t=\mathcal G(\bPhi,\bU).
    \nonumber
\end{align}
Assume that $\mathcal F$ and $\mathcal G$ are sufficiently smooth and denote
\begin{align*} 
& \mathcal F_n =\mathcal F(\bPhi(t_n),\bU(t_n)), \quad \mathcal G_n =\mathcal G(\bPhi(t_n),\bU(t_n)), \\
& F_{\bPhi,n} = \partial_{\bPhi}\mathcal F(\bPhi(t_n),\bU(t_n)), \quad F_{\bU,n}= \partial_{\bU}\mathcal F(\bPhi(t_n),\bU(t_n)), \\
& G_{\bPhi,n} = \partial_{\bPhi}\mathcal G(\bPhi(t_n),\bU(t_n)), \quad G_{\bU,n}= \partial_{\bU}\mathcal G(\bPhi(t_n),\bU(t_n)).
\end{align*} 
The Taylor expansions of the exact solution at $t_n$ are
\begin{align}
\begin{aligned}
    \bPhi(t_{n+1})
={}&
\bPhi(t_n)+\tau\mathcal F_n
+\frac{\tau^2}{2}
\left(
\mathcal F_{\bPhi,n}\mathcal F_n
+\mathcal F_{\bU,n}\mathcal G_n
\right)
+O(\tau^3),\\
\bU(t_{n+1})
={}&
\bU(t_n)+\tau\mathcal G_n
+\frac{\tau^2}{2}
\left(
\mathcal G_{\bPhi,n}\mathcal F_n
+\mathcal G_{\bU,n}\mathcal G_n
\right)
+O(\tau^3).
\end{aligned}
\nonumber
\end{align}
By applying the Strang-ETDRK2 scheme \eqref{eqn:ETDRK2_Strang} to the exact solution $(\bPhi(t_n),\bU(t_n))$, we obtain the numerical solution $(\bPhi^{n+1},\bU^{n+1})$ to approximate the exact solution $(\bPhi(t_{n+1}),\bU(t_{n+1}))$.

For $\bPhi^{\,n+\frac12}=\mathrm{ETDRK2}_{\phi,\frac{\tau}{2}}(\bPhi(t_n);\bU(t_n))$, we get
\begin{align}
    \bPhi^{\,n+\frac12}
    =
    \bPhi(t_n)+\frac{\tau}{2}\mathcal F_n
    +\frac{\tau^2}{8}
    \mathcal F_{\bPhi,n}\mathcal F_n
    +O(\tau^3), \nonumber
\end{align}
that leads to
\begin{align}
\begin{aligned}
   & \mathcal G(\bPhi^{n+\frac12},\bU(t_n)) ={} \mathcal G_n +\frac{\tau}{2} \mathcal G_{\bPhi,n}\mathcal F_n +O(\tau^2),\\
& \partial_{\bU} \mathcal G(\bPhi^{\,n+\frac12},\bU(t_n))=\mathcal G_{\bU,n}+O(\tau). 
\end{aligned}
\label{eq:G-stage-expansion}
\end{align}
Using \eqref{eq:G-stage-expansion}, $\bU^{\,n+1}=\mathrm{ETDRK2}_{u,\tau}(\bU(t_n);\bPhi^{\,n+\frac12})$ satisfies
\begin{align}
\begin{aligned}
    \bU^{n+1} ={}& \bU(t_n) +\tau\mathcal G(\bPhi^{n+\frac12},\bU(t_n)) + \frac{\tau^2}{2} \partial_{\bU}\mathcal G(\bPhi^{n+\frac12},\bU(t_n)) \mathcal G(\bPhi^{n+\frac12},\bU(t_n)) +O(\tau^3) \\ 
={}& \bU(t_n)+\tau\mathcal G_n +\frac{\tau^2}{2} \left( \mathcal G_{\bPhi,n}\mathcal F_n + \mathcal G_{\bU,n}\mathcal G_n \right) +O(\tau^3), 
\end{aligned}
\nonumber
\end{align}
which yields
\begin{align}
&\mathcal F(\bPhi^{n+\frac12},\bU^{n+1}) ={} \mathcal F_n +\frac{\tau}{2} \mathcal F_{\bPhi,n}\mathcal F_n +\tau \mathcal F_{\bU,n}\mathcal G_n +O(\tau^2),\nonumber\\
&\partial_{\bPhi}\mathcal F(\bPhi^{\,n+\frac12},\bU^{\,n+1})
    ={}\mathcal F_{\bPhi,n}+O(\tau).
\nonumber
\end{align}
At last, $\bPhi^{\,n+1}=\mathrm{ETDRK2}_{\phi,\frac{\tau}{2}}(\bPhi^{\,n+\frac12};\bU^{\,n+1})$ gives
\begin{align} 
\begin{aligned}
    \bPhi^{n+1} ={}& \bPhi^{n+\frac12} +\frac{\tau}{2} \mathcal F(\bPhi^{n+\frac12},\bU^{n+1}) \\ 
&+ \frac{\tau^2}{8} \partial_{\bPhi}\mathcal F(\bPhi^{n+\frac12},\bU^{n+1}) \mathcal F(\bPhi^{n+\frac12},\bU^{n+1}) +O(\tau^3) \\ 
={}& \bPhi(t_n)+\tau\mathcal F_n +\frac{\tau^2}{2} \left( \mathcal F_{\bPhi,n}\mathcal F_n + \mathcal F_{\bU,n}\mathcal G_n \right) +O(\tau^3).
\end{aligned}
\nonumber
\end{align}
Consequently, we have the local truncation error of order $O(\tau^3)$
\begin{align*}
    \|\bPhi^{n+1}-\bPhi(t_{n+1})\|_{2,h}+\|\bU^{n+1}-\bU(t_{n+1})\|_{2,h} = O(\tau^3). 
\end{align*}
For a fixed spatial grid, standard stability assumptions then give second-order temporal convergence.
\end{remark}

\subsection{Energy Dissipation of the Stabilized Alternating ETD Schemes}
\label{sec:energy-stability-reconstructed} 

We next establish the discrete energy dissipation law of the alternating ETD1 and Strang-ETDRK2 schemes. The proofs requires the following assumptions to be applied successively to each membrane and protein updates.

\begin{assumption} \label{ass:actual-stage-bounds}
Let $d\leq3$. For fixed interface parameters $\ve_\phi$ and $\ve_u$, there exist constants
\begin{align}
C_{\phi}>0,
\quad
C_{u}>0,
\nonumber
\end{align}
independent of mesh size $h$ and time step size $\tau$, such that any $\bPhi\in\{\bPhi^{n}\}_{n\geq0}$ and $\bU \in \{\bU^{n}\}_{n\geq0}$ generated by scheme \eqref{eqn:ETDRK1}, any $\bPhi\in\{\bPhi^{n},\widetilde{\bPhi}^{\,n+\frac{1}{2}}, \bPhi^{n+\frac{1}{2}}, \widetilde{\bPhi}^{\,n+1}\}_{n\geq0}$ and $\bU \in \{\bU^{n}, \widetilde{\bU}^{\,n+1}\}_{n\geq0}$ generated by scheme \eqref{eqn:ETDRK2_Strang} satisfy the following discrete regularity assumptions
\begin{align}
\|\bPhi\|_{H_h^4}
\leq C_{\phi},
\quad
\|\bU\|_{H_h^3}
\leq C_{u}.
\nonumber
\end{align}
Here, $H_h^m$ denotes the standard discrete Sobolev norm on the periodic
grid \cite{Strikwerda2004,GustafssonKreissOliger2013}.
\end{assumption}

\begin{remark} \label{rem:discrete-regularity}
On a periodic grid in dimensions $d\leq3$, the discrete regularity assumptions in Assumption~\ref{ass:actual-stage-bounds} provide a sufficient condition
for the following three bounds required in the proof of the discrete energy dissipation 
\begin{align}
  & \norm{\bU}_{\infty,h}\leq R_u, \ \norm{\bB_h(\bPhi)}_{\infty,h}\leq R_B,\ \norm{\boldsymbol e_{u,h}(\bU)}_{\infty,h}\leq R_e,
\nonumber
\end{align}
where $R_u,R_B,R_e>0$ are independent of mesh size $h$ and time step size $\tau$ for fixed interface parameters $\ve_\phi$ and $\ve_u$. These bounds are obtained from the discrete Sobolev embedding inequalities.

Although these uniform bounds are not proved in the current work, systematic numerical experiments for various spatial mesh sizes $h$ and temporal step sizes $\tau$ suggest that such bounds exist. In the future work, we will further investigate whether the regularity assumptions in Assumption~\ref{ass:actual-stage-bounds} can be established from sufficiently smooth initial data.
\end{remark}

We recall the following constants
\begin{align*}
    &  L_{\kappa''} = \|\kappa_{{\rm bend}}''\|_{L^{\infty}}, \quad \kappa_{\rm ref}=\max\{\kappa_{-},\kappa_{+}\}, \\ & L_{W''} = \|W''\|_{L^{\infty}}, \quad
     L_g = \|g\|_{L^{\infty}},  \quad L_{g'} = \|g'\|_{L^{\infty}},
\end{align*}
and denote by $L_{g''}$ the Lipschitz constant of $g'$ and by $L_{W'''}$ the Lipschitz constant of $W''$. It follows that, for all $a,b\in\mathbb{R}$,
\begin{align}
  & |W'(a)-W'(b)-W''(b)(a-b)|
  \leq
  \frac{L_{W'''}}{2}|a-b|^2, 
  \label{eq:global-W-lip} \\
  & |g'(a)-g'(b)|
  \leq L_{g''}|a-b|, \quad
  |g(a)-g(b)-g'(b)(a-b)|
  \leq
  \frac{L_{g''}}{2}|a-b|^2.
\label{eq:global-g-lip}
\end{align}
Furthermore, for any $f\in \calM_{h}$, following an argument similar to that in \cite{luoETDmembrane,xu2020}, we have
\begin{align}
  \iph{(-\Delta_{h})^{-1} f}{f}
  \leq C_{\Delta,2} \norm{f}_{2,h}^2, \quad \norm{(-\Delta_{h})^{-1} f}_{\infty,h}\leq 2C_{\Delta,\infty}\norm{f}_{\infty,h}.
\label{eq:inverse-Linf-bound}
\end{align}
where $C_{\Delta,2}, C_{\Delta,\infty}$ depend on $\Omega$ but are independent of the mesh size.

Using the bounds and constants introduced above, we first establish the estimate of the energy difference for $\phi$-substep in Lemma~\ref{prop:direct-phi-increment}, which will be used in the proofs of Theorems~\ref{thm:ETD1-direct-energy} and \ref{thm:ETDRK2-direct-energy}.

\begin{lemma}\label{prop:direct-phi-increment}
Fixed $\bU$ and let $\bPhi,\bPhi^+\in \calM_{h}$ satisfy Assumption~\ref{ass:actual-stage-bounds}. If the stabilization parameters $\alpha_1,\alpha_2$ satisfy
\begin{align}
  \alpha_1&\geq\frac12L_{W''}, \quad
  \alpha_1\alpha_2\geq
  L_{W''}^2
  +\frac{\ve_{\phi}}{2}
  R_BL_{W'''}
  +\frac{\ve_{\phi}^3}{2\kappa_{\rm ref}}C_{\phi,{\rm rest}},
\label{eq:A1A2-direct-bound}
\end{align}
where 
\begin{align}
& C_{\phi,{\rm rest}}
={}K_{{\rm vol}}\abs\Omega
+\lambda_u(L_{g''}R_e
+C_{{\rm nl},\phi}
+C_{{\rm mass},\phi}), \nonumber\\
& C_{{\rm nl},\phi}
={}\gamma(R_u+\omega)^2\left[
2L_{g''}C_{\Delta,\infty}L_g
+C_{\Delta,2}L_{g'}^2
\right],\nonumber\\
& C_{{\rm mass},\phi}
={} K_{{\rm mass}}(R_u+\omega)^2\abs\Omega
\left(L_gL_{g''}+L_{g'}^2\right).
\nonumber
\end{align}
Then, 
\begin{align}
& E_{{\rm total},h}(\bPhi^+,\bU)- E_{{\rm total},h}(\bPhi,\bU)\leq
\iph{\frac{\delta  E_{{\rm total},h}(\bPhi,\bU)}{\delta\bPhi}
+\calL_{\phi,h}(\bPhi^+-\bPhi)}
{\bPhi^+-\bPhi}.
\nonumber
\end{align}
\end{lemma}

\begin{proof}
We fix $\bU$ and let $(\bPhi,\bPhi^+)$ satisfy Assumption~\ref{ass:actual-stage-bounds}. Decompose the energy difference as
\begin{align}
&E_{{\rm total},h}(\bPhi^+,\bU)-E_{{\rm total},h}(\bPhi,\bU)
\nonumber\\
={}&
\underbrace{
E_{{\rm bend},h}(\bPhi^+,\bU)
-E_{{\rm bend},h}(\bPhi,\bU)
}_{\mathbf I_\phi}
+
\underbrace{
E_{{\rm geom},h}(\bPhi^+)-E_{{\rm geom},h}(\bPhi)
}_{\mathbf{II}_\phi}
\nonumber\\
&+\lambda_u\left[
\begin{aligned}
  &  \underbrace{
E_{{\rm loc},h}(\bPhi^+,\bU)
-E_{{\rm loc},h}(\bPhi,\bU)
}_{\mathbf{III}_\phi}
+
\underbrace{
E_{{\rm nl},h}(\bPhi^+,\bU)
-E_{{\rm nl},h}(\bPhi,\bU)
}_{\mathbf{IV}_\phi}
\nonumber\\
&+
\underbrace{
E_{{\rm mass},h}(\bPhi^+,\bU)
-E_{{\rm mass},h}(\bPhi,\bU)
}_{\mathbf V_\phi}
\end{aligned}\right].
\nonumber
\end{align}

For the bending energy, we have
\begin{align}
\begin{aligned}
\mathbf I_\phi
={}&\frac1{\ve_{\phi}}
\iph{\kb(\bU)\odotv\bB_h(\bPhi)}{\delta\bB}
+\frac1{2\ve_{\phi}}
\iph{\kb(\bU)\odotv(\delta\bB)^2}{\one},
\end{aligned}
\nonumber
\end{align}
where 
\begin{align}
  \delta\bB = \bB_h(\bPhi^+) - \bB_h(\bPhi)
  =\ve_{\phi}\Delta_h(\bPhi^+-\bPhi)
  -\frac1{\ve_{\phi}}\left[W'(\bPhi^+)-W'(\bPhi)\right].
\nonumber
\end{align}
Adding and subtracting
$W''(\bPhi)\odotv(\bPhi^+-\bPhi)$, we obtain
\begin{align}
\begin{aligned}
\mathbf I_\phi
={}&\iph{\frac{\delta E_{{\rm bend},h}(\bPhi,\bU)}{\delta \bPhi}}{\bPhi^+-\bPhi}
+\frac1{2\ve_{\phi}}
\iph{\kb(\bU)\odotv(\delta\bB)^2}{\one}\\
&-\frac1{\ve_{\phi}^2}
\iph{\kb(\bU)\odotv\bB_h(\bPhi)\odotv
\left[W'(\bPhi^+)-W'(\bPhi)-W''(\bPhi)\odotv(\bPhi^+-\bPhi)\right]}{\one}.
\end{aligned}
\nonumber
\end{align}
Using the Lipschitz condition \eqref{eq:global-W-lip} and $(a+b)^2\leq2a^2+2b^2$,
\begin{align}
\begin{aligned}
    \mathbf I_\phi
\leq{}&\iph{\frac{\delta E_{{\rm bend},h}(\bPhi,\bU)}{\delta \bPhi}}{\bPhi^+-\bPhi}
+\kappa_{\max}\ve_{\phi}\norm{-\Delta_h(\bPhi^+-\bPhi)}_{2,h}^2\\
&+\left[
\frac{\kappa_{\max}L_{W''}^2}{\ve_{\phi}^3}
+\frac{\kappa_{\max}R_BL_{W'''}}{2\ve_{\phi}^2}
\right]\norm{\bPhi^+-\bPhi}_{2,h}^2,
\end{aligned}
\label{eq:bending-difference-bound}
\end{align}
where $\kappa_{\max} = \max\{\kappa_{-},\kappa_{+}\}$.
For the surface and volume terms,
\begin{align}
\begin{aligned}
\mathbf{II}_\phi
\leq{}&
\iph{
\frac{\delta E_{{\rm geom},h}(\bPhi)}{\delta \bPhi}
}{\bPhi^+-\bPhi}
+
\frac{\lambda_{\rm surf}\ve_\phi}{2}
\iph{-\Delta_h(\bPhi^+-\bPhi)}{\bPhi^+-\bPhi}\\
&+
\left[
\frac{\lambda_{\rm surf}L_{W''}}{2\ve_\phi}
+\frac{K_{{\rm vol}}\abs{\Omega}}2
\right]
\norm{\bPhi^+-\bPhi}_{2,h}^2.
\end{aligned}
\label{eq:surface-difference-bound}
\end{align}
Using the Lipschitz condition \eqref{eq:global-g-lip} and
\begin{align}
  E_{{\rm loc},h}(\bPhi,\bU)
  =\iph{g(\bPhi)\odotv
  \boldsymbol e_{u,h}(\bU)}{\one},
\nonumber
\end{align}
the short-range term satisfies
\begin{align}
\mathbf{III}_\phi
\leq
\iph{\frac{\delta E_{{\rm loc},h}(\bPhi,\bU)}{\delta \bPhi}}
{\bPhi^+-\bPhi}
+
\frac{L_{g''}R_e}{2}
\norm{\bPhi^+-\bPhi}_{2,h}^2.
\label{eq:local-phi-difference-bound}
\end{align}
For the long-range term, we have
\begin{align}
\begin{aligned}
\mathbf{IV}_\phi
={}&\gamma\iph{\bPsi_h}
{\delta g\odotv(\bU-\omega\one)}
+\frac\gamma2\iph{(-\Delta_h)^{-1}\left[\delta g\odotv\left(\bU-\omega\one\right)  \right]}{\delta g\odotv\left(\bU-\omega\one\right)}.
\end{aligned}
\nonumber
\end{align}
where $\delta g=g(\bPhi^+)-g(\bPhi)$. 
By \eqref{eq:global-g-lip} and \eqref{eq:inverse-Linf-bound}, we obtain 
\begin{align}
\begin{aligned}
\mathbf{IV}_\phi
\leq{}\iph{\frac{\delta E_{{\rm nl},h}(\bPhi,\bU)}{\delta \bPhi}}{\bPhi^+-\bPhi} + \frac{C_{{\rm nl},\phi}}{2}\norm{\bPhi^+-\bPhi}_{2,h}^2,
\end{aligned}
\label{eq:nonlocal-phi-difference-bound}
\end{align}
where $C_{{\rm nl},\phi}$ is given by
\begin{align*}
 C_{{\rm nl},\phi}
={}&\gamma(R_u+\omega)^2\left[
2L_{g''}C_{\Delta,\infty}L_g
+C_{\Delta,2}L_{g'}^2
\right].
\end{align*}
Finally, since
$\norm{\bU-\omega\one}_{\infty,h}\leq R_u+\omega$ and
\begin{align}
&m_h(\bPhi^+,\bU)-m_h(\bPhi,\bU)
=\iph{\delta g\odotv(\bU-\omega\one)}{\one},
\nonumber
\end{align}
the soft penalty term gives
\begin{align}
\begin{aligned}
\mathbf V_\phi
\leq{}&\iph{\frac{\delta E_{{\rm mass},h}(\bPhi,\bU)}{\delta \bPhi}}{\bPhi^+-\bPhi}
+\frac12K_{{\rm mass}}(R_u+\omega)^2\abs\Omega
  \left(L_gL_{g''}+L_{g'}^2\right)\norm{\bPhi^+-\bPhi}_{2,h}^2 \\
={}&\iph{\frac{\delta E_{{\rm mass},h}(\bPhi,\bU)}{\delta \bPhi}}{\bPhi^+-\bPhi}
+\frac{C_{{\rm mass},\phi}}2\norm{\bPhi^+-\bPhi}_{2,h}^2
\end{aligned}
\label{eq:mass-phi-difference-bound}
\end{align}
Then, denoting
\begin{align}
C_{\phi,{\rm rest}}
={}&K_{{\rm vol}}\abs\Omega
+\lambda_u(L_{g''}R_e
+C_{{\rm nl},\phi}
+C_{{\rm mass},\phi}),
\nonumber
\end{align}
and adding \eqref{eq:bending-difference-bound}, \eqref{eq:surface-difference-bound}, \eqref{eq:local-phi-difference-bound}, \eqref{eq:nonlocal-phi-difference-bound}, and \eqref{eq:mass-phi-difference-bound}, we obtain
\begin{align}
\begin{aligned}
&E_{{\rm total},h}(\bPhi^+,\bU)-E_{{\rm total},h}(\bPhi,\bU)\\
\leq{}&\iph{\frac{\delta E_{{\rm total},h}(\bPhi,\bU)}{\delta\bPhi}}{\bPhi^+-\bPhi}
+\kappa_{\max}\ve_{\phi}\norm{-\Delta_h(\bPhi^+-\bPhi)}_{2,h}^2
+\frac{\lambda_{\rm surf}\ve_{\phi}}{2}
\iph{-\Delta_h(\bPhi^+-\bPhi)}{\bPhi^+-\bPhi}\\
&+\Bigg[
\frac{\kappa_{\max}L_{W''}^2}{\ve_{\phi}^3}
+\frac{\kappa_{\max}R_BL_{W'''}}{2\ve_{\phi}^2}
+\frac{\lambda_{\rm surf}L_{W''}}{2\ve_{\phi}}
+\frac{C_{\phi,{\rm rest}}}{2}
\Bigg]\norm{\bPhi^+-\bPhi}_{2,h}^2.
\end{aligned}
\nonumber
\end{align}
Also, since
\begin{align}
  \iph{(-\Delta_h)^2(\bPhi^+-\bPhi)}{\bPhi^+-\bPhi}
  =\norm{-\Delta_h(\bPhi^+-\bPhi)}_{2,h}^2,\quad
  \iph{-\Delta_h(\bPhi^+-\bPhi)}{\bPhi^+-\bPhi}\geq0,
\nonumber
\end{align}
the definition \eqref{eq:Sphi-direct} gives
\begin{align}
\begin{aligned}
    \iph{\calL_{\phi,h}(\bPhi^+-\bPhi)}{\bPhi^+-\bPhi}
={}&
\kappa_{\rm ref}\ve_{\phi}
\norm{-\Delta_h(\bPhi^+-\bPhi)}_{2,h}^2
\\
&+
\left[
  \frac{\kappa_{\rm ref}(\alpha_1+\alpha_2)}{\ve_{\phi}}
  +\lambda_{\rm surf}\ve_{\phi}
\right]
\iph{-\Delta_h(\bPhi^+-\bPhi)}{\bPhi^+-\bPhi}\\
&+
\left[
  \frac{\kappa_{\rm ref}\alpha_1\alpha_2}{\ve_{\phi}^3}
  +\frac{\lambda_{\rm surf}\alpha_1}{\ve_{\phi}}
\right]
\norm{\bPhi^+-\bPhi}_{2,h}^2.
\end{aligned}
\label{eq:Sphi-quadratic-form}
\end{align}
Then, denoting
\begin{align}
\begin{aligned}
    \mathcal Q_{\phi,h}(\bPhi^+-\bPhi)
={}&
\kappa_{\max}\ve_{\phi}
\norm{-\Delta_h(\bPhi^+-\bPhi)}_{2,h}^2
+\frac{\lambda_{\rm surf}\ve_{\phi}}{2}
\iph{-\Delta_h(\bPhi^+-\bPhi)}{\bPhi^+-\bPhi}\\
&+
\Bigg[
  \frac{\kappa_{\max}L_{W''}^2}{\ve_{\phi}^3}
  +\frac{\kappa_{\max}R_BL_{W'''}}{2\ve_{\phi}^2}
  +\frac{\lambda_{\rm surf}L_{W''}}{2\ve_{\phi}}
  +\frac{C_{\phi,{\rm rest}}}{2}
\Bigg]
\norm{\bPhi^+-\bPhi}_{2,h}^2.
\end{aligned}
\label{eq:Qphi-direct}
\end{align}
Since $\alpha_1, \alpha_2$ satisfy \eqref{eq:A1A2-direct-bound}, subtracting \eqref{eq:Qphi-direct} from
\eqref{eq:Sphi-quadratic-form} yields
\begin{align}
\begin{aligned}
    &\iph{\calL_{\phi,h}(\bPhi^+-\bPhi)}{\bPhi^+-\bPhi}
-\mathcal Q_{\phi,h}(\bPhi^+-\bPhi)\\
={}&
(\kappa_{\rm ref}-\kappa_{\max})\ve_{\phi}
\norm{-\Delta_h(\bPhi^+-\bPhi)}_{2,h}^2+
\left[
  \frac{\kappa_{\rm ref}(\alpha_1+\alpha_2)}{\ve_{\phi}}
  +\frac{\lambda_{\rm surf}\ve_{\phi}}{2}
\right]
\iph{-\Delta_h(\bPhi^+-\bPhi)}{\bPhi^+-\bPhi}\\
&+
\frac{\lambda_{\rm surf}}{\ve_{\phi}}
\left(
  \alpha_1-\frac12L_{W''}
\right)
\norm{\bPhi^+-\bPhi}_{2,h}^2\\
&+
\Bigg[
  \frac{\kappa_{\rm ref}\alpha_1\alpha_2}{\ve_{\phi}^3}
  -\frac{\kappa_{\max}L_{W''}^2}{\ve_{\phi}^3}
  -\frac{\kappa_{\max}R_BL_{W'''}}{2\ve_{\phi}^2}
  -\frac{C_{\phi,{\rm rest}}}{2}
\Bigg]
\norm{\bPhi^+-\bPhi}_{2,h}^2 \geq 0. 
\end{aligned}
\nonumber
\end{align}
Therefore, we conclude that 
\begin{align}
  \mathcal Q_{\phi,h}(\bPhi^+-\bPhi)
  \leq
  \iph{\calL_{\phi,h}(\bPhi^+-\bPhi)}{\bPhi^+-\bPhi},
\nonumber
\end{align}
and thus
\begin{align}
\begin{aligned}
    &E_{{\rm total},h}(\bPhi^+,\bU)-E_{{\rm total},h}(\bPhi,\bU)
\leq{}
\iph{
  \frac{\delta E_{{\rm total},h}(\bPhi,\bU)}{\delta\bPhi}
  +\calL_{\phi,h}(\bPhi^+-\bPhi)
}{\bPhi^+-\bPhi}.
\end{aligned}
\nonumber
\end{align}
\end{proof}

Similarly, we can also establish the estimate of the energy difference for $u$-substep in Lemma~\ref{prop:direct-u-increment}.

\begin{lemma}\label{prop:direct-u-increment}
Fixed $\bPhi$ and let $\bU,\bU^+\in\calM_{h}$ satisfy Assumption~\ref{ass:actual-stage-bounds}. If the stabilization parameters $\beta_1,\beta_2$ satisfy
\begin{align}
  \beta_1\geq\frac12L_g, \quad
  \beta_2\geq\frac{C_{u,0}}{2\lambda_u},
\label{eq:B1B2-direct-bound}
\end{align}
where 
\begin{align}
C_{u,0}
={}&\frac{L_{\kappa''}R_B^2}{2\ve_{\phi}}
+\lambda_u\left[\frac{L_gL_{W''}}{\ve_u}
+C_{\Delta,2}\gamma L_g^2
+K_{{\rm mass}}L_g^2\abs\Omega\right].
\label{eq:Cu0-explicit}
\end{align}
Then,
\begin{align}
&E_{{\rm total},h}(\bPhi,\bU^+)-E_{{\rm total},h}(\bPhi,\bU)\leq
\iph{\frac{\delta E_{{\rm total},h}(\bPhi,\bU)}{\delta \bU}
+\calL_{u,h}(\bU^+-\bU)}
{\bU^+-\bU}.
\nonumber
\end{align}
\end{lemma}

\begin{proof}

We fix $\bPhi$ and decompose the energy difference
\begin{align}
&E_{{\rm total},h}(\bPhi,\bU^+)-E_{{\rm total},h}(\bPhi,\bU)
\nonumber\\
={}&
\underbrace{
E_{{\rm bend},h}(\bPhi,\bU^+)
-E_{{\rm bend},h}(\bPhi,\bU)
}_{\mathbf I_u} \nonumber\\
&+
\lambda_u\left[\begin{aligned}
&\underbrace{
E_{{\rm loc},h}(\bPhi,\bU^+)
-E_{{\rm loc},h}(\bPhi,\bU)
}_{\mathbf{II}_u} + \underbrace{
E_{{\rm nl},h}(\bPhi,\bU^+)
-E_{{\rm nl},h}(\bPhi,\bU)
}_{\mathbf{III}_u}\\
&+
\underbrace{
E_{{\rm mass},h}(\bPhi,\bU^+)
-E_{{\rm mass},h}(\bPhi,\bU)
}_{\mathbf{IV}_u}
\end{aligned}
\right].
\nonumber
\end{align}

The bending term satisfies
\begin{align}
\begin{aligned}
\mathbf I_u
\leq{}&\iph{\frac{\delta E_{{\rm bend},h}(\bPhi,\bU)}{\delta \bU}}{\bU^+-\bU}
+\frac{L_{\kappa''}R_B^2}{4\ve_{\phi}}\norm{\bU^+-\bU}_{2,h}^2.
\end{aligned}
\label{eq:bending-u-difference-bound}
\end{align}
The short-range energy satisfies
\begin{align}
\begin{aligned}
\mathbf{II}_u
\leq{}&\iph{\frac{\delta E_{{\rm loc},h}(\bPhi,\bU)}{\delta \bU}}{\bU^+-\bU}
\\& +\frac{\ve_uL_g}2
\iph{-\Delta_h(\bU^+-\bU)}{\bU^+-\bU}+\frac{L_gL_{W''}}{2\ve_u}\norm{\bU^+-\bU}_{2,h}^2.
\end{aligned}
\label{eq:local-u-difference-bound}
\end{align}
For the long-range term, we get 
\begin{align}
\begin{aligned}
\mathbf{III}_u
={}&\gamma\iph{g(\bPhi)\odotv\bPsi_h}{\bU^+-\bU} +\frac\gamma2\iph{(-\Delta_h)^{-1}\left[g(\bPhi)\odotv\left(\bU^+-\bU\right)\right]}{g(\bPhi)\odotv\left(\bU^+-\bU\right)}. \nonumber
\end{aligned}
\nonumber
\end{align}
Since $0\leq g(\bPhi)\leq L_g\one$, we have $\norm{g(\bPhi)\odotv(\bU^+-\bU)}_{2,h}\leq L_g\norm{\bU^+-\bU}_{2,h}$.
It follows from \eqref{eq:inverse-Linf-bound} that
\begin{align}
\begin{aligned}
\mathbf{III}_u
\leq{}&\iph{\frac{\delta E_{{\rm nl},h}(\bPhi,\bU)}{\delta \bU}}{\bU^+-\bU}
+\frac{\gamma C_{\Delta,2}L_g^2}{2}\norm{\bU^+-\bU}_{2,h}^2.
\end{aligned}
\label{eq:nonlocal-u-difference-bound}
\end{align}
Finally,
\begin{align}
\begin{aligned}
\mathbf{IV}_u
={}&\iph{\frac{\delta E_{{\rm mass},h}(\bPhi,\bU)}{\delta \bU}}{\bU^+-\bU}
+\frac{K_{{\rm mass}}}{2}\iph{g(\bPhi)}{\bU^+-\bU}^2\\
\leq{}&\iph{\frac{\delta E_{{\rm mass},h}(\bPhi,\bU)}{\delta \bU}}{\bU^+-\bU}
+\frac{K_{{\rm mass}}L_g^2\abs\Omega}{2}\norm{\bU^+-\bU}_{2,h}^2.
\end{aligned}
\label{eq:mass-u-difference-bound}
\end{align}
Therefore, denoting
\begin{align}
C_{u,0}
={}&\frac{L_{\kappa''}R_B^2}{2\ve_{\phi}}
+\lambda_u\left[\frac{L_gL_{W''}}{\ve_u}
+C_{\Delta,2}\gamma L_g^2
+K_{{\rm mass}}L_g^2\abs\Omega\right],
\nonumber
\end{align}
and applying \eqref{eq:bending-u-difference-bound}, \eqref{eq:local-u-difference-bound}, \eqref{eq:nonlocal-u-difference-bound}, and \eqref{eq:mass-u-difference-bound}, we obtain
\begin{align}
\begin{aligned}
&E_{{\rm total},h}(\bPhi,\bU^+)-E_{{\rm total},h}(\bPhi,\bU)\\
\leq{}&
\iph{\frac{\delta E_{{\rm total},h}(\bPhi,\bU)}{\delta \bU}}{\bU^+-\bU}
+\frac{\lambda_u\ve_uL_g}2
  \iph{-\Delta_h(\bU^+-\bU)}{\bU^+-\bU}\\
&+
\Bigg[
  \frac{L_{\kappa''}R_B^2}{4\ve_{\phi}}
  +\frac{\lambda_uL_gL_{W''}}{2\ve_u}
  +\frac{C_{\Delta,2}\lambda_u\gamma L_g^2}{2}
  +\frac{\lambda_uK_{{\rm mass}}L_g^2\abs\Omega}{2}
\Bigg]
\norm{\bU^+-\bU}_{2,h}^2.
\end{aligned}
\nonumber
\end{align}
By the definition \eqref{eq:Cu0-explicit}, we have
\begin{align}
& E_{{\rm total},h}(\bPhi,\bU^+)- E_{{\rm total},h}(\bPhi,\bU)\
\leq{}
\iph{\frac{\delta E_{{\rm total},h}(\bPhi,\bU)}{\delta \bU}}{\bU^+-\bU}
+\mathcal Q_{u,h}(\bU^+-\bU),
\nonumber
\end{align}
where
\begin{align}
\mathcal Q_{u,h}(\bU^+-\bU)
={}&
\frac{\lambda_u\ve_uL_g}2
\iph{-\Delta_h(\bU^+-\bU)}{\bU^+-\bU}
+\frac{C_{u,0}}2\norm{\bU^+-\bU}_{2,h}^2.
\label{eq:Qu-direct-definition}
\end{align}
Then, from the definition~\eqref{eq:Su-direct}, we have
\begin{align}
\iph{\calL_{u,h}(\bU^+-\bU)}{\bU^+-\bU}
={}&
\lambda_u\ve_u\beta_1
\iph{-\Delta_h(\bU^+-\bU)}{\bU^+-\bU}
+\lambda_u\beta_2\norm{\bU^+-\bU}_{2,h}^2.
\label{eq:Su-quadratic-form-direct}
\end{align}
Subtracting \eqref{eq:Qu-direct-definition} from \eqref{eq:Su-quadratic-form-direct} yields
\begin{align}
\begin{aligned}
    &\iph{\calL_{u,h}(\bU^+-\bU)}{\bU^+-\bU}
-\mathcal Q_{u,h}(\bU^+-\bU)\\
={}&
\lambda_u\ve_u
\left(
  \beta_1-\frac12L_g
\right)
\iph{-\Delta_h(\bU^+-\bU)}{\bU^+-\bU}
+
\left(
  \lambda_u\beta_2-\frac12C_{u,0}
\right)
\norm{\bU^+-\bU}_{2,h}^2\\
={}&
\lambda_u\ve_u
\left(
  \beta_1-\frac12L_g
\right)
\iph{-\Delta_h(\bU^+-\bU)}{\bU^+-\bU}+
\lambda_u
\left(
  \beta_2-\frac{C_{u,0}}{2\lambda_u}
\right)
\norm{\bU^+-\bU}_{2,h}^2.
\end{aligned}
\nonumber
\end{align}
By \eqref{eq:B1B2-direct-bound}, the stabilization constants satisfy
\begin{align}
  \beta_1-\frac12L_g
  &\geq0,\quad
  \beta_2-\frac{C_{u,0}}{2\lambda_u}
  \geq0,
\nonumber
\end{align}
which yields
\begin{align}
  \mathcal Q_{u,h}(\bU^+-\bU)
  \leq
  \iph{\calL_{u,h}(\bU^+-\bU)}{\bU^+-\bU}.
\nonumber
\end{align}
Therefore, we conclude that
\begin{align}
& E_{{\rm total},h}(\bPhi,\bU^+)- E_{{\rm total},h}(\bPhi,\bU)
\leq{}
\iph{
  \frac{\delta E_{{\rm total},h}(\bPhi,\bU)}{\delta \bU}
  +\calL_{u,h}(\bU^+-\bU)
}{\bU^+-\bU}.
\nonumber
\end{align}
\end{proof}

We now prove the discrete energy dissipation of the alternating ETD1 \eqref{eqn:ETDRK1} and Strang-ETDRK2 \eqref{eqn:ETDRK2_Strang} schemes. For the proofs of Theorems~\ref{thm:ETD1-direct-energy} and \ref{thm:ETDRK2-direct-energy}, we introduce the following functions and negative-definite operators.
Define functions $\eta_1(a)$ and $\eta_2(a)$ as 
\begin{align}
  \eta_1(a)=a-\frac{a}{1-e^{-a}},\quad
  \eta_2(a)=a-\frac{a^2}{e^{-a}-1+a}, \quad a\neq 0.
\nonumber
\end{align}
It is clear that for every $a>0$,
\begin{align}
  \eta_1(a)<0,
  \quad
  \eta_2(a)-\frac12\eta_1(a)
  =\frac{a\left[e^{-a}(3+a-e^{-a})-2\right]}
  {2(e^{-a}-1+a)(1-e^{-a})}
  \leq0.
\nonumber
\end{align}
For a self-adjoint positive-definite linear operator $\calL$, we further define
\begin{align}
\calH_1(\tau,\calL)
&=\calL-\calL\left(I-e^{-\tau\calL}\right)^{-1} = \frac{\eta_1(\calL\tau)}{\tau},
\nonumber\\
\calH_2(\tau,\calL)
&=\calL-\tau\calL^2
\left(e^{-\tau\calL}-I+\tau\calL\right)^{-1} = \frac{\eta_2(\calL\tau)}{\tau}.
\nonumber
\end{align}
As shown in \cite{luoETDmembrane,fu2022}, these operators are negative-definite,
\begin{align}
  \calH_1(\tau,\calL)\prec0,
  \quad
  \calH_2(\tau,\calL)-\frac12\calH_1(\tau,\calL)\preceq0.
\label{eq:H-signs-direct}
\end{align}

\begin{theorem}[Energy dissipation of the alternating ETD1 method]
\label{thm:ETD1-direct-energy}
Let Assumption~\ref{ass:actual-stage-bounds} hold, and suppose the stabilization parameters satisfy Lemma \ref{prop:direct-phi-increment} and \ref{prop:direct-u-increment}. Then the alternating ETD1 method \eqref{eqn:ETDRK1} satisfies, for every $\tau>0$,
\begin{align}
   E_{{\rm total},h}(\bPhi^{n+1},\bU^{n+1})
  \leq E_{{\rm total},h}(\bPhi^n,\bU^n).
\nonumber
\end{align}
More precisely, 
we have
\begin{align}
\begin{aligned}
    & E_{{\rm total},h}(\bPhi^{n+1},\bU^{n+1})
- E_{{\rm total},h}(\bPhi^n,\bU^n)\\
\leq{}&
\iph{\calH_1(\tau,\calL_{\phi,h})(\bPhi^{n+1}-\bPhi^n)}
{\bPhi^{n+1}-\bPhi^n}
+
\iph{\calH_1(\tau,\calL_{u,h})(\bU^{n+1}-\bU^n)}
{\bU^{n+1}-\bU^n}
\leq0.
\end{aligned}
\nonumber
\end{align}
\end{theorem}

\begin{proof}
First, for $\bPhi^{n+1} = \mathrm{ETD1}_{\phi,\tau}(\bPhi^{n};\bU^{n})$ in \eqref{eqn:ETDRK1}, Lemma~\ref{prop:direct-phi-increment} gives
\begin{align}
\begin{aligned}
    & E_{{\rm total},h}(\bPhi^{n+1},\bU^n)- E_{{\rm total},h}(\bPhi^n,\bU^n)\\
&\quad\leq
\iph{\frac{\delta E_{{\rm total},h}(\bPhi^n,\bU^n)}{\delta\bPhi}
+\calL_{\phi,h}(\bPhi^{n+1}-\bPhi^n)}
{\bPhi^{n+1}-\bPhi^n}\\
&\quad=
\iph{\calL_{\phi,h}\bPhi^{n+1}
-\calR_{\phi,h}(\bPhi^n;\bU^n)}
{\bPhi^{n+1}-\bPhi^n}.
\end{aligned}
\label{eq:phi-before-H1}
\end{align}
From \eqref{eqn:ETDRK1} and $\tau\varphi_1(\tau\calL)=\calL^{-1}(I-e^{-\tau\calL})$,
\begin{align}
\calR_{\phi,h}(\bPhi^n;\bU^n)
={}&\calL_{\phi,h}\bPhi^n
+\calL_{\phi,h}
\left(I-e^{-\tau\calL_{\phi,h}}\right)^{-1}
(\bPhi^{n+1}-\bPhi^n).
\nonumber
\end{align}
Substitution into \eqref{eq:phi-before-H1} yields
\begin{align}
& E_{{\rm total},h}(\bPhi^{n+1},\bU^n)- E_{{\rm total},h}(\bPhi^n,\bU^n)\leq
\iph{\calH_1(\tau,\calL_{\phi,h})(\bPhi^{n+1}-\bPhi^n)}
{\bPhi^{n+1}-\bPhi^n}
\leq0.
\label{eq:phi-ETD1-decay-direct}
\end{align}
Second, for $\bU^{n+1} = \mathrm{ETD1}_{u,\tau}(\bU^{n};\bPhi^{n+1})$ in \eqref{eqn:ETDRK1}, repeating the same calculation with Lemma~\ref{prop:direct-u-increment} gives
\begin{align}
& E_{{\rm total},h}(\bPhi^{n+1},\bU^{n+1})
- E_{{\rm total},h}(\bPhi^{n+1},\bU^n)\leq
\iph{\calH_1(\tau,\calL_{u,h})(\bU^{n+1}-\bU^n)}
{\bU^{n+1}-\bU^n}
\leq0.
\label{eq:u-ETD1-decay-direct}
\end{align}
Adding \eqref{eq:phi-ETD1-decay-direct} and \eqref{eq:u-ETD1-decay-direct}, we conclude that
\begin{align}
\begin{aligned}
& E_{{\rm total},h}(\bPhi^{n+1},\bU^{n+1})
- E_{{\rm total},h}(\bPhi^n,\bU^n)
\\
\leq{}&
\iph{
\calH_1(\tau,\calL_{\phi,h})
(\bPhi^{n+1}-\bPhi^n)
}{
\bPhi^{n+1}-\bPhi^n
}
+
\iph{
\calH_1(\tau,\calL_{u,h})
(\bU^{n+1}-\bU^n)
}{
\bU^{n+1}-\bU^n
}
\leq0.
\end{aligned}
\nonumber
\end{align}
\end{proof}

\begin{theorem}[Energy dissipation of the alternating Strang-ETDRK2 method] \label{thm:ETDRK2-direct-energy} 
Let Assumption~\ref{ass:actual-stage-bounds} hold for the scheme \eqref{eqn:ETDRK2_Strang} and the stabilization parameters satisfy the conditions in Lemmas~\ref{prop:direct-phi-increment} and \ref{prop:direct-u-increment}. Then, for every $\tau>0$, the alternating Strang-ETDRK2 scheme \eqref{eqn:ETDRK2_Strang} satisfies 
\begin{align}  E_{{\rm total},h}(\bPhi^{n+1},\bU^{n+1}) \leq  E_{{\rm total},h}(\bPhi^n,\bU^n). \label{eq:Strang-full-energy-decay} 
\end{align} 
\end{theorem}

\begin{proof} 
Similar as in Theorem~\ref{thm:ETD1-direct-energy}, we apply Lemmas~\ref{prop:direct-phi-increment} and \ref{prop:direct-u-increment} to the three steps. For $\bPhi^{n+\frac{1}{2}} = \mathrm{ETDRK2}_{\phi,\frac{\tau}{2}}(\bPhi^{n};\bU^{n})$ in \eqref{eqn:ETDRK2_Strang}, we have
\begin{align} 
& \calR_{\phi,h}(\bPhi^n;\bU^n) ={} \calL_{\phi,h}\bPhi^n + \calL_{\phi,h} \left(I-e^{-\frac{\tau}{2}\calL_{\phi,h}}\right)^{-1} (\widetilde{\bPhi}^{\,n+\frac{1}{2}}-\bPhi^n), \label{eq:Strang-first-R-predictor} \\
& \calR_{\phi,h}(\widetilde{\bPhi}^{\,n+\frac{1}{2}};\bU^n) ={} \calR_{\phi,h}(\bPhi^n;\bU^n) + \frac{\tau}{2}\calL_{\phi,h}^2 \left( e^{-\frac{\tau}{2}\calL_{\phi,h}}-I+\frac{\tau}{2}\calL_{\phi,h} \right)^{-1} (\bPhi^{n+\frac{1}{2}}-\widetilde{\bPhi}^{\,n+\frac{1}{2}}). \label{eq:Strang-first-R-corrector}
\end{align} 
Applying Lemma~\ref{prop:direct-phi-increment} first to $(\bPhi^n,\widetilde{\bPhi}^{\,n+\frac{1}{2}})$ and then to $(\widetilde{\bPhi}^{\,n+\frac{1}{2}},\bPhi^{n+\frac{1}{2}})$ with $\bU^n$ fixed, we have
\begin{align} 
\begin{aligned} 
& E_{{\rm total},h}(\bPhi^{n+\frac{1}{2}},\bU^n) - E_{{\rm total},h}(\bPhi^n,\bU^n) \\ \leq{}&  \iph{ \calL_{\phi,h}\widetilde{\bPhi}^{\,n+\frac{1}{2}} -\calR_{\phi,h}(\bPhi^n;\bU^n) }{ \widetilde{\bPhi}^{\,n+\frac{1}{2}}-\bPhi^n } \\ &+  \iph{ \calL_{\phi,h}\bPhi^{n+\frac{1}{2}} -\calR_{\phi,h}(\widetilde{\bPhi}^{\,n+\frac{1}{2}};\bU^n) }{ \bPhi^{n+\frac{1}{2}}-\widetilde{\bPhi}^{\,n+\frac{1}{2}} }. 
\end{aligned} 
\label{eq:Strang-first-phi-two-increments} 
\end{align} 
Substituting \eqref{eq:Strang-first-R-predictor} and \eqref{eq:Strang-first-R-corrector} into \eqref{eq:Strang-first-phi-two-increments} yields 
\begin{align} 
\begin{aligned} 
& E_{{\rm total},h}(\bPhi^{n+\frac{1}{2}},\bU^n) - E_{{\rm total},h}(\bPhi^n,\bU^n) \\ 
\leq{}& \frac1{2} \iph{\calH_1\left(\frac{\tau}{2},\calL_{\phi,h}\right)(\widetilde{\bPhi}^{\,n+\frac{1}{2}}-\bPhi^n)} {\widetilde{\bPhi}^{\,n+\frac{1}{2}}-\bPhi^n} + \frac1{2} \iph{\calH_1\left(\frac{\tau}{2},\calL_{\phi,h}\right)(\bPhi^{n+\frac{1}{2}}-\bPhi^n)} {\bPhi^{n+\frac{1}{2}}-\bPhi^n} \\ 
&+  \iph{ \left( \calH_2\left(\frac{\tau}{2},\calL_{\phi,h}\right) -\frac12\calH_1\left(\frac{\tau}{2},\calL_{\phi,h}\right) \right)(\bPhi^{n+\frac{1}{2}}-\widetilde{\bPhi}^{\,n+\frac{1}{2}}) }{ \bPhi^{n+\frac{1}{2}}-\widetilde{\bPhi}^{\,n+\frac{1}{2}} } \leq0. 
\end{aligned} 
\label{eq:Strang-first-phi-energy-decay} 
\end{align} 
The last inequality follows from \eqref{eq:H-signs-direct}. For $\bU^{n+1} = \mathrm{ETDRK2}_{u,\tau}(\bU^{n};\bPhi^{n+\frac{1}{2}})$ in \eqref{eqn:ETDRK2_Strang}, Lemma~\ref{prop:direct-u-increment} gives 
\begin{align} 
\begin{aligned} 
& E_{{\rm total},h}(\bPhi^{n+\frac{1}{2}},\bU^{n+1}) - E_{{\rm total},h}(\bPhi^{n+\frac{1}{2}},\bU^n) \\ 
\leq{}& \frac1{2} \iph{\calH_1(\tau,\calL_{u,h})(\widetilde{\bU}^{\,n+1}-\bU^n)} {\widetilde{\bU}^{\,n+1}-\bU^n} + \frac1{2} \iph{\calH_1(\tau,\calL_{u,h})(\bU^{n+1}-\bU^n)} {\bU^{n+1}-\bU^n} \\
&+  \iph{ \left( \calH_2(\tau,\calL_{u,h}) -\frac12\calH_1(\tau,\calL_{u,h}) \right)(\bU^{n+1}-\widetilde{\bU}^{\,n+1}) }{ \bU^{n+1}-\widetilde{\bU}^{\,n+1} } \leq0. \end{aligned} 
\label{eq:Strang-u-energy-decay} 
\end{align} 
For $\bPhi^{n+1} = \mathrm{ETDRK2}_{\phi,\frac{\tau}{2}}(\bPhi^{n+\frac{1}{2}};\bU^{n+1})$ in \eqref{eqn:ETDRK2_Strang}, Lemma~\ref{prop:direct-phi-increment} gives an estimate analogous to \eqref{eq:Strang-first-phi-energy-decay}
\begin{align} 
\begin{aligned} 
& E_{{\rm total},h}(\bPhi^{n+1},\bU^{n+1}) - E_{{\rm total},h}(\bPhi^{n+\frac{1}{2}},\bU^{n+1})\leq0.
\end{aligned} 
\label{eq:Strang-second-phi-energy-decay} 
\end{align} 
Therefore, combining \eqref{eq:Strang-first-phi-energy-decay}, \eqref{eq:Strang-u-energy-decay}, and \eqref{eq:Strang-second-phi-energy-decay}, we conclude that 
\begin{align} 
\begin{aligned} 
&  E_{{\rm total},h}(\bPhi^{n+1},\bU^{n+1}) - E_{{\rm total},h}(\bPhi^n,\bU^n)\leq0,  
\end{aligned}\nonumber 
\end{align} 
which proves \eqref{eq:Strang-full-energy-decay}. 
\end{proof} 

\section{Numerical Experiments}\label{sec:numerics}

In this section, we present several numerical experiments to verify the discrete energy dissipation of the proposed alternating ETD schemes and to explore the effects of key model parameters on protein phase separation and membrane deformation. All simulations are performed on Cartesian grids with periodic boundary conditions. Spatial derivatives are approximated using the second-order finite difference operators introduced above. The discretized ETD operators are implemented efficiently using the fast Cosine transform \cite{du2021mbp}. At each time step, the total discrete free energy is computed from \eqref{eq:discrete-energy}. All computations were implemented on a laptop with a 2.7 GHz Intel CPU and 32 GB of memory.

For the 2D simulations, we take $\Omega=[-1,1)^2$ with $N_x=N_y=256$ and mesh size $h_x=h_y=h=\frac{2}{256}$. The interface widths are $\ve_{\phi}=10h$ and $\ve_u=5h$, and the time step size is $\tau=10^{-3}$. Unless otherwise stated, we set $\lambda_{\rm surf}=0.5$, $K_{\rm vol}=10$, and $K_{\rm mass}=10000$. The stabilization parameters $\alpha_1$, $\alpha_2$, $\beta_1$, and $\beta_2$ are chosen sufficiently large to satisfy the conditions in Lemmas~\ref{prop:direct-phi-increment} and~\ref{prop:direct-u-increment}. The prescribed volume $V_0$ is determined by the initial phase field vesicle $\bPhi^{0}$. The parameters $\omega$, $\kappa_-$, $\kappa_+$, $\lambda_u$, and $\gamma$ are specified separately for each experiment. For the 3D simulations, we use the domain $\Omega=[-1,1)^3$ with $N_x=N_y=N_z=128$. We take the time step size $\tau=10^{-3}$ and the interface widths $\ve_{\phi}=10h$ and $\ve_u=5h$. The fixed parameters are $\lambda_{\rm surf}=0.5$, $K_{{\rm vol}}=10$, $K_{{\rm mass}}=10000$. In the 2D figures, the white curve represents the membrane interface, and the color along the curve indicates the protein variable $u$. Red segments correspond to protein-rich subdomains with $u\approx1$, whereas white segments correspond to protein-poor subdomains with $u\approx0$. In the 3D figures, the membrane interface is represented by the isosurface $\phi=1/2$, which is also colored according to the protein variable $u$.

\subsection{Discrete Energy Dissipation of the ETD Schemes}
\label{subsec:numerical-energy-stability}

We first validate the discrete energy dissipation of the alternating ETD1 and Strang-ETDRK2 schemes. In both experiments, the initial phase field vesicle $\bPhi^0$ is given by
\begin{align}
        \bPhi^0 = 0.5+0.5\tanh{\left(\frac{3(r_0-r)}{\ve_{\phi}}\right)}, \quad r_0 = 0.4, \label{eqn:ini_phi}
\end{align}
where $r$ is the distance from each grid point to the center of the domain. The top and bottom panels of Figure~\ref{fig:energy_stab} start with $7$ and $6$ protein-rich subdomains on the membrane, respectively. In both experiments, we set $\omega=0.5$, $\kappa_-=3$, $\kappa_+=0.2$, and $\lambda_u=20$, with $\gamma=12000$ for the alternating ETD1 scheme and $\gamma=15000$ for the Strang-ETDRK2 scheme. During $0\leq t\leq200$, the membrane is held fixed, and only the protein variable $u$ evolves on the circular membrane. At $t=200$, when the protein pattern has nearly reached a stationary state, we allow the membrane to evolve by the fully coupled dynamics. The snapshots in Figure~\ref{fig:energy_stab} are taken at $t=0$, $10$, $200$, and $210$. In both experiments, the total discrete free energy decreases rapidly at the beginning and then slowly as the protein pattern approaches an equilibrium state on the fixed membrane. Once the membrane starts to evolve at $t=200$, the coupled dynamics leads to membrane deformation and a further decrease in the total free energy. These observations are consistent with the discrete energy dissipation law in Theorems~\ref{thm:ETD1-direct-energy} and \ref{thm:ETDRK2-direct-energy}.

\begin{figure}[t]
    \begin{center}
    \includegraphics[width=0.85\textwidth]{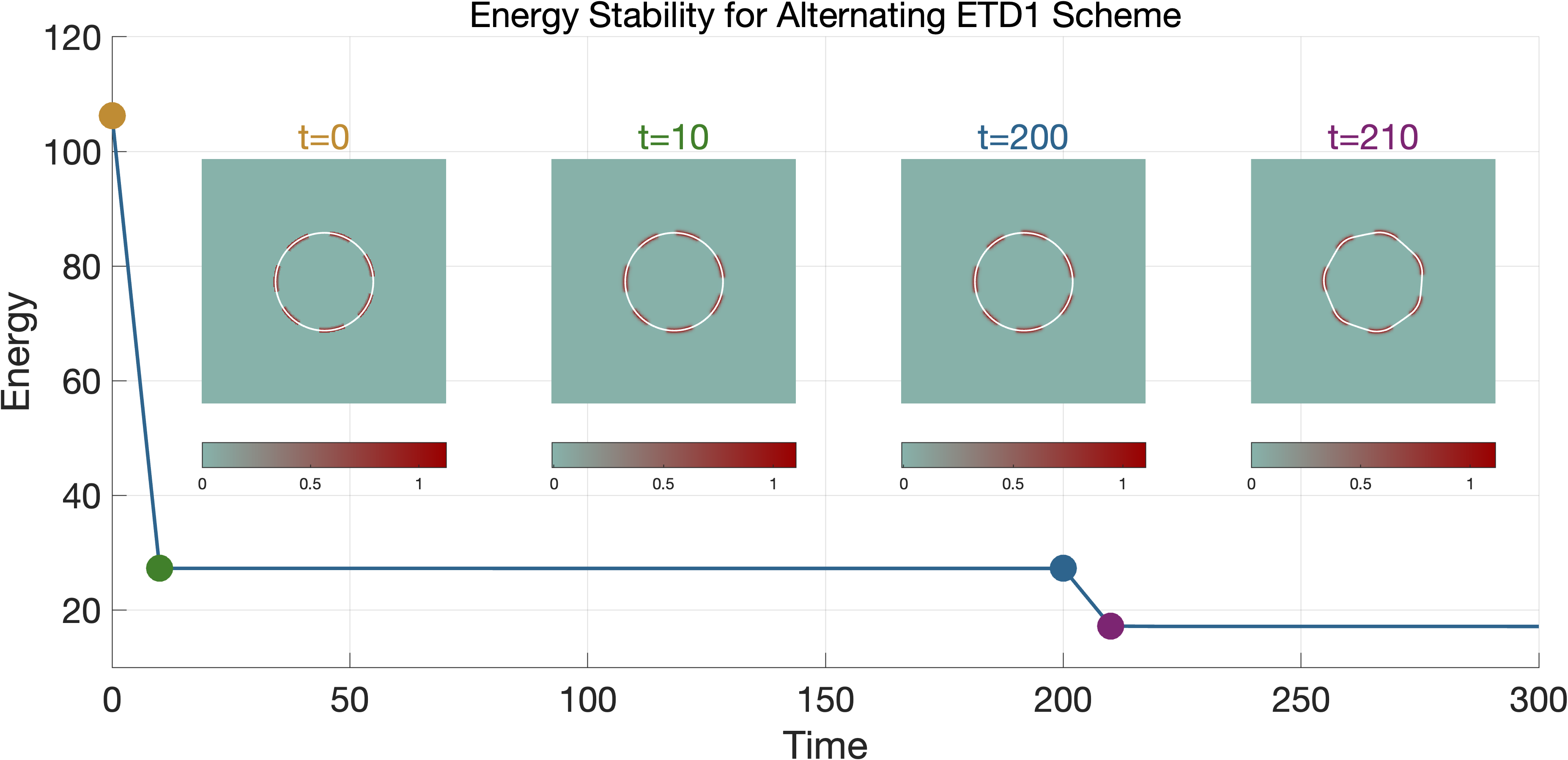} \\
    \vspace{2mm}
    \includegraphics[width=0.85\textwidth]{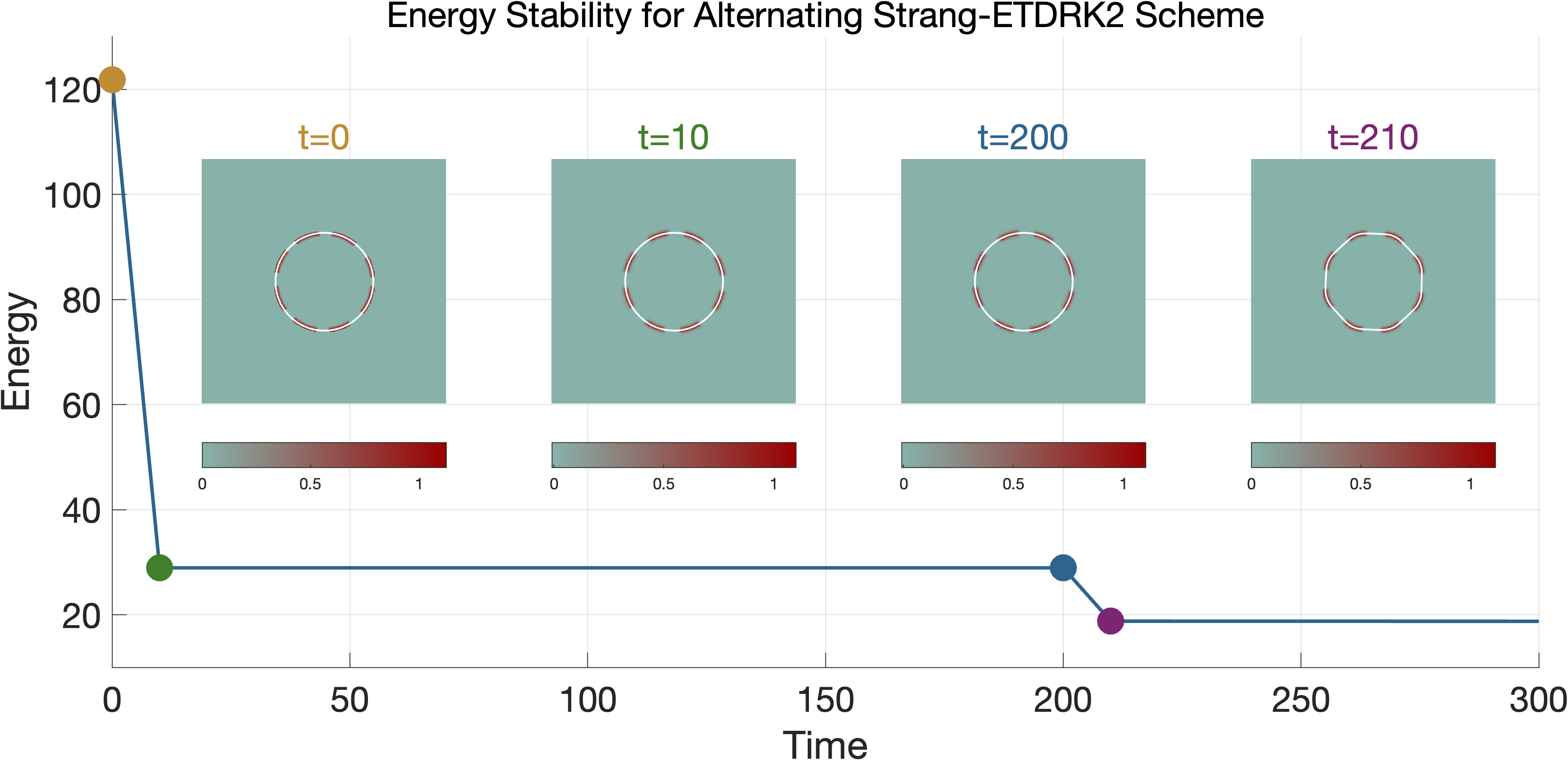}
    \end{center}
    \vspace{-0mm}
    \caption{Discrete energy evolution generated by the alternating ETD1 scheme (top) and the alternating Strang-ETDRK2 scheme (bottom). The membrane is fixed for $0\leq t\leq200$ and is then evolved. The insets show the numerical states at $t=0$, $10$, $200$, and $210$. For the first experiment, $\omega=0.5$, $(\kappa_-,\kappa_+)=(3,0.2)$, $\lambda_u=20$, and $\gamma=12000$. For the second experiment, $\omega=0.5$, $(\kappa_-,\kappa_+)=(3,0.2)$, $\lambda_u=20$, and $\gamma=15000$.}   
    \label{fig:energy_stab}
\end{figure}

\subsection{Effect of the Long-range Interaction Strength $\gamma$} \label{subsec:gamma-effect}

We next use the Strang-ETDRK2 scheme \eqref{eqn:ETDRK2_Strang} to explore the effect of the long-range interaction strength $\gamma$ on protein phase separation on a circular membrane. We fix $\omega=0.5$, $\kappa_-=4$, $\kappa_+=0.2$, and $\lambda_u=20$, and vary $\gamma=14000$, $17000$, $20000$, and $25000$. In all four simulations, the initial phase field vesicle is given by \eqref{eqn:ini_phi}. For each $\gamma$, the simulations start from a randomly generated initial protein variable $u$ on the diffuse membrane. Each simulation is then evolved until the equilibrium state is reached.

Figure~\ref{fig:gamma-effect} shows that $\gamma = 14000,17000,20000,25000$ give eight, nine, ten, and eleven protein-rich subdomains, respectively. For $\gamma=14000$, the system evolves to a pattern with eight relatively large protein-rich subdomains. As $\gamma$ increases, the stronger long-range interaction leads to more protein-rich subdomains with nearly uniform spacing along the membrane. This behavior is consistent with the competition between short-range and long-range interactions in the OK energy, in which the short-range interaction favors coarsening, whereas the long-range interaction favors smaller domains.

\begin{figure}[t]
    \begin{center}
    \includegraphics[width=0.2\textwidth]{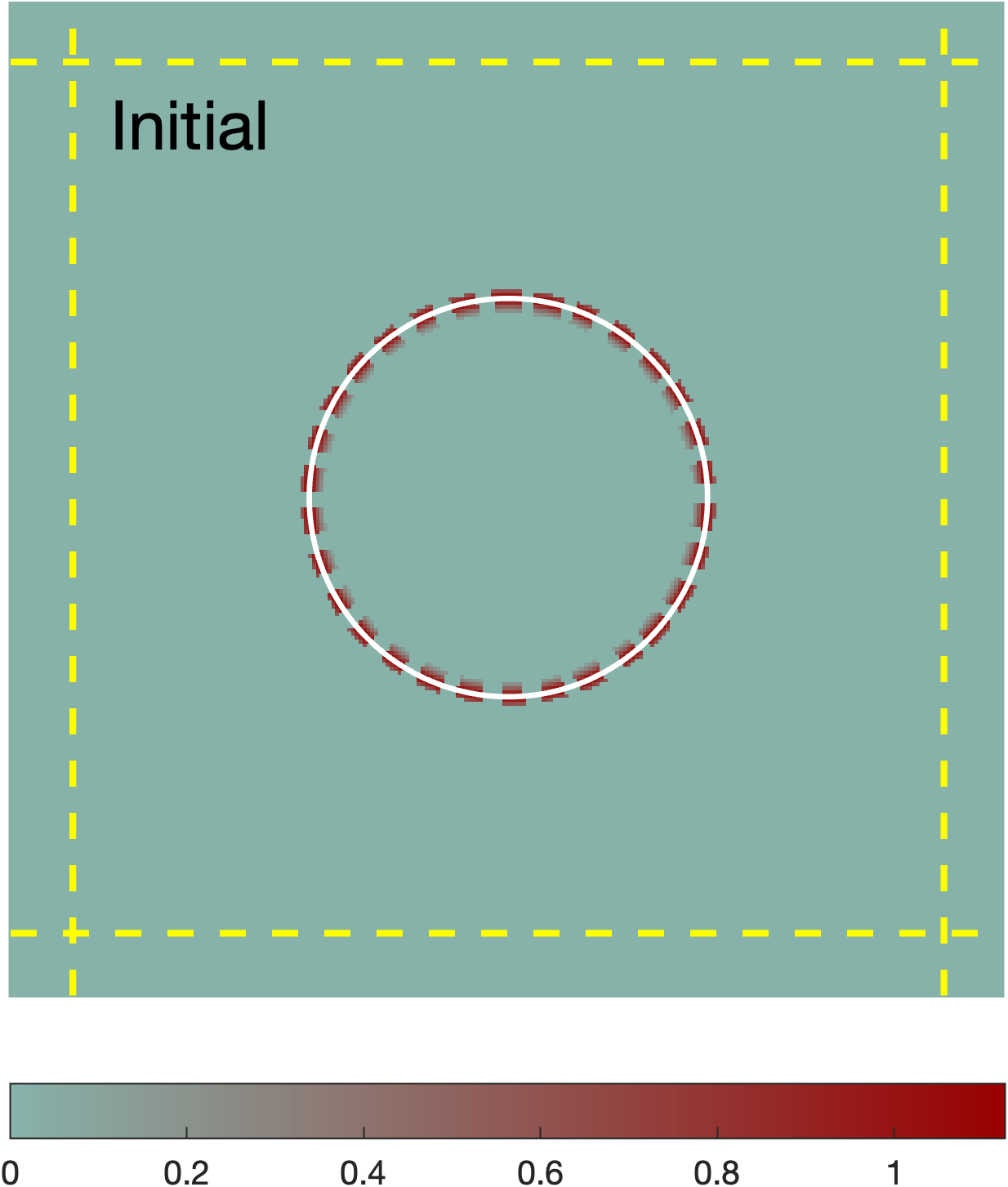} 
    \includegraphics[width=0.2\textwidth]{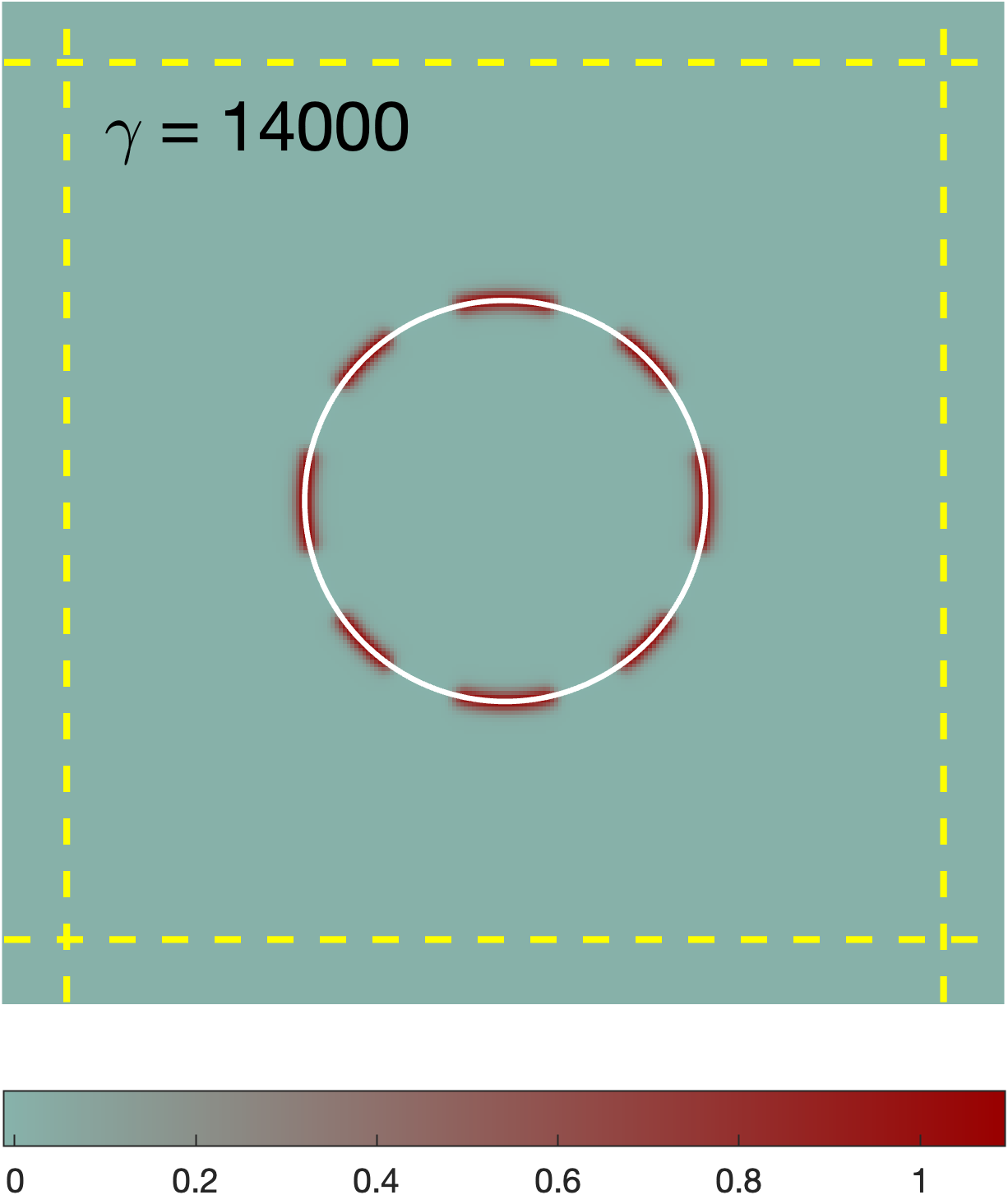}\qquad
    \includegraphics[width=0.2\textwidth]{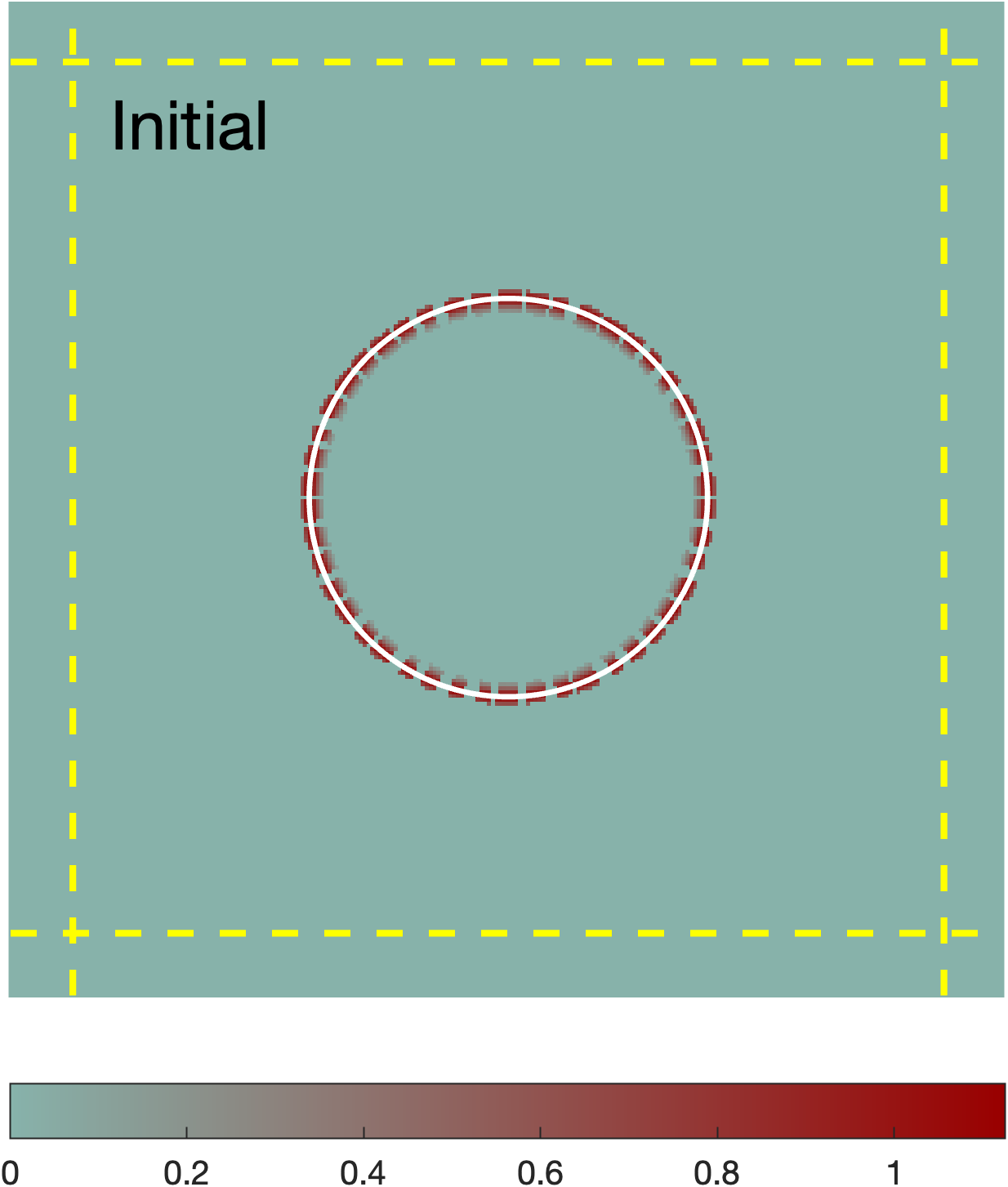} 
    \includegraphics[width=0.2\textwidth]{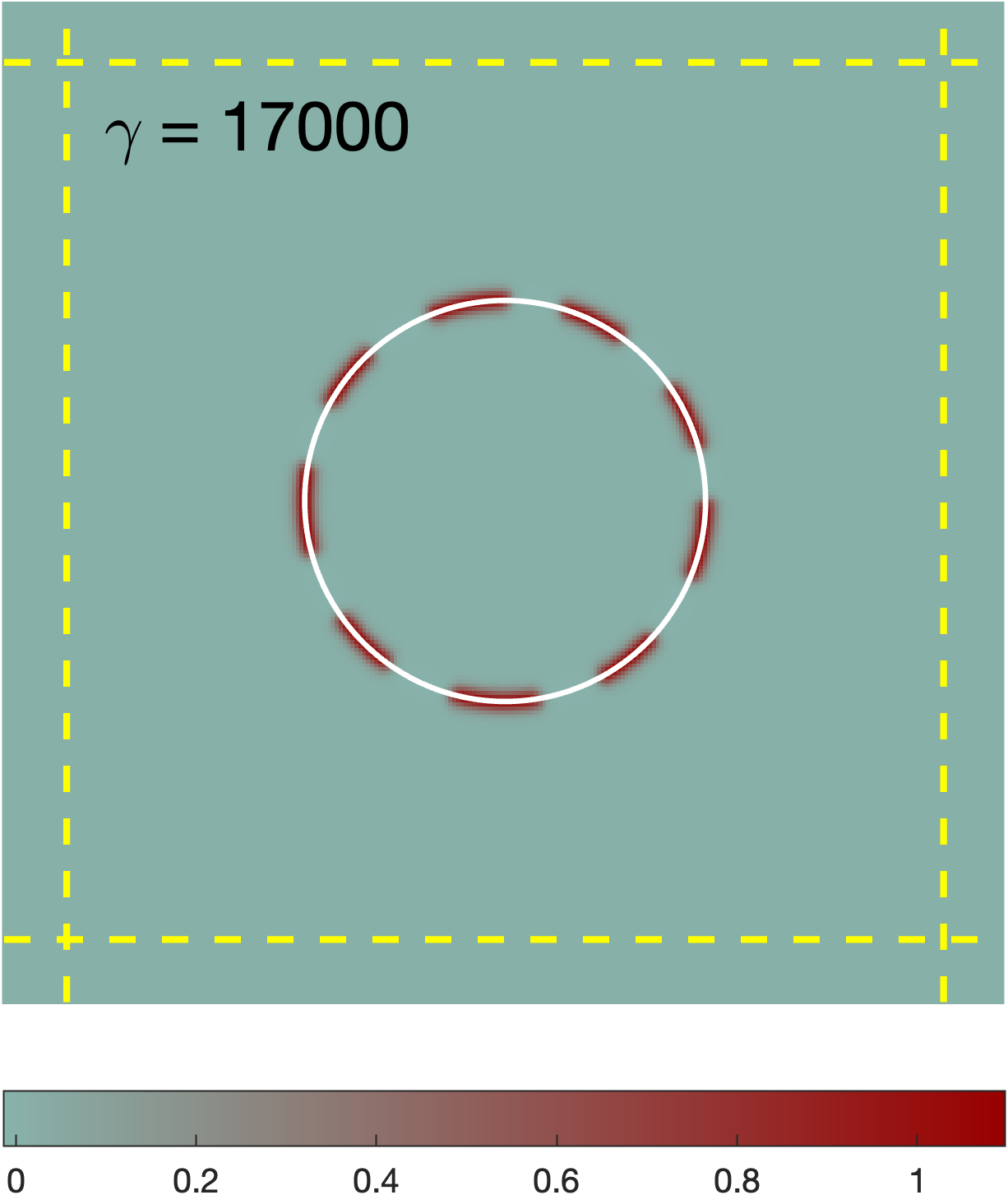}\\
    \vspace{3mm}
    \includegraphics[width=0.2\textwidth]{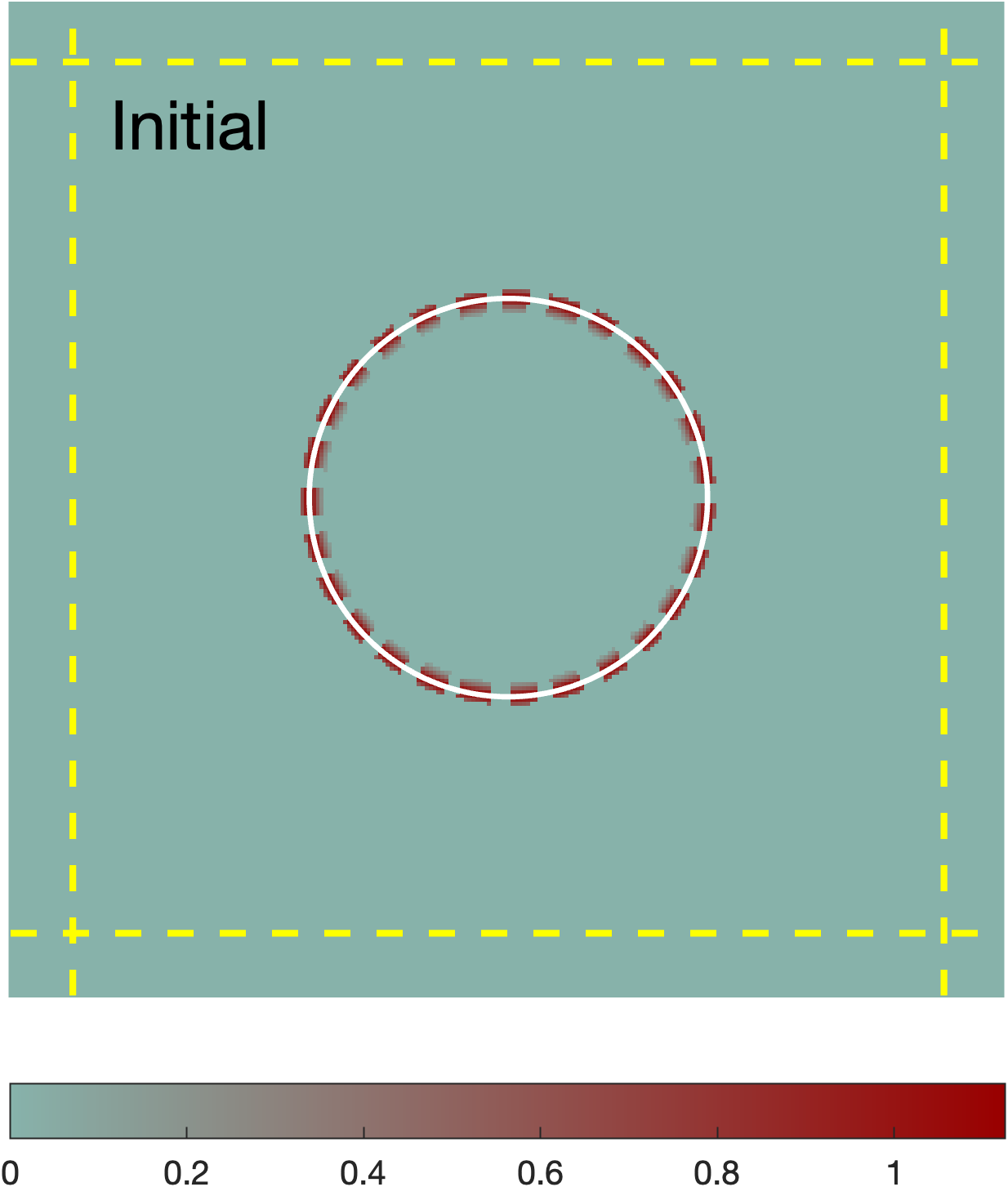} 
    \includegraphics[width=0.2\textwidth]{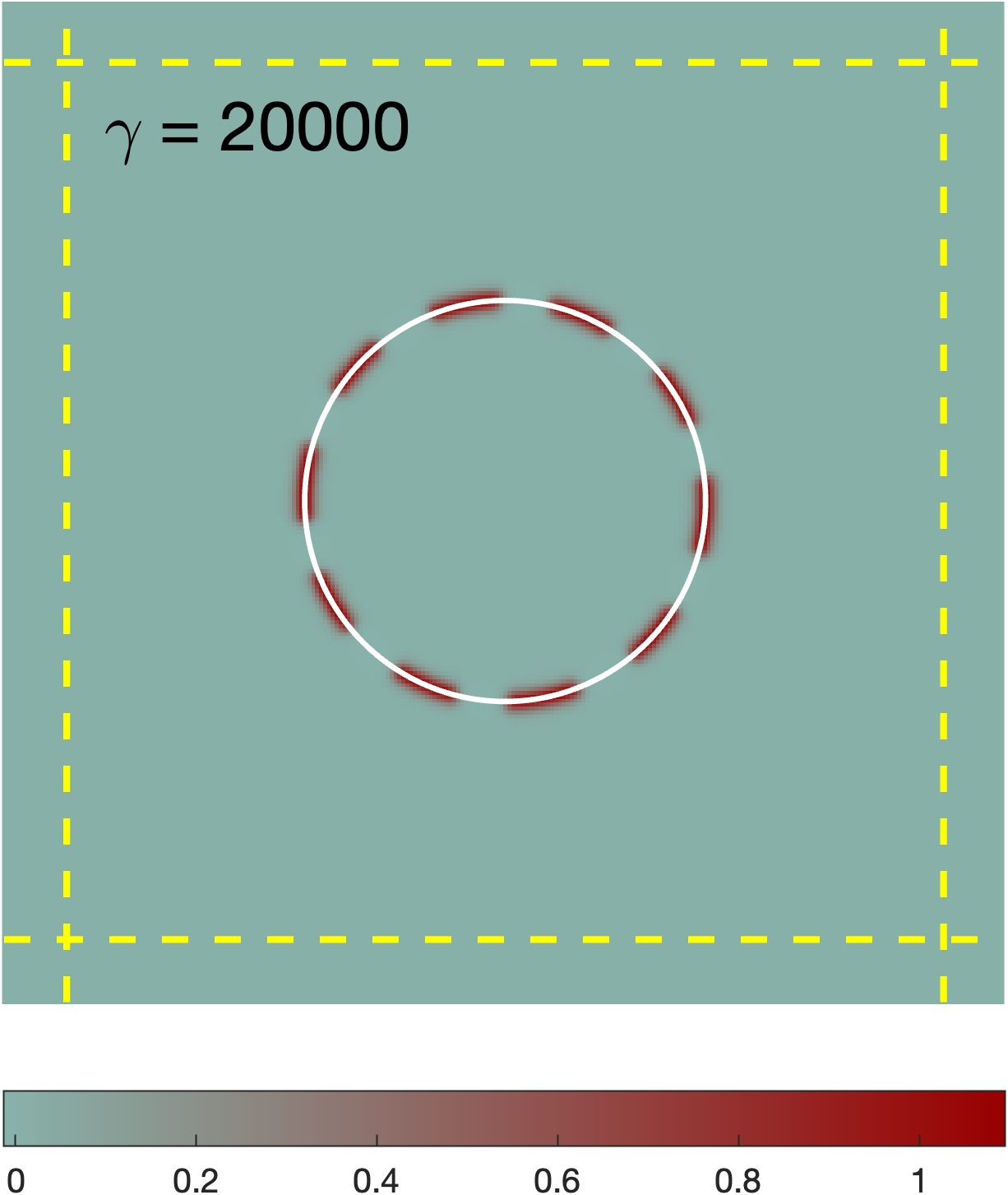}\qquad
    \includegraphics[width=0.2\textwidth]{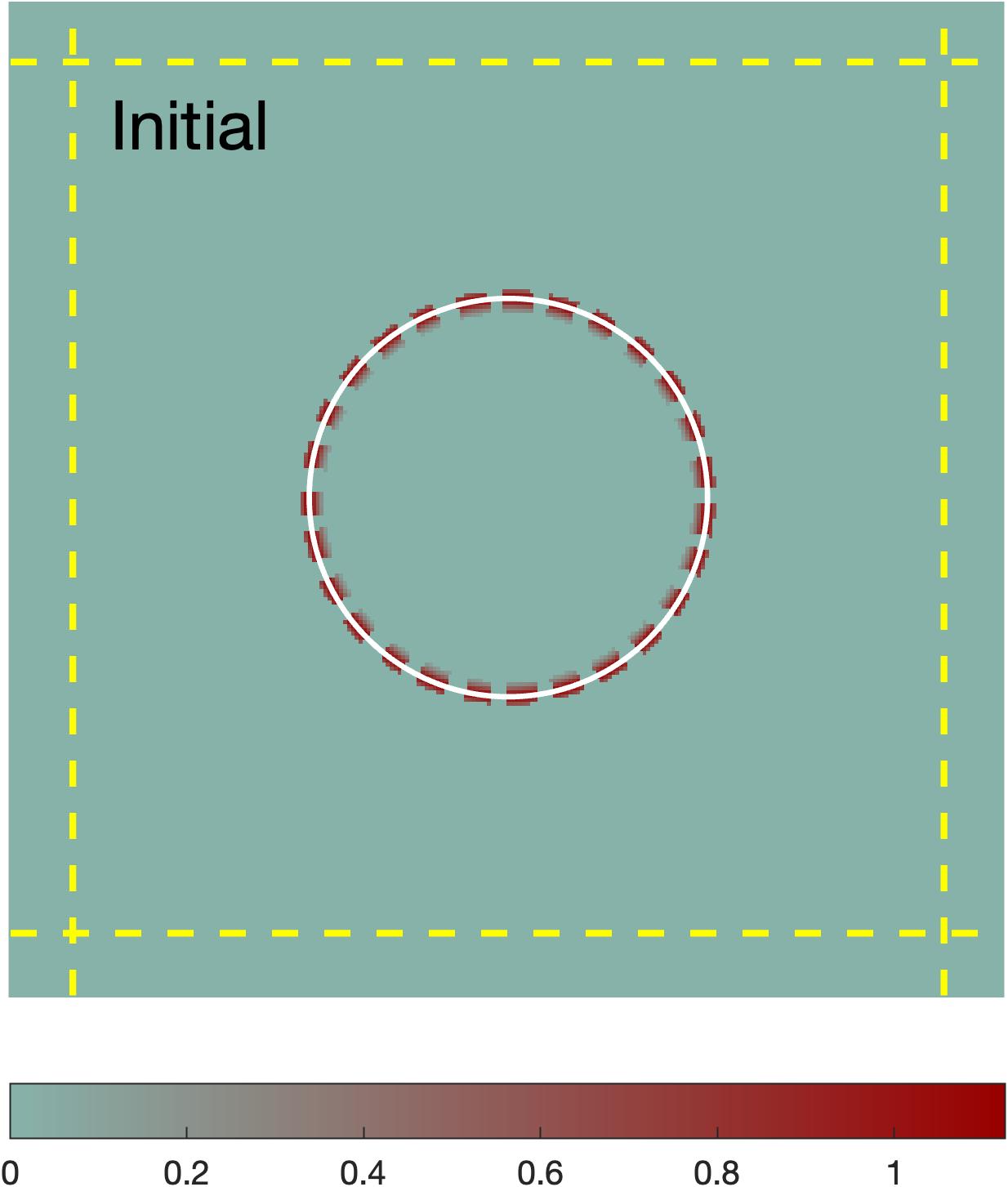}
    \includegraphics[width=0.2\textwidth]{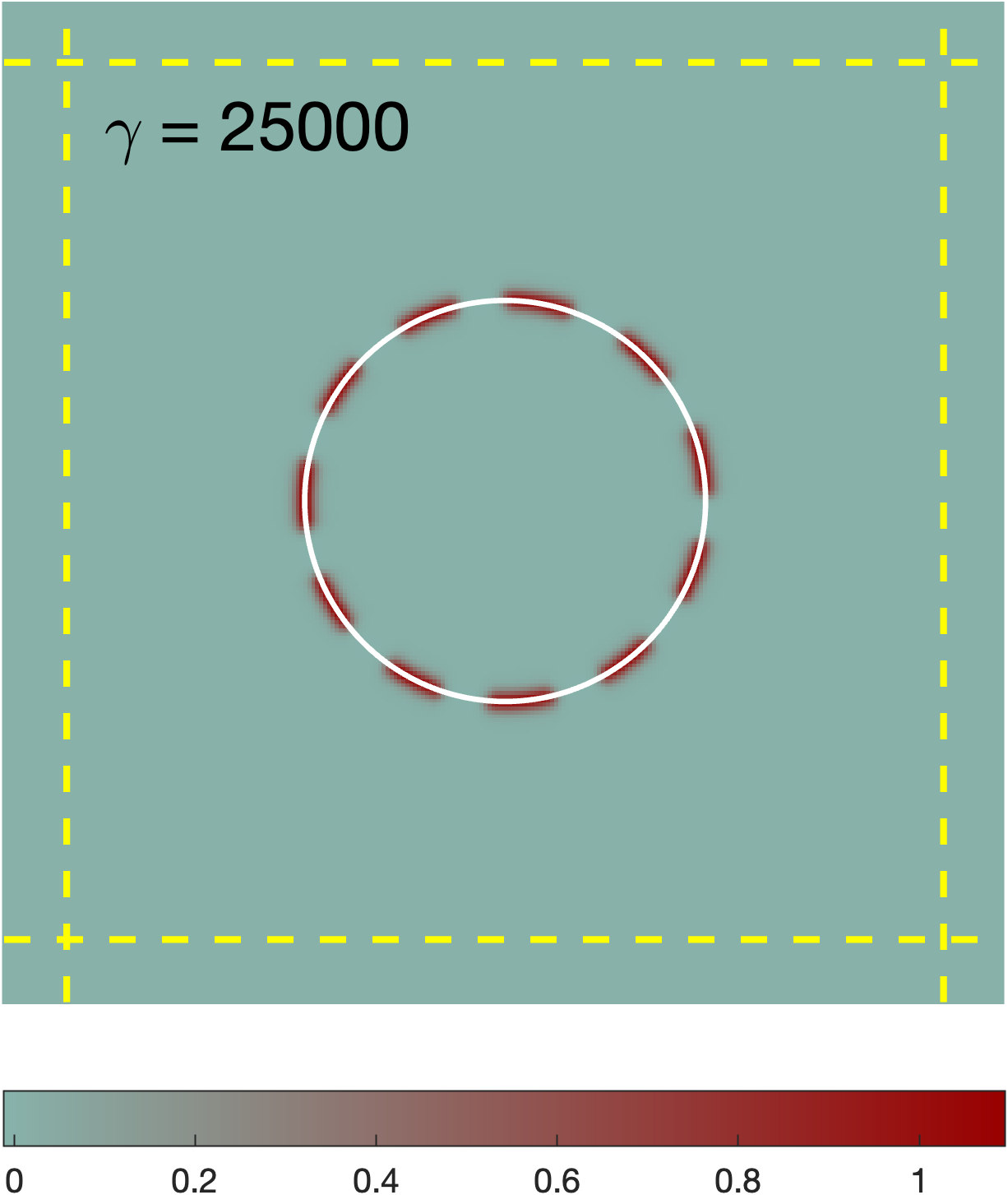}
    \end{center}
    \vspace{-0mm}
    \caption{Effect of the long-range interaction strength $\gamma$ on protein phase separation. Within each pair, the left panel shows a random initial protein distribution on the membrane and the right panel shows the corresponding equilibrium state. The top-left, top-right, bottom-left, and bottom-right pairs correspond to $\gamma=14000$, $17000$, $20000$, and $25000$, producing eight, nine, ten, and eleven protein-rich subdomains, respectively. In all four experiments, $\omega=0.5$, $(\kappa_-,\kappa_+)=(4,0.2)$, and $\lambda_u=20$.}   
    \label{fig:gamma-effect}
\end{figure}

\subsection{Effects of $\lambda_u$ and the Bending Rigidities $\kappa_+$ and $\kappa_-$}\label{subsec:lambdakappa-effect}

We now use the Strang-ETDRK2 scheme \eqref{eqn:ETDRK2_Strang} to explore the effects of the bending rigidities $(\kappa_-,\kappa_+)$ and the membrane-associated OK energy scale $\lambda_u$ on membrane deformation. Recall that the parameters $\kappa_-$ and $\kappa_+$ denote the bending rigidities of the protein-poor and protein-rich subdomains, respectively, and the parameter $\lambda_u$ is the trade-off between the membrane energy and the membrane-associated OK energy.

The top row in Figure~\ref{fig:lambdakappa-effect} shows the effect of the variable bending rigidity. The initial phase field vesicle is given by \eqref{eqn:ini_phi}, and the initial protein distribution contains three prescribed protein-rich subdomains on the membrane. We fix $\omega=0.5$, $\lambda_u=10$, and $\gamma=5000$, and vary $(\kappa_-,\kappa_+)=(1,0.2)$, $(10,0.2)$, and $(0.2,5)$. In the first two cases where $\kappa_+<\kappa_-$, the protein-rich subdomains are softer than the protein-poor subdomains and prefer to stay in high-curvature regions. The stiffer protein-poor subdomains tend to locate in low-curvature regions. As $\kappa_-$ increases from $1$ to $10$, the membrane develops a more pronounced triangular shape. For $(\kappa_-,\kappa_+)=(0.2,5)$, the protein-rich subdomains are stiffer and stay in low-curvature regions, whereas the softer protein-poor subdomains locate in high-curvature regions, leading to a different membrane morphology.

For the simulations shown in the bottom row of Figure~\ref{fig:lambdakappa-effect}, the initial phase field vesicle is again given by \eqref{eqn:ini_phi}, with one protein-rich subdomain on the upper part of the membrane. We fix $\omega=0.4$, $(\kappa_-,\kappa_+)=(3,0.5)$ and $\gamma=100$, and vary $\lambda_u=10$, $100$, and $1000$. In all three equilibrium states, the protein-rich subdomain forms the upper curved regions, while the protein-poor subdomain stays in the lower part of the membrane. Varying $\lambda_u$ produces small changes in the vesicle size and the curvatures of the protein-rich and protein-poor subdomains. As explained in Remark~\ref{rem:line-tension-OK}, the short-range term $\lambda_u E_{\mathrm{loc}}$ in $E_{\rm OK}[\phi,u]$ \eqref{eq:MOK-energy} approximates the line-tension energy in \cite{wang2008}. The effect associated with line tension is less apparent in these two-dimensional examples, where the two phases only interact at two boundary points. In three dimensions, the boundary between the protein-rich and protein-poor subdomains is a curve on the membrane surface, and the short-range term $\lambda_u E_{\mathrm{loc}}$ favors a reduction in its length, which can lead to more pronounced membrane deformation through competition with elastic bending, surface tension, and the volume constraint \cite{wang2008}. The three-dimensional examples in Figure~\ref{fig:3d-examples} in the next subsection further show membrane deformation generated by the coupled system.

\begin{figure}[t]
    \begin{center}
    \includegraphics[width=0.22\textwidth]{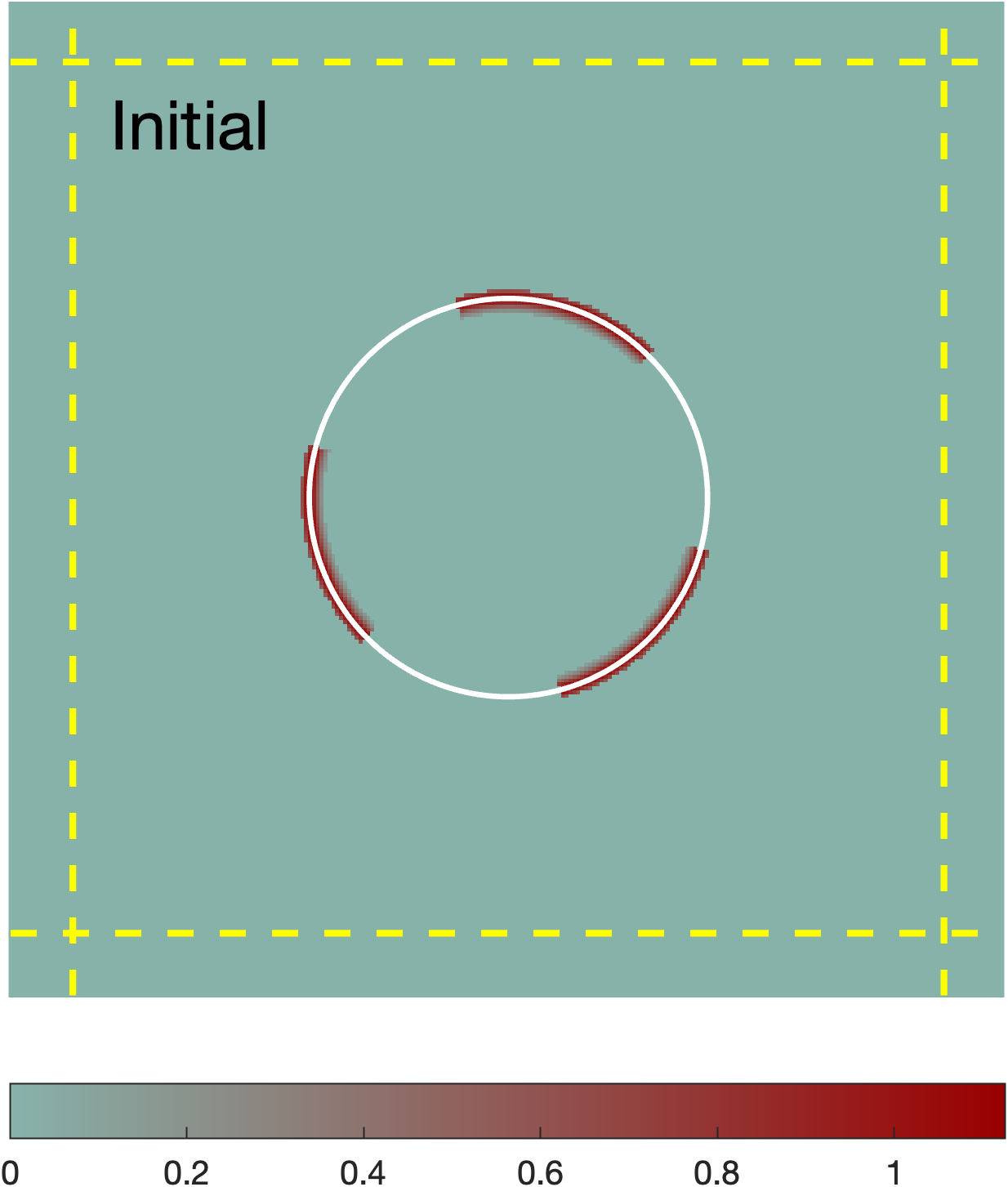}
    \includegraphics[width=0.22\textwidth]{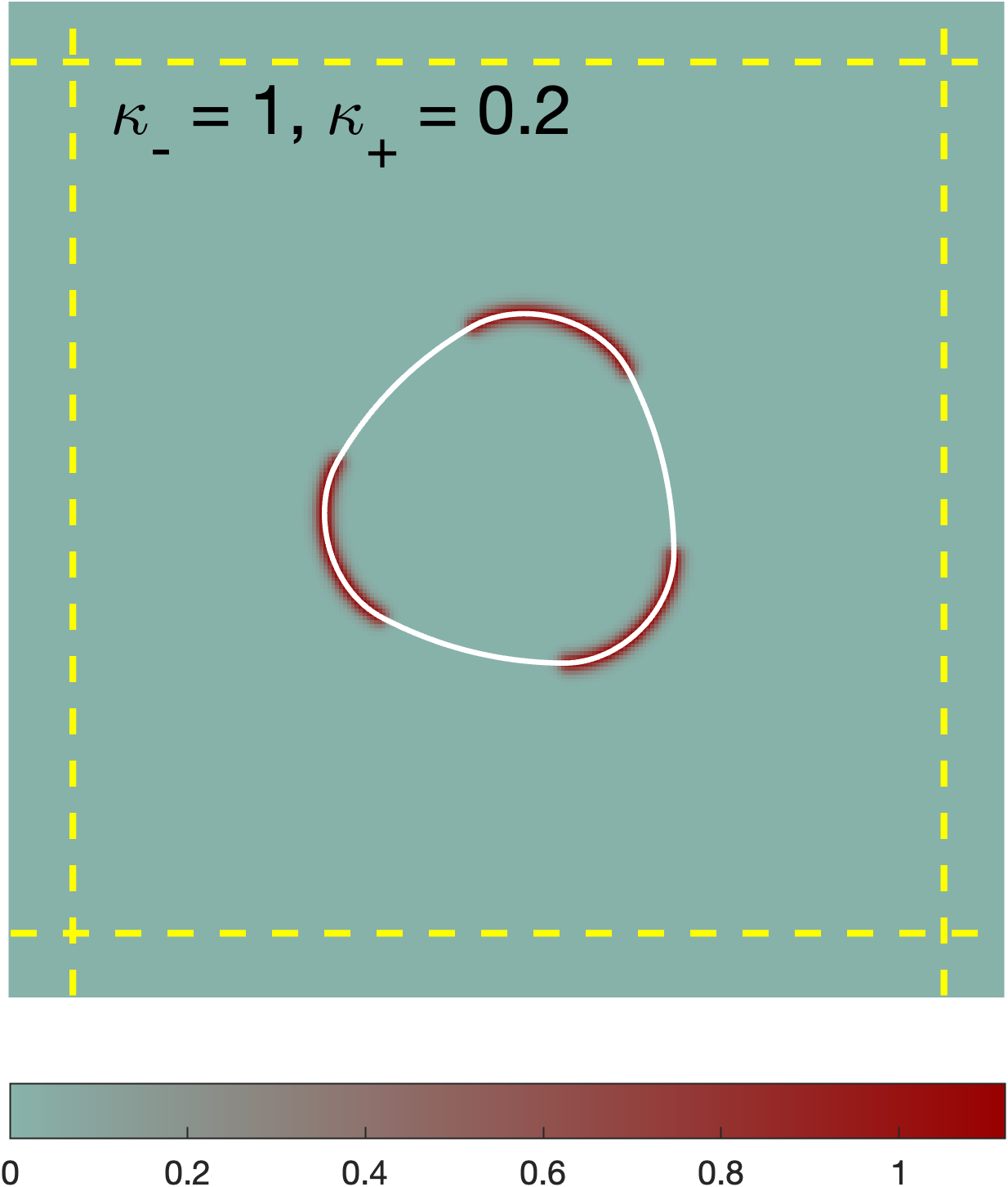}
    \includegraphics[width=0.22\textwidth]{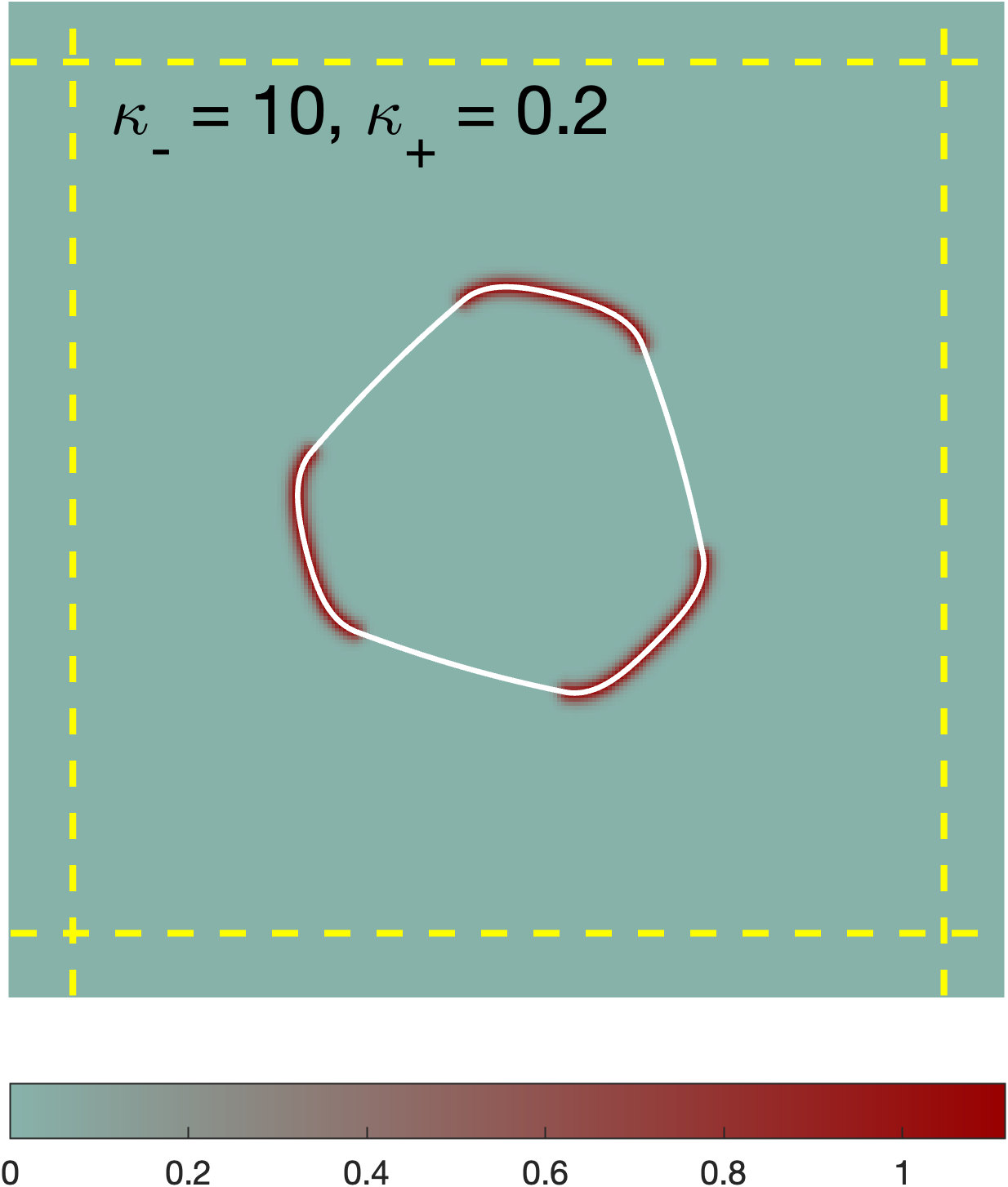}
    \includegraphics[width=0.22\textwidth]{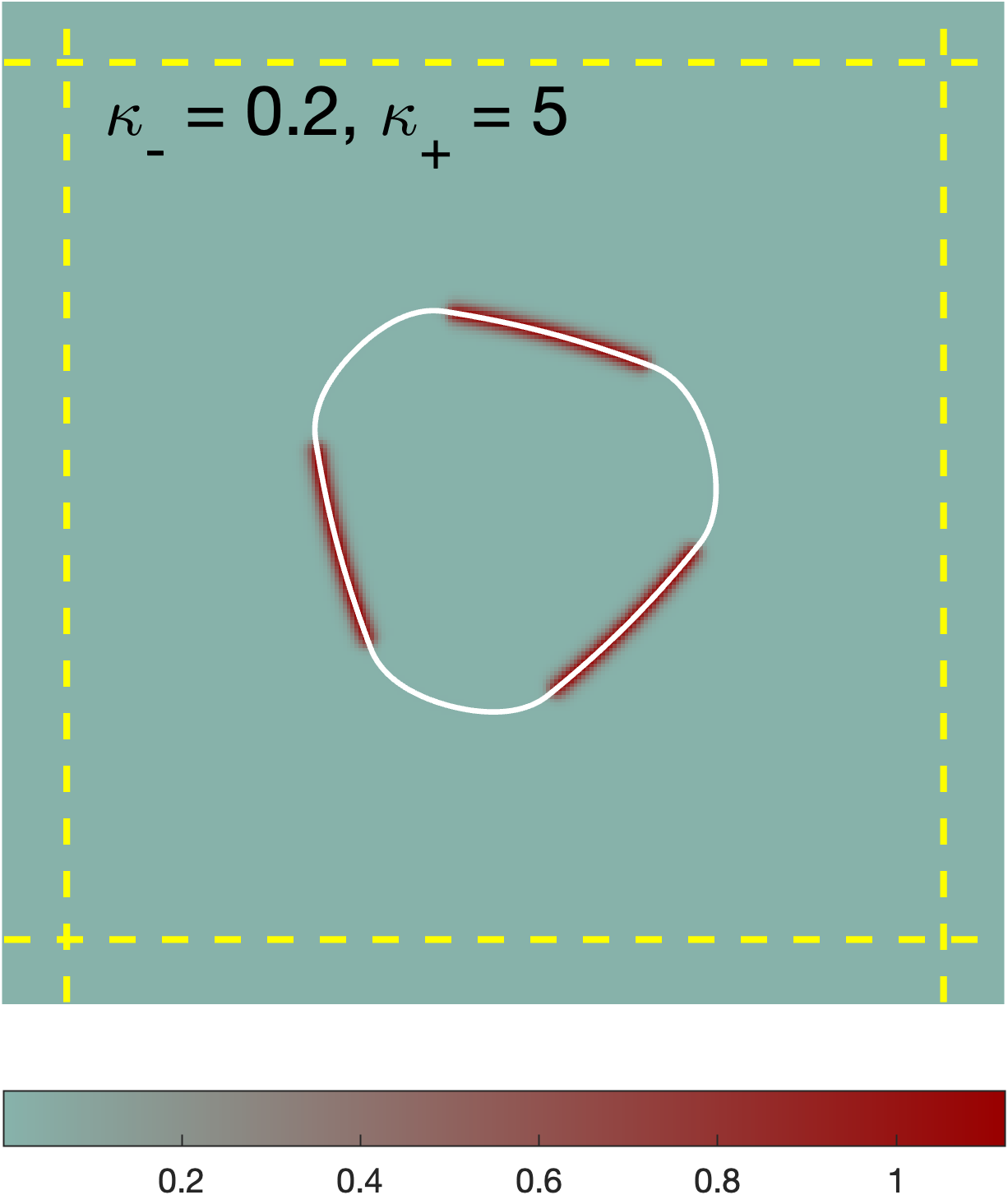}\\
    \vspace{3mm}
    \includegraphics[width=0.22\textwidth]{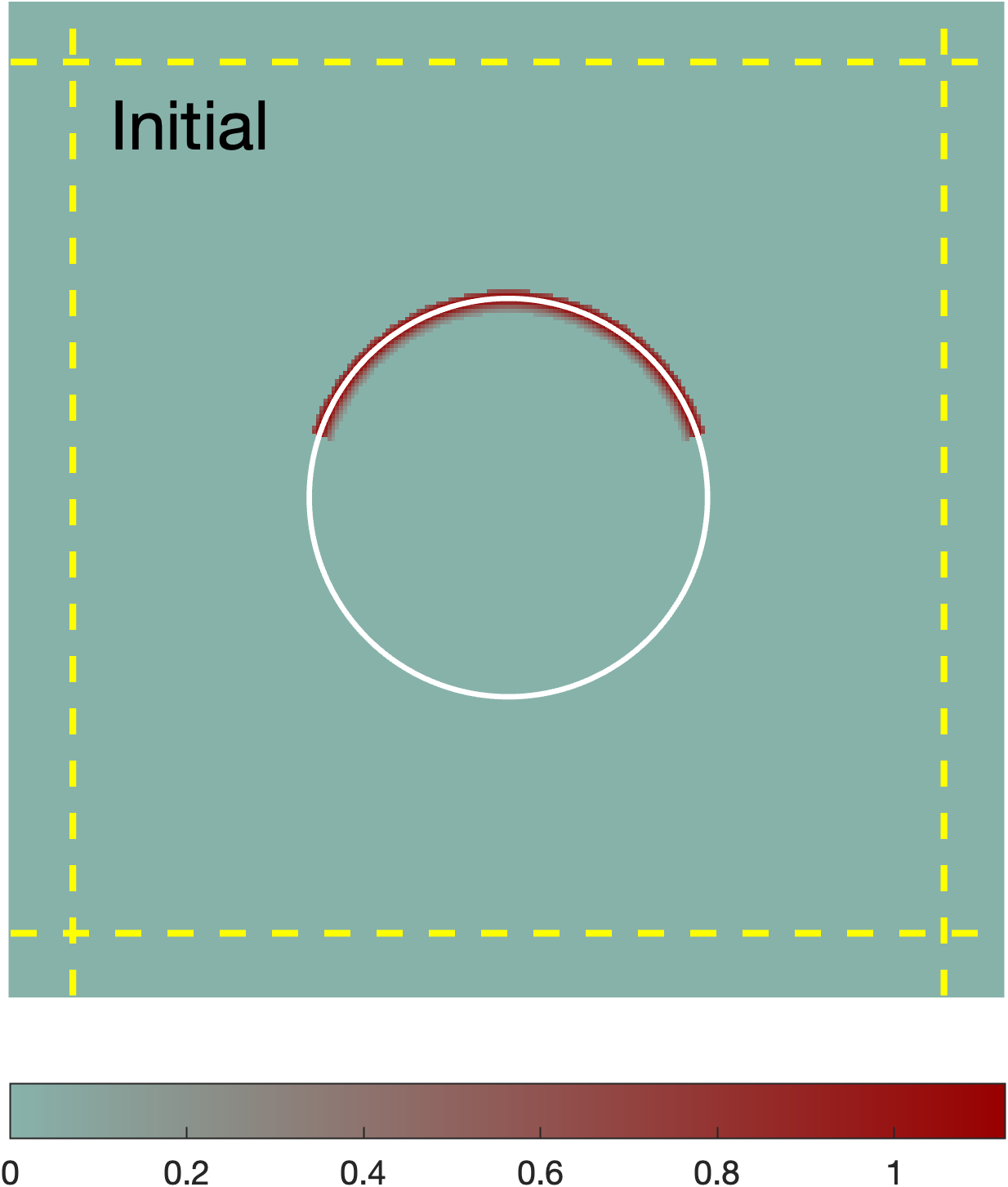}
    \includegraphics[width=0.22\textwidth]{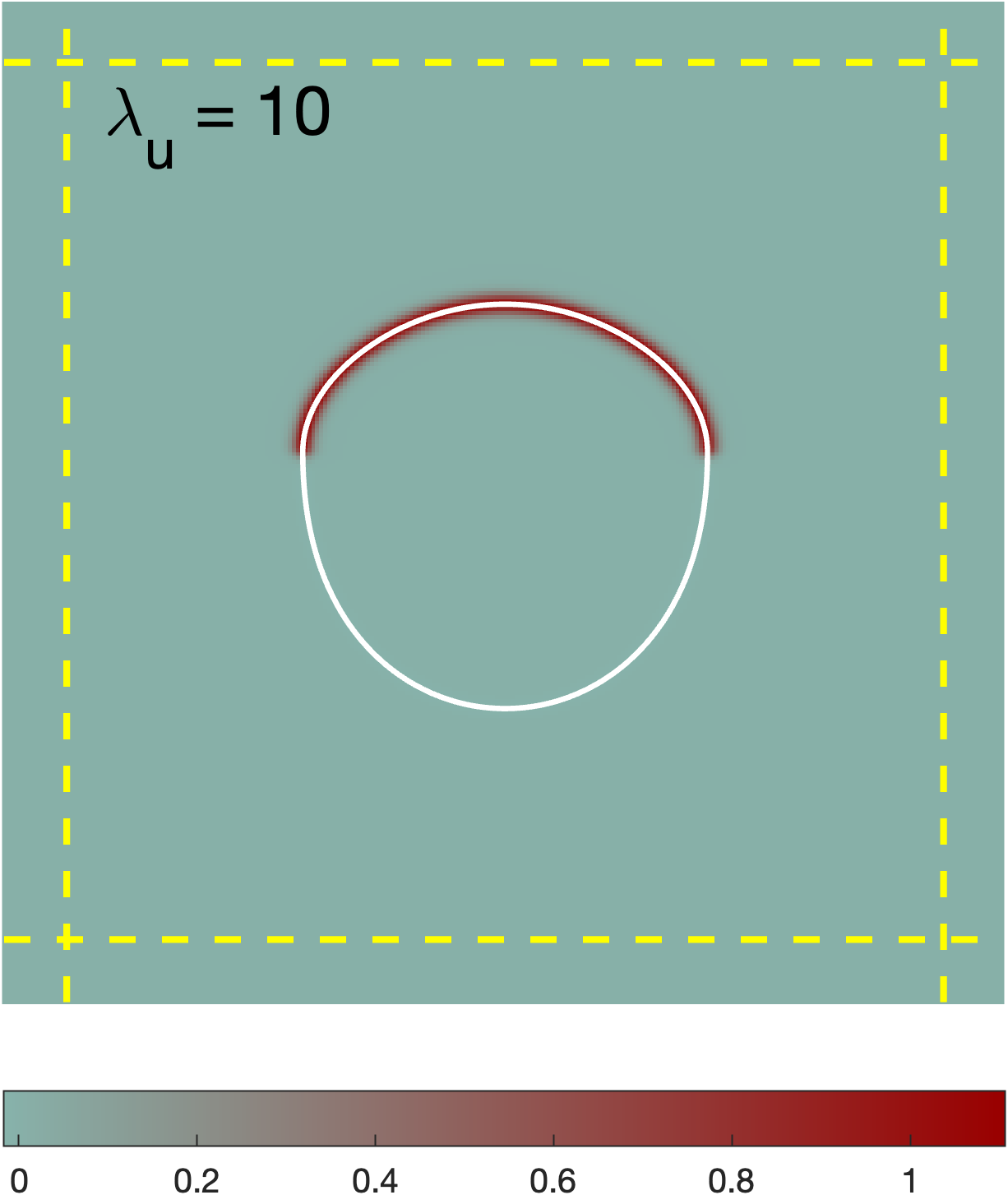}
    \includegraphics[width=0.22\textwidth]{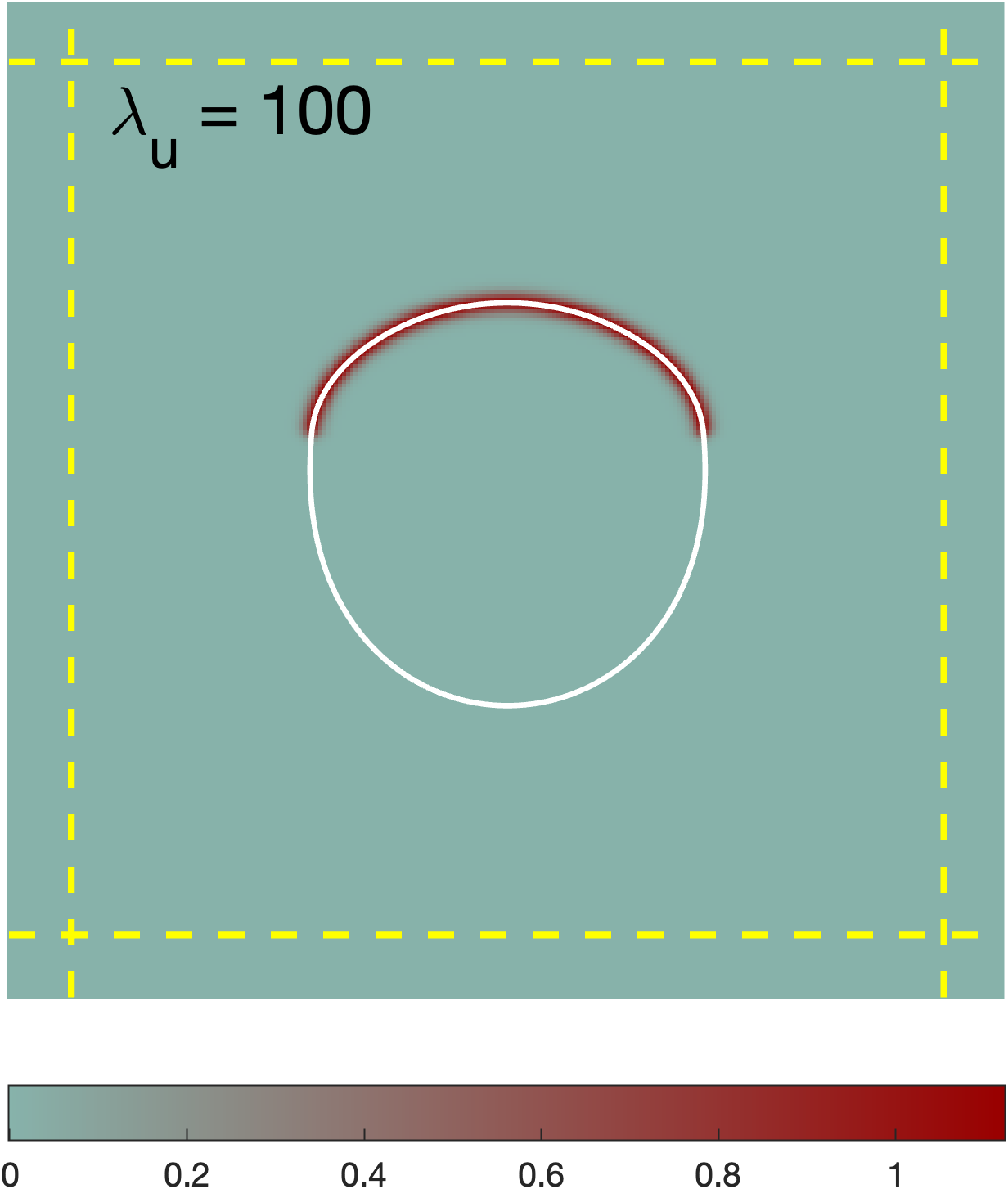}
    \includegraphics[width=0.22\textwidth]{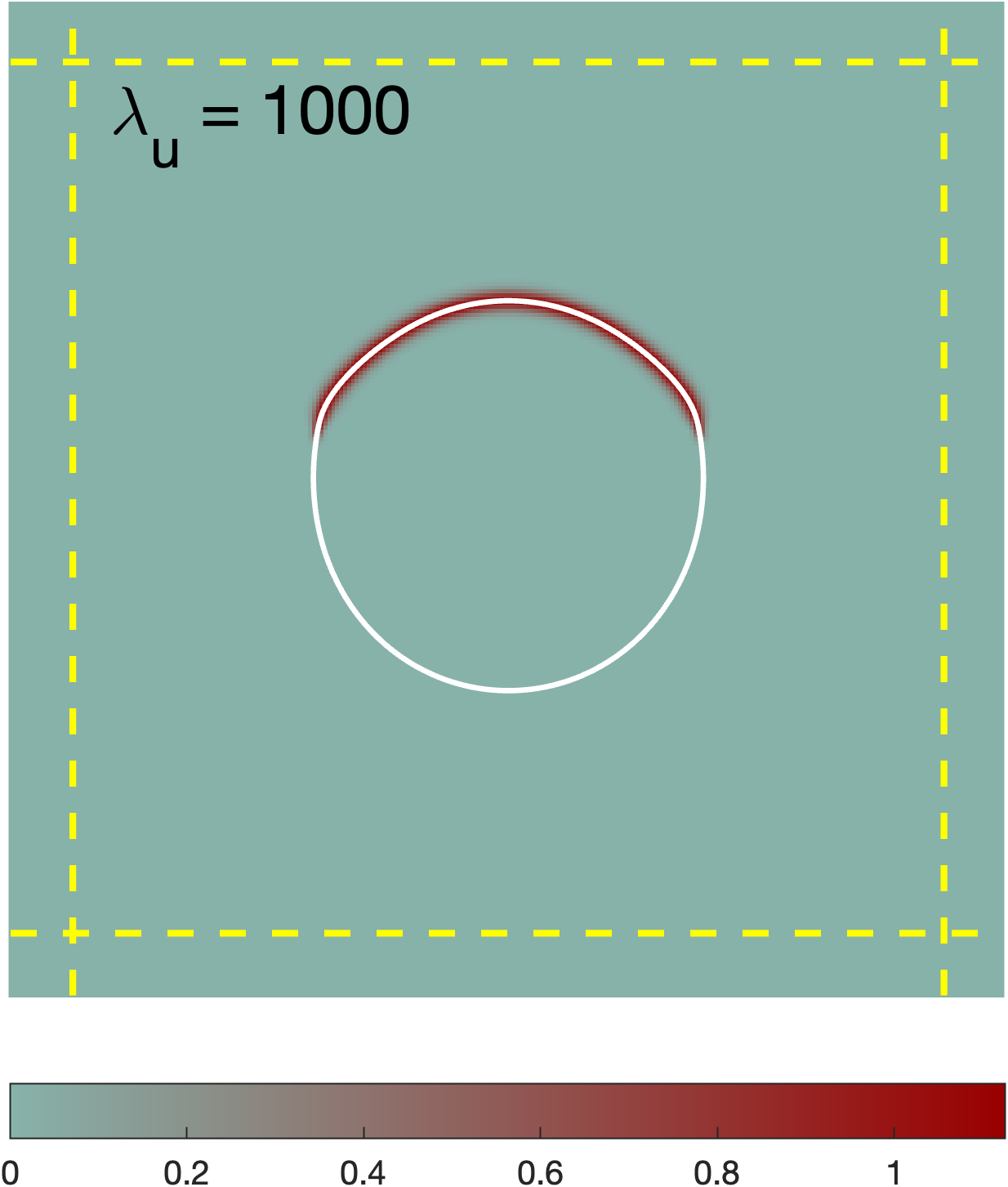} 
    \end{center}
    \vspace{-0mm}
    \caption{Effects of the protein-dependent bending rigidities $\kappa_{+}$, $\kappa_{-}$ and the scale $\lambda_u$ on membrane deformation. Top row from left to right: the prescribed three-bump initial state and the equilibrium states for $(\kappa_-,\kappa_+)=(1,0.2)$, $(10,0.2)$, and $(0.2,5)$, with $\omega=0.5$, $\lambda_u=10$ and $\gamma=5000$. Bottom row from left to right: the prescribed phase-separated initial state and the equilibrium states for $\lambda_u=10$, $100$, and $1000$, with $\omega=0.4$, $(\kappa_-,\kappa_+)=(3,0.5)$ and $\gamma=100$. 
    }   
    \label{fig:lambdakappa-effect}
\end{figure}

\subsection{Representative Three-dimensional Patterns}
\label{subsec:three-dimensional-configurations}

Finally, we apply the proposed model in three dimension to examine the coupling between protein segregation and membrane morphology. The initial phase field vesicle is given by
\begin{align}
        \bPhi^0 = 0.5+0.5\tanh{\left(\frac{3(r_0-r)}{\ve_{\phi}}\right)}, \quad r_0 = 0.4, \nonumber
\end{align}
where $r$ is the distance from each grid point to the center of the 3D domain. The initial protein distributions are chosen close to phase-separated states. We then evolve the fully coupled system until an equilibrium state is reached in each case. The parameters are specified separately for each experiment below. The membrane is represented by the isosurface $\phi=1/2$, colored by the protein density $u$.

The top row of Figure~\ref{fig:3d-examples} presents three completely phase-separated patterns with prescribed protein fractions $\omega=0.1$, $0.3$, and $0.85$, from left to right. The corresponding parameter sets $(\lambda_u,\gamma,\kappa_-,\kappa_+)$ are $(30,100,2,1)$, $(10,100,5,0.5)$, and $(30,100,3,1)$, respectively. Although each case contains one protein-rich subdomain, the different protein fractions and bending rigidities lead to different membrane morphologies. The bottom row of Figure~\ref{fig:3d-examples} shows three examples with different numbers and locations of protein-rich subdomains. In the bottom-left panel, an equatorial protein-rich band produces a dumbbell-shape membrane with $\lambda_u=100$, $\omega=0.2$, $\gamma=300$, and $(\kappa_-,\kappa_+)=(0.2,6)$. In the bottom-middle panel, two protein-rich subdomains located near the two poles of the vesicle produce a pear-shape membrane with $\lambda_u=30$, $\omega=0.4$, $\gamma=500$, and $(\kappa_-,\kappa_+)=(2,1)$. In the bottom-right panel, twelve protein-rich subdomains generate a multi-bump membrane with $\lambda_u=40$, $\omega=0.5$, $\gamma=7000$, and $(\kappa_-,\kappa_+)=(5,0.5)$. These examples show that the proposed model can generate completely phase-separated, dumbbell-shape, pear-shape, and multi-bump membrane morphologies through the coupling between protein segregation and membrane deformation.

\begin{figure}[t]
  \centering

  \ThreeDPanel{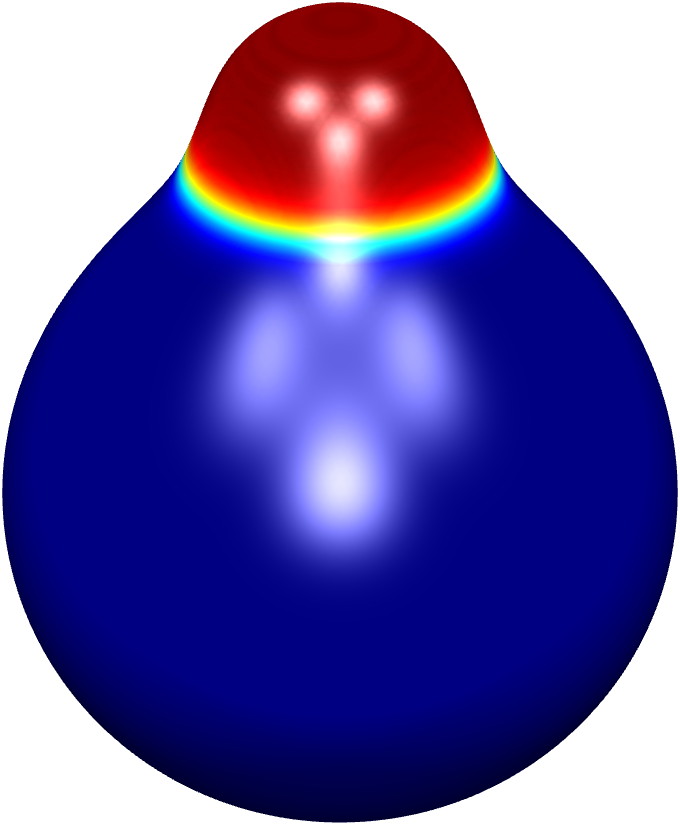}\hfill
  \ThreeDPanel{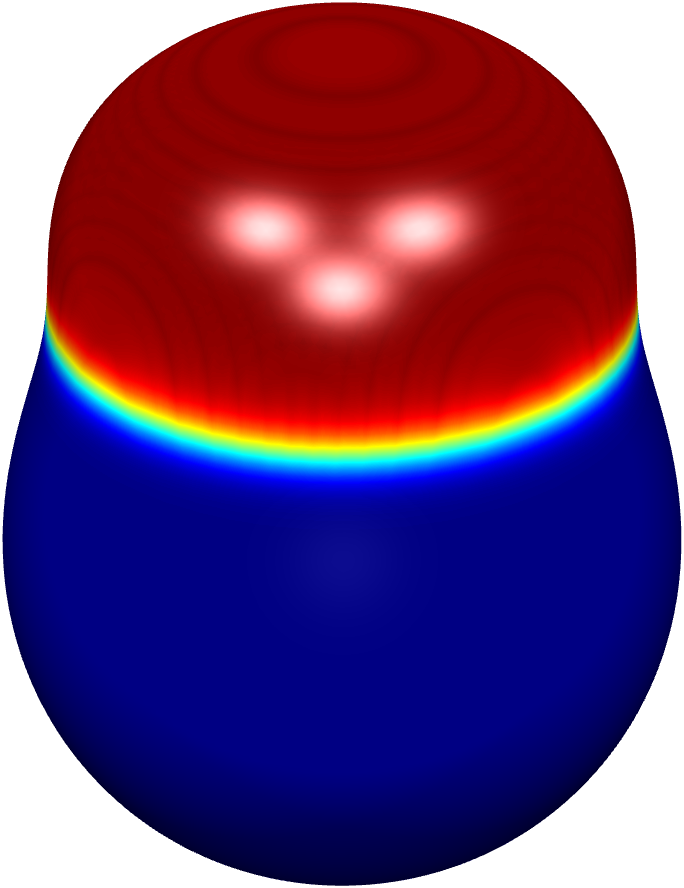}\hfill
  \ThreeDPanel{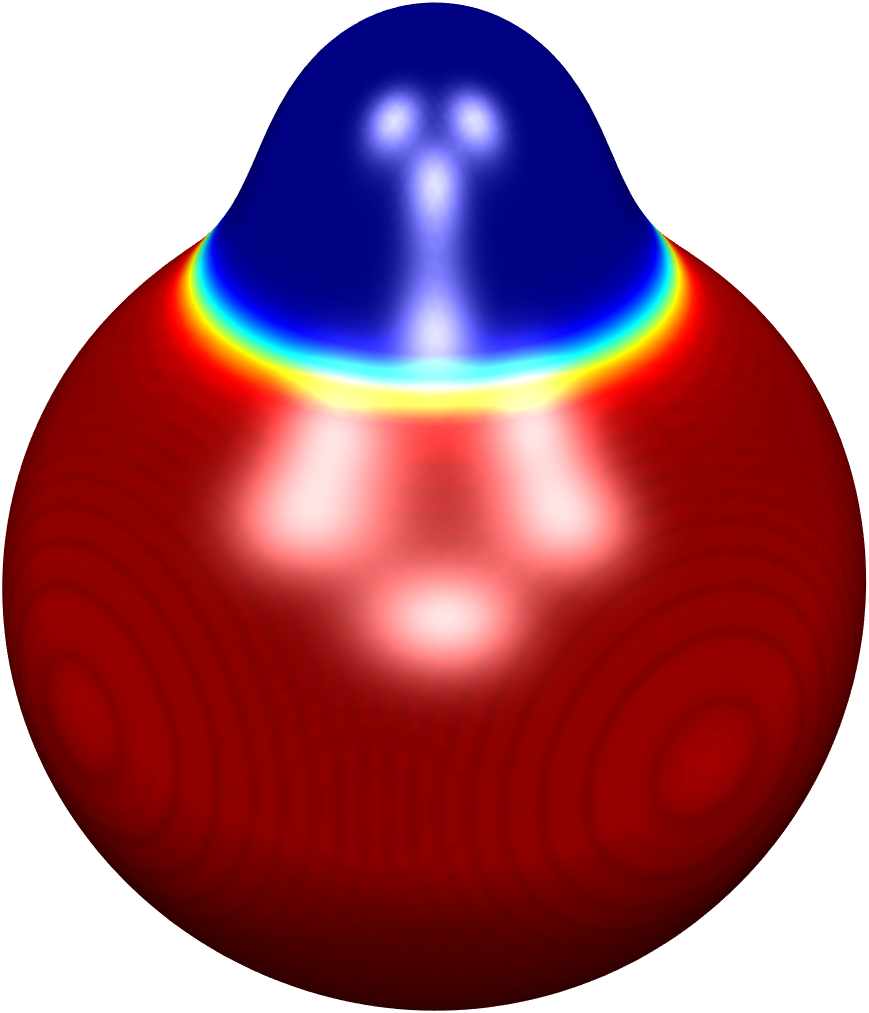}

  \vspace{5mm}

  \ThreeDPanel{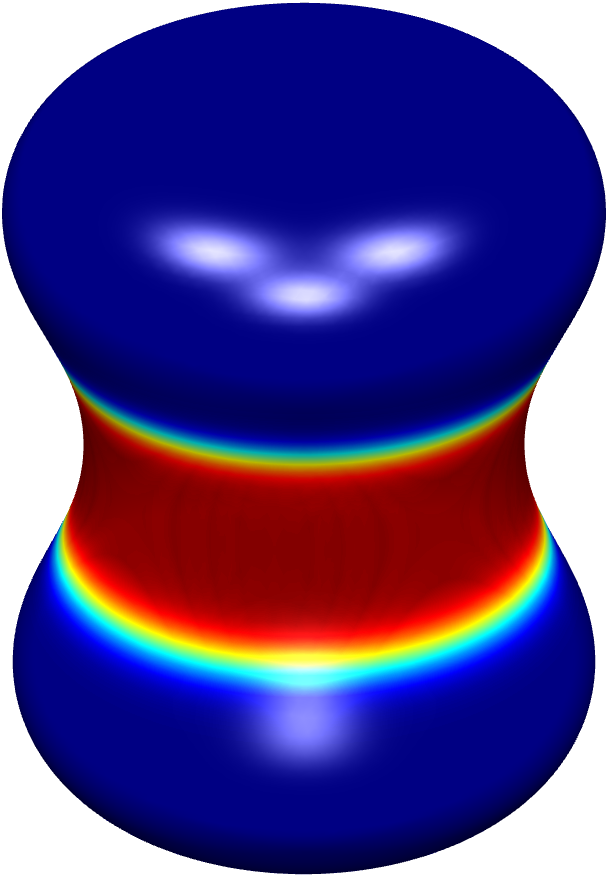}\hfill
  \ThreeDPanel{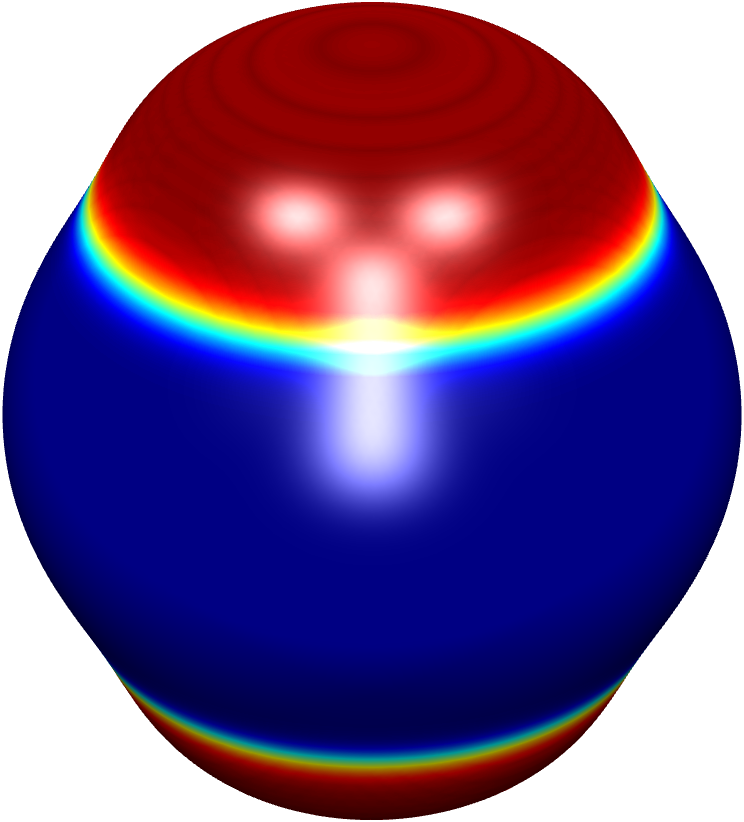}\hfill
  \ThreeDPanel{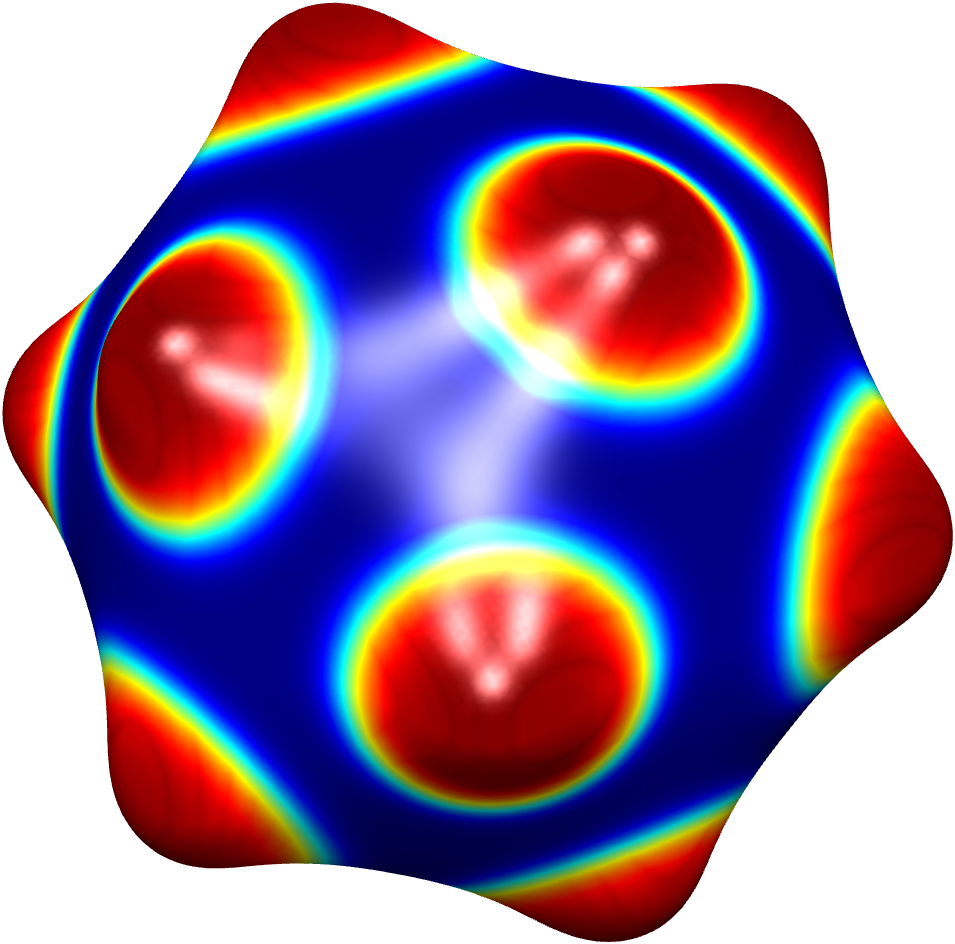}

  \caption{ Representative three-dimensional equilibrium patterns. Top row from left to right: completely phase-separated patterns with prescribed protein fractions $\omega=0.1$, $0.3$, and $0.85$, respectively. Bottom row from left to right: a dumbbell-shape pattern, a pear-shape pattern, and a multi-bump pattern. The membrane is the isosurface $\phi=1/2$; red and blue indicate protein-rich and protein-poor subdomains, respectively.}
  \label{fig:3d-examples}
\end{figure}

\section{Conclusion}\label{sec:conclusion}

In this work, we developed a thermodynamically consistent diffuse-interface model for multicomponent membranes that couples membrane deformation and membrane-associated protein segregation through a coupled free energy. This energy combines protein-dependent diffuse bending energy, surface tension, enclosed-volume penalty, and membrane-associated OK energy. In contrast to the mechanochemical model developed in our previous work \cite{luo2026ok,luoETDmembrane}, the coupling in the current model is derived from the coupled free energy in which protein composition determines the bending rigidity, while membrane geometry influences the domain for the protein segregation. By applying the Onsager variational principle, we obtained a coupled $L^2$ gradient flow system and its continuous energy dissipation law. We then introduced a stabilized linear-nonlinear splitting, constructed alternating ETD1 and Strang-ETDRK2 schemes for the coupled system, and proved the discrete energy dissipation law for both schemes. The numerical experiments validated the discrete energy dissipation and showed how the key model parameters affect the membrane morphology. The three-dimensional simulations produced completely phase-separated, dumbbell-shape, pear-shape, and multi-bump membranes morphologies through the coupled system.

The work can be extended in several directions. Firstly, the current analysis of discrete energy dissipation law depends on the regularity assumptions of the numerical solutions. In the future, we will investigate whether these regularity assumptions can be obtained from sufficiently smooth initial data. Moreover, we will conduct rigorous convergence analysis for the proposed alternating ETD schemes and further develop higher-order structure-preserving numerical schemes. Secondly, we will study the well-posedness and sharp-interface limit of the coupled system. Finally, the proposed model can be extended to incorporate the hydrodynamics or adhesion effects when the membranes interact with the surrounding environments (substrate, extracellular matrix). Such extensions will provide a more flexible framework to explore the morphologies of the multicomponent vesicle membrane.

\section*{Acknowledgments} 

{The first and second authors were partially supported by NSFC/RGC Joint Research Scheme grant N\_PolyU5145/24 and Hong Kong Research Grants Council GRF grant 15305624.}

\section*{Data Availability} 

The datasets generated in this study are available from the corresponding author upon request.

\section*{Conflict of Interest} 

The authors declare that they have no conflict of interest.

\bibliographystyle{elsarticle-num}
\bibliography{citation}

\end{document}